\documentclass{article}
\usepackage{graphicx} %
\usepackage[margin=1in]{geometry}
 \usepackage{times}
 \usepackage{amsmath,amsfonts,amsthm,amssymb}
\newif\ifsiopt
\sioptfalse

\title{Pushing the Complexity Boundaries of Fixed-Point Equations: Adaptation to Contraction and Controlled Expansion}
\author{Jelena Diakonikolas\\
University of Wisconsin-Madison\\
\texttt{jelena@cs.wisc.edu}}
\date{}

 \usepackage{bm}

\usepackage{dsfont}
\usepackage{mathtools}
\usepackage[sort]{cite}
\usepackage[colorlinks,urlcolor=blue,citecolor=blue,linkcolor=blue]{hyperref}
 \usepackage{color}

\usepackage{cleveref}

\usepackage{thm-restate}

\usepackage{algorithm,algpseudocode}

\usepackage{subcaption}

\newtheorem{theorem}{Theorem}
\newtheorem{lemma}{Lemma}
\newtheorem{corollary}{Corollary}

\newtheorem{proposition}{Proposition}
\newtheorem{assumption}{Assumption}
\newtheorem{fact}{Fact}

\newtheorem{definition}{Definition}

\def\1{\bm{1}}

\newcommand{\norm}[1]{\left\lVert#1\right\rVert}

\def\ebar{{\overline{\epsilon}}}

\def\vzero{{\bm{0}}}
\def\vone{{\bm{1}}}

\def\vg{{\bm{g}}}

\def\vu{{\bm{u}}}
\def\vv{{\bm{v}}}

\def\vx{{\bm{x}}}

\def\vy{{\bm{y}}}

\def\vxh{\hat{\bm{x}}}

\def\mF{{\bm{F}}}
\def\mG{{\bm{G}}}

\def\mId{{\bm{\mathrm{Id}}}}
\def\mJ{{\bm{J}}}

\def\mR{{\bm{R}}}
\def\mS{{\bm{S}}}
\def\mT{{\bm{T}}}
\def\mTt{\widetilde{\bm{T}}}

\def\mFh{\widehat{\bm{F}}}

\def\cC{{\mathcal{C}}}

\def\cE{{\mathcal{E}}}

\def\cG{{\mathcal{G}}}

\def\cL{{\mathcal{L}}}

\def\cN{{\mathcal{N}}}
\def\cO{{\mathcal{O}}}

\def\cS{{\mathcal{S}}}

\def\cX{{\mathcal{X}}}

\def\sR{{\mathbb{R}}}

\newcommand{\dist}{\mathrm{dist}}

\DeclareMathOperator*{\argmin}{arg\,min}

\newcommand{\innp}[1]{\langle #1 \rangle}
\usepackage[colorinlistoftodos]{todonotes}

\newcommand{\dom}{\mathrm{dom}}
\newcommand{\prox}{\mathrm{prox}}
\newcommand{\proj}{\mathrm{Proj}}

\begin{document}

\maketitle

\begin{abstract}
Fixed-point equations with Lipschitz operators are central to areas such as optimization, game theory, economics, and dynamical systems. When the operator is contractive or nonexpansive (i.e., when its Lipschitz constant is $\gamma \leq 1$), decades of work have established efficient algorithms with tight oracle complexity guarantees. In sharp contrast, even mildly expansive operators ($\gamma > 1$) render fixed-point computation intractable in the worst case, with exponential oracle lower bounds. This dichotomy leaves open a fundamental question: are there intermediate regimes where efficient approximation is still possible? 
Our main contribution is in establishing such regimes. First, we show that a seemingly misguided idea---running Halpern iteration with a fixed step size---provably finds $\epsilon$-approximate fixed points at near-optimal oracle complexity for contractive and nonexpansive operators, and it succeeds for mildly expansive operators up to the hardness frontier. Building on this insight, we design the Gradual Halpern Algorithm (GHAL) and its parameter-free variant AdaGHAL, which adapt automatically to the operator’s Lipschitz constant and recover optimal complexity across the contractive and nonexpansive regimes, while extending guarantees into the mildly expansive case. Finally, we introduce the class of gradually expansive operators, permitting constant expansion (up to $\approx 1.4$), and prove that AdaGHAL finds $\epsilon$-approximate fixed points in 
$O(1/\epsilon)$ iterations---demonstrating for the first time that efficient fixed-point computation is possible under controlled expansion significantly beyond the nonexpansive regime. 
Our results apply in general normed vector spaces---including infinite-dimensional Banach spaces, and generalize to non-positively curved geodesic metric spaces. For contractive, nonexpansive, and mildly expansive operators, the guarantees we obtain are essentially tight against known lower bounds; for gradually expansive operators, our results identify a new tractable frontier in a regime where no complexity theory previously existed.
\end{abstract}

\section{Introduction}

Fixed-point equations in a normed vector space $(\cE, \norm{\cdot})$ ask for identifying an element of the space $\vx \in \cE$ with the property that $\mT(\vx) = \vx,$ where $\mT:\cE\to \cE$ is a given operator (or a mapping). If there exists $\gamma \geq 0$ such that for any $\vx, \vy \in \cE,$ we have $\norm{\mT(\vx) - \mT(\vy)} \leq \gamma \norm{\vx - \vy},$ then $\mT$ is said to be $\gamma$-Lipschitz. In this work, we are concerned with approximately solving fixed point equations: given an error parameter $\epsilon >0,$ the goal is to identify $\vx \in \cE$ such that $\norm{\mT(\vx) - \vx}\leq \epsilon.$ We refer to such $\vx$ as ``$\epsilon$-approximate fixed points.''

While simple to state, fixed-point equations encompass a broad range of problems in engineering, game theory, economics, mathematical optimization, and dynamical systems, among others (see, e.g., \cite{facchinei2007finite,yannakakis2009equilibria,todd2013computation} for many illustrative examples across these areas). Given their centrality in these scientific disciplines, the topic of how to efficiently solve fixed-point equations has been intensely investigated over the past several decades. 

Classical results in this domain have provided iterative algorithms based on oracle access to $\mT$ (i.e., algorithms that only rely on evaluations of $\mT$) with guaranteed convergence whenever the operator is contractive (i.e., Lipschitz-continuous with $\gamma < 1$), nonexpansive (i.e., Lipschitz-continuous with $\gamma = 1$), or satisfies other slightly more general but closely related properties such as quasi-nonexpansiveness or pseudo-contractivity; see  \cite{browder1967convergence,ishikawa1974fixed,mann1953mean,krasnosel1955two,halpern1967fixed,edelstein1966remark,opial1967weak}. Such results have been established in broad generality, even applying to possibly infinite-dimensional normed vector spaces. At present, the complexity of solving fixed-point equations for contractive and nonexpansive operators (i.e., with Lipschitz constant $\gamma \leq 1$) in high-dimensional settings\footnote{I.e., in settings where the oracle complexity is independent of the dimension and polynomial dependence on $1/\epsilon$ or $1/(1-\gamma)$ (when $\gamma < 1$) can be tolerated. By contrast, there remain open questions on polynomial-time solvability---with oracle complexity polynomial in $d$ and $\log(1/\epsilon)$---for general $\ell_p$ settings with $p \neq 2$; see \cite{chen2024computing} for a relevant discussion and recent progress in this domain.} is well-understood \cite{chou1987optimality,nemirovsky1991optimality,baillion1996rate,sikorski2009computational,Sabach:2017,lieder2019convergence,kim2019accelerated,diakonikolas2021potential,park2022exact,bravo2019rates,bravo2024stochastic}, with some of the recent results even characterizing the complexity of common algorithms down to constants \cite{bravo2021universal,contreras2023optimal}. For operators that are allowed even a small amount of expansion (i.e., with $\gamma$ slightly larger than one), existing results require a type of a (local) error bound that relates the fixed-point error $\|\mT(\vx) - \vx\|$ to some notion of distance between $\vx$ and a nearest fixed point to establish nonasymptotic (local or global) convergence of considered algorithms  \cite{natarajan1993condition,russell2018quantitative,lauster2021convergence,hermer2023rates}. 

In another direction, research in theoretical computer science investigated the complexity of solving fixed-point equations in a  finite dimensional vector space $\sR^d$ equipped with a norm $\norm{\cdot}$ (typically the $\ell_\infty$ or $\ell_2$ norm). A prototypical problem in this area concerns solutions to fixed points of Lipschitz-continuous operators $\mT$ that map the unit hypercube to itself; a problem commonly referred to as the Brouwer fixed point problem. The existence of fixed points in this case is guaranteed by the classical Brouwer's fixed point theorem \cite{brouwer1911abbildung}, the proof of which is non-constructive. The first constructive proof, based on Sperner's lemma \cite{kuhn1968simplicial} and earlier algorithms by \cite{lemke1964equilibrium,lemke1965bimatrix} for computation of Nash equilibria in two-person games, was obtained by Scarf \cite{scarf1967approximation}. For general mappings $\mT$ with Lipschitz constant $\gamma > 1$, the complexity of Scarf's algorithm (and the algorithms that soon followed) is exponential. This raised the questions about complexity of solving such fixed-point equations (on the unit hypercube, with Lipschitz constant $\gamma > 1$, to error $\epsilon > 0$). In a seminal work, \cite{hirsch1989exponential} proved that any algorithm for this problem that is based on oracle queries must incur exponential complexity in the dimension in the worst case, and this happens as soon as the Lipschitz constant exceeds one by a constant factor in $\epsilon$. The gap in the precise exponential growth between the upper and lower bounds was settled in \cite{chen2008matching}. More generally, the discussed Brouwer's fixed-point problem is considered computationally hard (in terms of reduction-based complexity) even under white-box access to the operator $\mT$, as it is well-known to be PPAD-complete \cite{papadimitriou1994complexity}\footnote{In fact, the  Brouwer's fixed-point problem served as the most important motivating example for the introduction of the complexity class PPAD, as noted in the same work \cite{papadimitriou1994complexity}.}. A very recent work \cite{attias2025fixed} has similarly established exponential  complexity lower bounds  for oracle-based algorithms concerning $\gamma$-Lipschitz operators with $\gamma > 1$ mapping the unit $\ell_2$ ball to itself. Their lower bound applies even under the ``smoothed analysis'' framework of \cite{spielman2004smoothed}, which effectively means that the problem remains hard even under smooth perturbations of problem instances. This is commonly understood as the ``worst-case'' instances being, in fact, common in the considered problem class. %

The above discussion of complexity thus raises the question: 
\begin{center}
\emph{What lies between the tractable world of nonexpansive operators and the intractable world of expansive ones?}
\end{center}
Put differently: under what conditions on a Lipschitz operator 
$\mT: \cE\to\cE$ can we guarantee efficient convergence to approximate fixed points, beyond the classical contractive and nonexpansive regimes?  In addition to being a well-recognized research challenge \cite{baillion1996rate,daskalakis2022non}, this question is not merely of theoretical interest: such problems commonly arise in a variety of applications like machine learning (e.g., in training of Generative Adversarial Networks or in Multi-Agent Reinforcement Learning; see \cite{daskalakis2022non}), feasibility problems and structured nonconvex optimization \cite{russell2018quantitative}, or in X-ray free electron laser (X-FEL) imaging \cite{hermer2023rates}.  A desirable property of algorithms addressing such problems is that they can gracefully adapt to input instances, meaning that if an input operator happens to be  nonexpansive or contractive, the algorithm should exhibit similar convergence properties to optimal algorithms for the corresponding class of instances. In this work, we make progress on all these fundamental challenges.

\subsection{Contributions and Paper Overview}

This paper takes a step beyond the classical boundary between tractable nonexpansive operators and intractable expansive operators. We show that, contrary to conventional wisdom, even seemingly ill-suited variants of Halpern iteration can provide sharp guarantees once the focus shifts to identifying a point with the \emph{target approximation} of the fixed point error. Building on this insight, we develop adaptive algorithms that unify the contractive, nonexpansive, and mildly expansive\footnote{We informally say an operator is mildly expansive if its Lipschitz constant is slightly larger than one, by a quantity proportional to the error parameter $\epsilon > 0.$} regimes, and introduce a new tractable class of gradually expansive operators. Our results apply broadly in normed vector spaces and (non-positively curved) geodesic metric spaces, pushing algorithmic guarantees to the frontier of known lower bounds where such bounds exist. Our main contributions are summarized below. 

To state the results, we use the standard asymptotic notation: $f = \cO(g)$ indicates that there exists a constant $C > 0$ independent of $f, g,$ such that $f \leq C g.$ The notation $f = \Omega(g)$ is defined by $g = \cO(f).$ Finally, $f = \Theta(g)$ indicates that simultaneously $f = \cO(g)$ and $g = \cO(f);$ i.e., it indicates that $f$ and $g$ are of the same order.  

\paragraph{Fixed-step Halpern iteration as a surprising algorithmic tool} Our work starts from investigating the properties of a simple variant of classical Halpern iteration \cite{halpern1967fixed}, where in place of the decaying step sizes $\lambda_k$ (usually of order $1/k$), we use a constant value $\lambda_k = \lambda,$ $\forall k;$ i.e., the iteration is %
\begin{equation}\label{eq:main-iteration}\tag{FixHal}
    \vx_{k+1} = \lambda \vx_0 + (1-\lambda) \mT(\vx_k).
\end{equation}
Considering such an iteration seems counterintuitive at first, as it cannot be guaranteed to converge to zero fixed point error $\|\mT(\vx_k) - \vx_k\|,$ even asymptotically (as $k \to \infty$) and assuming the operator $\mT$ is contractive.\footnote{Consider, for instance, $\mT(\vx) = \gamma \vx$ for $\gamma \in (0, 1)$. Then $\vx_k \to \frac{\lambda}{1-(1-\lambda)\gamma} \vx_0$ as $k$ tends to infinity, while the only fixed point of this operator is the zero element of the space. Thus, unless initialized at the fixed point, \eqref{eq:main-iteration} does not converge to it.} However, as we show in \Cref{lemma:fixed-step-convergence}, for any target error $\epsilon > 0$ and a $\gamma$-Lipschitz operator with $\gamma \in (0, 1],$  it is possible to choose a value of the step size $\lambda \in (0, 1)$ such that \eqref{eq:main-iteration} converges to a point $\vx_k$ with fixed-point error $\|\mT(\vx_k) - \vx_k\| \leq \epsilon$ within the number of iterations competitive with the number of iterations of algorithms that are optimal (in terms of their query oracle complexity) when the operator is nonexpansive ($\gamma = 1$) or contractive ($\gamma < 1$). Further reasoning (through the concept of a ``resolvent'') about why \eqref{eq:main-iteration} is a meaningful iteration is provided in \Cref{sec:fix-Hal-as-resolvent}.

An appealing property of \eqref{eq:main-iteration} that we leverage is that it contracts distances between successive iterates of nonexpansive operators:  for any $k \geq 1,$
\begin{equation}\label{eq:increment-contraction}
    \|\vx_{k+1} - \vx_k\| = (1-\lambda)\|\mT(\vx_k) - \mT(\vx_{k-1})\|.
\end{equation}
We observe that this property is preserved even if the operator $\mT$ is not nonexpansive, but allows for slight expansion, with its Lipschitz constant $\gamma$ being slightly larger than one (e.g., $\gamma = 1 + \lambda/2$). We use this observation to argue that \eqref{eq:main-iteration} can be used to approximate fixed points of operators that are \emph{mildly expansive} (or ``almost nonexpansive;'' see \cite{russell2018quantitative}). In particular, if $\mT$ is a $\gamma$-Lipschitz operator that maps a bounded convex set of diameter $D$ to itself and if $\gamma < 1 + \epsilon/D$ for some $\epsilon > 0$, then applying \eqref{eq:main-iteration} with an appropriate choice of the step size $\lambda$ can ensure convergence to a point with a fixed-point error $ \epsilon$. Further, the number of required iterations is near-linear in $1/\epsilon$ as long as $\gamma \leq 1 + c\epsilon/D$ for any $c \in (0, 1)$ such that $1/(1-c) = \cO(1).$ Although this result may not seem impressive on its own, we note that any algorithm that is based solely on oracle queries to $\mT$ when applied to a $\gamma$-Lipschitz operator $\mT:[0, D]^d \to [0, D]^d$ with $\gamma \geq 1 + C \epsilon/D$ for an absolute constant $C$ (independent of any of the problem parameters), must make an exponential in the dimension number of queries to $\mT$ to find  $\vx$ with fixed-point error $\|\mT(\vx) - \vx\| \leq \epsilon$ \cite{hirsch1989exponential}; see \Cref{sec:fixed-mildly-expansive} for a more detailed discussion. Thus, although modest, up to an absolute constant $C$, our result gets to the boundary of known hardness results. This reframes a ``wrong-looking'' iteration into a boundary-optimal algorithm.

\paragraph{Adaptive algorithms (GHAL and AdaGHAL)} 
Starting from these basic results about \eqref{eq:main-iteration}, we then expand our discussion and contributions by introducing algorithms that gradually decrease the value of the step size parameter $\lambda$. We first consider the case in which $D$, (an upper bound on) the diameter of the set containing algorithm iterates, is known, but the operator $\mT$ is not necessarily nonexpansive (though it is $\gamma$-Lipschitz). In \Cref{sec:GHAL-and-AdaGHAL}, we first introduce Gradual Halpern Algorithm (GHAL, \Cref{algo:grad-iterated-Halpern}), which gradually adjusts the step size $\lambda$ as it progresses. For any $\epsilon > 0$ and without a priori knowledge of the Lipschitz constant $\gamma,$ GHAL simultaneously recovers the optimal oracle complexity for both nonexpansive ($\gamma = 1$) and contractive ($\gamma <1$) operators, as long as $\gamma$ is not trivially small relative to $\epsilon/\norm{\mT(\vx_0 - \vx_0)}$ (see \cite{traub1984optimal,sikorski1987complexity,chou1987optimality,nemirovsky1991optimality,diakonikolas2021potential,park2022exact} for lower bounds). Our second algorithm, Adaptive Gradual Halpern Algorithm (AdaGHAL, \Cref{algo:adaGHAL-rev}), further removes the assumption about knowledge of $D$ (and is thus parameter-free), while maintaining the convergence properties of GHAL. 

We also prove approximation error guarantees and convergence bounds for GHAL and AdaGHAL when $\gamma > 1$ and $\mT$ maps a bounded convex  set of diameter $D$ to itself. In particular, we prove that GHAL can guarantee convergence to a point $\vx$ with fixed-point error $\norm{\mT(\vx) - \vx} \leq c(\gamma - 1)D$ within $\cO(\min\{D/\epsilon, 1/(\gamma - 1)\})$ iterations, each querying $\mT$ a constant number of times, for any $c > 1$ such that $1/(c-1) = \cO(1).$ This result is primarily useful for values of $\gamma$ between one and two, since when the diameter of the set is bounded by $D,$ we have that $\|\mT(\vx) - \vx\| \leq D.$ For $\gamma < 1 + \epsilon/D,$ this result guarantees convergence to a point with fixed-point error $\epsilon$ and, as argued above for \eqref{eq:main-iteration}, gets to the boundary of what is information-theoretically possible for polynomial-time oracle-based algorithms.

\paragraph{A new tractable class: gradually expansive operators} 
We introduce the class of $\alpha$-\emph{gradually expansive} operators, defined as operators mapping a bounded convex set $\cC$ of diameter $D$ to itself with the property that for some $\alpha \in [0, 1)$ and all $\vx, \vy \in \cC$, 
\begin{equation}\label{eq:alpha-grad-nonexpansive-def}
    \|\mT(\vx) - \mT(\vy)\| \leq \Big(1 + \frac{\alpha\max\{\norm{\mT(\vx) - \vx},\, \norm{\mT(\vy) - \vy}\}}{D}\Big)\|\vx - \vy\|.
\end{equation}
Observe that such operators allow for the Lipschitz constant of $\mT$ to be as large as $1 + \alpha$. Moreover, large expansion is possible even  between fixed points of $\mT$ and other points in the set, due to the ``max'' of fixed-point errors in the condition \eqref{eq:alpha-grad-nonexpansive-def}. We prove that for any $\epsilon > 0,$ AdaGHAL applied to an $\alpha$-gradually expansive operator $\mT$ with $\alpha \leq 0.4$ outputs  a point $\vx \in \cC$ with fixed-point error $\norm{\mT(\vx) - \vx} \leq \epsilon$ after at most $\cO(D/\epsilon)$ iterations, each requiring a constant number of oracle queries to $\mT.$ More generally, for $\alpha = \sqrt{2} -1 - \delta$ and any $\delta > 0$, the convergence of AdaGHAL to an $\epsilon$-approximate fixed point is guaranteed with polynomially (in $1/\epsilon, \ln(1/\delta), D$) many queries to $\mT,$ for any $\epsilon > 0.$

\paragraph{Generality and robustness} 
Our guarantees hold in arbitrary normed vector spaces, extend to infinite-dimensional Banach settings, and further generalize to non-positively curved geodesic metric spaces (Busemann spaces). This broad scope highlights the geometric robustness of our framework. 
For contractive, nonexpansive, and mildly expansive operators, our algorithms attain convergence guarantees that match known oracle complexity lower bounds up to universal constants. For gradually expansive operators, our results instead establish a new tractable regime where no prior complexity theory exists. 
These contributions show that fixed-point computation can be handled in a unified, adaptive way across contractive, nonexpansive, mildly expansive, and gradually expansive operators.

Finally, we provide simple numerical examples in \Cref{sec:num-exp} to illustrate the performance of proposed algorithms. More comprehensive numerical evaluation of the algorithms, although of interest, is beyond the scope of the present work, as our primary contributions are theoretical. We conclude the paper with a discussion of open questions.

\subsection{Further Notes on Related Work}

There is a long line of related work on solving fixed-point equations, spanning areas such as convex analysis, mathematical programming, game theory, and theoretical computer science. Here we focus on reviewing those works that are most closely related to ours. 

As mentioned earlier, the basic iteration \eqref{eq:main-iteration} that underlies our algorithms is a fixed-step variant of classical Halpern iteration \cite{halpern1967fixed}. Halpern iteration was originally developed as a method for solving fixed-point equations with nonexpansive operators. Its convergence properties have been studied for many decades. For instance, asymptotic convergence was established by \cite{wittmann1992approximation}, while first nonasymptotic results about convergence of Halpern iteration were established by \cite{leustean2007rates}, using  proof-theoretic techniques of~\cite{kohlenbach2008applied}. These first nonasymptotic convergence results from \cite{leustean2007rates} are rather loose, showing that the fixed point error $\|T(\vx_k) - \vx_k\|$ decays at rate roughly of the order $1/\log(k)$. This convergence rate was tightened to order-$(1/\sqrt{k})$ by \cite{kohlenbach2011quantitative}. The later results of \cite{Sabach:2017} imply that the convergence rate of Halpern iteration with step size $\lambda_k = 1/(k+1)$ can be tightened to order-$\log(k)/k$, while \cite{contreras2023optimal} provided tight rate estimates for different step sizes, proving, among other results, that a slightly different step size $\lambda_k$ (but still of order $1/k$) leads to the optimal convergence rate of order $1/k.$  All these results  \cite{leustean2007rates,kohlenbach2011quantitative,Sabach:2017,contreras2023optimal} apply to general normed real vector spaces. Further, stochastic variants of Halpern iteration have been used to address problems in reinforcement learning, where the focus is on $\ell_\infty$ setups; see, for instance, \cite{bravo2024stochastic,lee2025near}. 

In a somewhat separate line of work, convergence of Halpern iteration was also investigated specifically in Euclidean (or, slightly more generally, Hilbert) spaces. Here, the work of \cite{lieder2019convergence} and slightly later of \cite{kim2019accelerated} provide rates of the order $1/k$ for Halpern iteration with step size $\lambda_k = \Theta(1/k)$ (slightly tighter than the result from \cite{Sabach:2017}), using a computer-assisted approach of \cite{drori2014performance} to analyze Halpern iteration. In \cite{diakonikolas2020halpern}, methods based on Halpern iteration were developed to solve monotone inclusion (or root-finding problems) for monotone Lipschitz operators, using equivalence (up to scaling and translation) between solving fixed-point equations with nonexpansive operators and solving root-finding problems for cocoercive operators, which holds in Hilbert spaces (but not more generally). The same work used a restart-based strategy to address root-finding problems with strongly monotone Lipschitz operators, and the basic ideas from this strategy can be used to address fixed-point problems with potentially contractive operators. These results were further extended to remove logarithmic factors by introducing extragradient-like variants of Halpern iteration in \cite{Yoon2021OptimalGradientNorm,tran2021halpern,lee2021fast}, analyze a variant with adaptive step sizes in \cite{he2024convergence}, and to address stochastic and finite-sum settings in \cite{cai2022stochastic,cai2024variance,chen2024near}. 

Further, \cite{diakonikolas2020halpern} argued that in high dimensions,\footnote{For low-dimensional settings with nonexpansive operators, the optimal oracle complexity is of the order $d \log(1/\epsilon),$ where $d$ is the dimension and $\epsilon$ the target fixed-point error. We refer to \cite{sikorski2009computational} for a relevant discussion of related work.} the convergence rate $1/k$ of Halpern iteration is unimprovable up to a logarithmic factor, using a reduction from the lower bounds for min-max optimization in \cite{ouyang2021lower}. This lower bound was tightened to remove the logarithmic factor (even for affine operators) in \cite{diakonikolas2021potential}, where it was also proved that the convergence rate of the classical Krasnoselskii-Mann (KM) iteration \cite{mann1953mean,krasnosel1955two} (and similar iterations; see \cite{diakonikolas2021potential}) cannot be improved beyond $1/\sqrt{k}.$ Similar results were later obtained independently by \cite{contreras2023optimal} and \cite{park2022exact}, with \cite{park2022exact} also improving the constant in the general  lower bound $\Omega(1/k)$ from \cite{diakonikolas2021potential} for solving fixed-point equations with nonexpansive operators. 

Moving beyond nonexpansive operators, as mentioned in the introduction, approximating fixed-point equations quickly becomes intractable. For $\mT: [0,1]^d \to [0,1]^d$ that is $\gamma$-Lipschitz w.r.t.\ the $\ell_\infty$ norm, with $\gamma > 1,$ the classical results of \cite{hirsch1989exponential} lead to an exponential in the dimension oracle complexity lower bound for approximating the fixed-point error to some $\epsilon \in (0, 1/10)$, already for $\gamma = 1 + c \epsilon$, where $c$ is an absolute constant (see \Cref{sec:fixed-mildly-expansive} for a more detailed discussion). The challenges associated with solving fixed-point equations with general Lipschitz operators have also motivated the introduction of different complexity classes in theoretical computer science; see \cite{yannakakis2009equilibria} for an overview.  It was recently shown in \cite{attias2025fixed} that similar to \cite{hirsch1989exponential} exponential lower bounds also apply to operators mapping the unit $d$-dimensional Euclidean ball to itself and that are $\gamma$-Lipschitz with $\gamma > 1$ w.r.t.\ the $\ell_2$ norm. Further, it was shown in the same work that this exponential oracle complexity lower bound cannot be broken by slightly perturbing the problem, in the framework of smoothed analysis \cite{spielman2004smoothed}. 

We point out here that in these two (finite-dimensional $\ell_\infty$ and $\ell_2$) settings on bounded convex sets, the existence of a fixed point can be guaranteed via classical Brouwer fixed-point theorem \cite{brouwer1911abbildung}. It is well-known that this fixed-point theorem fails in infinite-dimensional Banach spaces even for nonexpansive operators,\footnote{Recall here that classical Banach fixed-point theorem \cite{banach1922operations} that guarantees existence of fixed points requires the operator to be contractive.} while for a $\gamma$-Lipschitz operator $\mT: \cC\to \cC$ with $\gamma >1,$ where $\cC$ is a bounded convex  set of diameter $D$, we have $\inf_{\vx \in \cC}\|\mT(\vx) - \vx\| \leq (1 - 1/\gamma)D$ and this inequality is tight in general \cite{goebel1973minimal}; see also \Cref{sec:fixed-mildly-expansive} for a related discussion. An open question regarding finding solutions $\vx \in \cC$ with $\|\mT(\vx) - \vx\| \leq (1 - 1/\gamma)D + \epsilon$ for $\epsilon > 0$ was raised in \cite{baillion1996rate}. Our results provide guarantees of the form $\norm{T(\vx)-\vx} \le (\gamma-1)D+\epsilon$ within $\cO\big(\frac{\gamma \log(\norm{\mT(\vx_0) - \vx_0}/\epsilon)}{\epsilon}\big)$ queries for $\gamma\in(1,2)${, i.e., within a factor of $\gamma$ of the best attainable bound}. Following the dissemination of this work, \cite{bravo2026minimax} has recently closed this gap by developing a Halpern-based algorithm that guarantees $\|\mT(\vx) - \vx\| \leq (1 - 1/\gamma)D + \epsilon$.

Finally, we note that there is a line of work that establishes convergence results for solving fixed-point equations with $\gamma$-Lipschitz operators with $\gamma > 1$, assuming there is an appropriate notion of a ``local error bound'' (see \cite{pang1997error} for a classical survey on this topic) that bounds distance to fixed-points of $\mT(\vx)$ by its fixed-point error $\norm{\mT(\vx) - \vx}$. The first such result we are aware of is \cite{natarajan1993condition}, which provided an algorithm for solving $\epsilon$-approximate fixed points in the $\ell_\infty$ setting discussed above for \cite{hirsch1989exponential}. The oracle complexity of the algorithm from \cite{natarajan1993condition} is logarithmic in $1/\epsilon$, but is however exponential otherwise, scaling with $(2\kappa (\gamma-1))^d,$ where $\kappa$ is a ``condition number'' associated with the assumed local error bound, further assumed to be bounded by $1/\epsilon.$  Additionally, a recent line of work \cite{russell2018quantitative,lauster2021convergence,hermer2023rates} provided convergence results for solving fixed-point equations with ``almost averaged operators,'' under a similar (though relaxed) local error bound known as metric (sub)regularity. Almost averaged operators are defined in \cite{russell2018quantitative} as operators $\mT$ expressible as $\mT = (1-\alpha)\mId + \alpha \mTt$, where $\alpha \in (0, 1),$ $\mId$ is the identity operator, and $\mTt$ is an operator that is $\gamma$-Lipschitz with $\gamma = \sqrt{1 + \epsilon}$ for some $\epsilon\in [0, 1)$. Under such a set of assumptions and with an appropriate condition that relates the constant of the local error bound to $\epsilon$ and $\alpha,$ it is possible to establish contraction of distance between algorithm iterates and the fixed points of $\mT,$ leading to (usually local and often linear) convergence. We consider this line of work based on local error bounds complementary to ours.

\section{Convergence of the Fixed-Step Halpern Iteration}\label{sec:fixed-step-halpern}

In this section, we analyze the basic fixed-step Halpern iteration \eqref{eq:main-iteration}. Clearly, because the step size does not diminish as we iterate, we cannot eliminate the contribution of the initial point $\vx_0$ and thus cannot hope for the iterates to always converge to a vector $\vx$ with an arbitrarily small fixed-point error. Nevertheless, convergence to an $\epsilon$-approximate fixed point is possible for a sufficiently small step size $\lambda(\epsilon)$ and it applies even if the operator is not guaranteed to be nonexpansive (but only mildly expansive), as we argue in this section. Additionally, convergence results from this section will serve as basic ingredients for obtaining more sophisticated results in later sections.

\subsection{Convergence for Nonexpansive Operators}

The main motivation for considering \eqref{eq:main-iteration} came from an observation about contraction of distances between successive iterates, as discussed in the introduction.\footnote{We note here that tracking distances between successive iterates is common in the analysis of fixed-point iterations; see, for example, \cite{baillion1996rate,cominetti2014rate,Sabach:2017,bravo2021universal,contreras2023optimal}.} %
This simple observation combined with the definition of the iteration \eqref{eq:main-iteration} can then be used to argue about convergence to an approximate fixed point error dependent on $\lambda,$ as follows.

\begin{lemma}[Convergence of the Fixed-Stepsize Iteration]\label{lemma:fixed-step-convergence}
    Given $\vx_0 \in \cE$ and iterates $\vx_{k+1}$ defined by \eqref{eq:main-iteration} for a $\gamma$-Lipschitz operator $\mT$ with $\gamma > 0,$ we have that for all $k \geq 1,$
    \begin{equation}\notag
        \begin{aligned}
             \|\mT(\vx_k) - \vx_k\| \leq\; & (1-\lambda)^{k}\gamma^k \|\mT(\vx_0) - \vx_0\| + \frac{\lambda}{1-\lambda}\|\vx_0 - \vx_{k}\|. %
        \end{aligned} 
    \end{equation}
    If $\gamma \in (0, 1]$ and $\vx_*$ is a fixed point of $\mT,$ then we further have
    \begin{equation}
        \|\mT(\vx_k) - \vx_k\| \leq (1-\lambda)^{k}\gamma^k \|\mT(\vx_0) - \vx_0\| + \frac{2\lambda}{1-\lambda} \|\vx_0 - \vx_{*}\|.
    \end{equation}
    As a consequence, for $\gamma \in (0, 1]$ and given any $\epsilon > 0,$ if $\lambda = \frac{\epsilon}{4D_* + \epsilon}$, where $D_* \geq  \|\vx_0 - \vx_*\|,$ then $\|\mT(\vx_k) - \vx_k\| \leq \epsilon$ after at most $k$ iterations, where
    \[
    k = \Big\lceil \frac{\ln(2 \|\mT(\vx_0) - \vx_0\|/\epsilon)}{\ln(1/(1 - \lambda)) + \ln(1/\gamma)} \Big\rceil = \cO\Big(\frac{\ln(\|\mT(\vx_0) - \vx_0\|/\epsilon)}{\epsilon/D_* + \ln(1/\gamma)}\Big).
    \]
\end{lemma}
\begin{proof}
    Observe first that by \eqref{eq:main-iteration}, we have $\vx_1 - \vx_0 = (1-\lambda)(\mT(\vx_0) - \vx_0).$ Further, by \eqref{eq:increment-contraction} and $\gamma$-Lipschitzness of $\mT,$ 
    \begin{equation}\label{eq:pf-inc-contractions}
        \|\vx_{k+1} - \vx_k\| \leq (1-\lambda)\gamma \|\vx_k - \vx_{k-1}\|
    \end{equation}
    for any $k \geq 1.$ Applying \eqref{eq:pf-inc-contractions} recursively we thus get
    \begin{equation}\label{eq:inc-bound}
        \|\vx_{k+1} - \vx_k\| \leq (1-\lambda)^{k}\gamma^k\|\vx_1 - \vx_0\| = (1-\lambda)^{k+1}\gamma^k\|\mT(\vx_0) - \vx_0\|.
    \end{equation}
    On the other hand, subtracting $\vx_k$ from \eqref{eq:main-iteration} gives $\vx_{k+1} - \vx_k = \lambda(\vx_0 - \vx_k) + (1 - \lambda)(\mT(\vx_k) - \vx_k),$ or, equivalently after a rearrangement, $\mT(\vx_k) - \vx_k = \frac{1}{1-\lambda}(\vx_{k+1} - \vx_k) - \frac{\lambda}{1-\lambda}(\vx_0 - \vx_k).$ Thus, by the triangle inequality,
    \begin{equation}\label{eq:fp-to-inc}
    \begin{aligned}
        \|\mT(\vx_k) - \vx_k\| &\leq \frac{1}{1-\lambda}\|\vx_{k+1} - \vx_k\| + \frac{\lambda}{1-\lambda}\|\vx_0 - \vx_k\|\\
        &\leq (1-\lambda)^{k}\gamma^k\|\mT(\vx_0) - \vx_0\|+ \frac{\lambda}{1-\lambda}\|\vx_0 - \vx_k\|,
    \end{aligned}
    \end{equation}
    where the second inequality is by \eqref{eq:inc-bound}. This gives the first inequality from the lemma statement. 

    For the second inequality in the lemma statement, we observe by the triangle inequality that $\|\vx_0 - \vx_k\| \leq \|\vx_0 - \vx_*\| + \|\vx_{k} - \vx_*\|,$ so it suffices to show that $\|\vx_k - \vx_*\| \leq \|\vx_0 - \vx_*\|,$ for $k\geq 0.$  This follows by induction on $k,$ since using that $\gamma \leq 1$ and $\mT(\vx_*) = \vx_*,$ we have for all $k \geq0,$
    \begin{align*}
        \|\vx_{k+1} - \vx_*\| &\leq \lambda\|\vx_0 - \vx_*\| + (1 - \lambda)\|\mT(\vx_k) - \mT(\vx_*)\|\\
        &\leq \lambda\|\vx_0 - \vx_*\| + (1 - \lambda)\|\vx_k - \vx_*\|. 
    \end{align*}
    To bound the number of iterations until $\|\mT(\vx_k) - \vx_k\| \leq \epsilon,$ observe first that when $\lambda = \frac{\epsilon}{4D_* + \epsilon},$ we have $\frac{2\lambda}{1-\lambda}\|\vx_0 -\vx_*\| \leq \frac{\epsilon}{2},$ so it suffices to bound the number of iterations $k$ until 
    \begin{equation}\notag
        (1-\lambda)^{k}\gamma^k\|\mT(\vx_0) - \vx_0\| \leq \frac{\epsilon}{2}.
    \end{equation}
    Solving the last inequality for $k,$ we get that any $k \geq \frac{\ln(2 \|\mT(\vx_0) - \vx_0\|/\epsilon)}{\ln(1/(1 - \lambda)) + \ln(1/\gamma)}$ suffices. Rounding the right-hand side to the larger integer number leads to the bound claimed in the lemma statement. 
\end{proof}

\subsection{Convergence for Mildly Expansive Operators}\label{sec:fixed-mildly-expansive}

We can observe from the analysis in the last subsection that for \eqref{eq:main-iteration} to converge to an approximate fixed point, $\mT$ is not required to be non-expansive. In particular, we can allow $\gamma$ to be slightly larger than one, and we informally refer to such operators as being ``mildly expansive.'' All that is needed to ensure convergence is that (1) the iterates remain in a bounded set so that $\|\vx_0 - \vx_{k}\|$ is bounded by some $D < \infty$ and (2) for the step size $\lambda$, we have $(1-\lambda)\gamma < 1.$ 
To keep the definitions and the main message reasonably simple, throughout this subsection and whenever we work with possibly expansive operators, we assume that $\mT$ maps a bounded convex  set of diameter $D < \infty$ to itself (\Cref{assp:compact-domain}). In infinite-dimensional spaces, this assumption is necessary to ensure that minimum possible fixed-point error is  bounded \cite{goebel1973minimal}.\footnote{We note here that there are alternative conditions that can guarantee the existence of a fixed point or boundedness of the infimal fixed-point error; however, they would require further specialization of the problem class. See, for instance, \cite{pang2011nonconvex} for one such example.} %

\begin{assumption}\label{assp:compact-domain}
    There exists a bounded convex set $\cC \subset \cE$ such that $\mT: \cC \to \cC$ and $\sup_{\vx, \vy \in \cC}\|\vx - \vy\| \leq D < \infty.$
\end{assumption}

Based on this assumption, we further assume that \eqref{eq:main-iteration} is initialized at $\vx_0 \in \cC$, so that all iterates remain in $\cC.$ We now formally prove the claimed result for a mildly expansive (a.k.a.\ almost nonexpansive \cite{russell2018quantitative}) operator $\mT.$ 

\begin{lemma}[Convergence of \eqref{eq:main-iteration} for a mildly expansive operator.]\label{lemma:mild-expansion} Given a $\gamma$-Lipschitz operator $\mT$ that satisfies \Cref{assp:compact-domain},  $\epsilon > 0$, and $\beta \in (0, 1),$ consider $\{\vx_k\}_{k \geq 0}$ initialized at some $\vx_0 \in \cC$ and updated according to \eqref{eq:main-iteration} for $k \geq 0$, where $\lambda = \frac{\beta \epsilon/D}{1 + \beta\epsilon/D}$ and $D$ is the diameter of $\cC.$ If $0 \leq \gamma < 1 + \frac{\beta \epsilon}{D},$ then $\|\mT(\vx_k) - \vx_k\| \leq \epsilon$ after at most $k  = \big\lceil \frac{\ln(\|\mT(\vx_0) - \vx_0\|/((1-\beta)\epsilon))}{-\ln((1-\lambda)\gamma)} \big\rceil$ iterations.  
\end{lemma}
\begin{proof}
    Applying \Cref{lemma:fixed-step-convergence}, we have that for all $k \geq 0,$
    \begin{align*}
        \|\mT(\vx_k) - \vx_k\| &\leq (1 - \lambda)^k \gamma^k \|\mT(\vx_0) - \vx_0\| + \frac{\lambda}{1 - \lambda}\|\vx_0 - \vx_{k+1}\|\\
        &\leq (1 - \lambda)^k \gamma^k \|\mT(\vx_0) - \vx_0\| + \beta \epsilon.
    \end{align*}
    The bound on $\gamma$ from the lemma statement ensures that $(1 - \lambda)\gamma < 1.$ Thus, to obtain the claimed bound on the number of iterations, we need to determine $k$ that ensures $(1 - \lambda)^k \gamma^k \|\mT(\vx_0) - \vx_0\| \leq (1-\beta)\epsilon.$ Solving this inequality for $k,$ we get that the inequality is satisfied for $k \geq \frac{\ln(\|\mT(\vx_0) - \vx_0\|/((1-\beta)\epsilon))}{-\ln((1-\lambda)\gamma)}$, and thus $k = \big\lceil \frac{\ln(\|\mT(\vx_0) - \vx_0\|/((1-\beta)\epsilon))}{-\ln((1-\lambda)\gamma)} \big\rceil$ iterations suffice. 
\end{proof}

Another way of phrasing \Cref{lemma:mild-expansion} is in terms of the attainable fixed-point error for an operator $\mT$ with an arbitrary Lipschitz constant $\gamma > 1.$ We state this result as the corollary below, which will be particularly useful for discussing this result in infinite-dimensional spaces.  

\begin{corollary}\label{cor:mild-expansion}
    Given a $\gamma$-Lipschitz operator $\mT$ that satisfies \Cref{assp:compact-domain}, where  $\gamma > 1$, and $\beta \in (0, 1),$ consider $\{\vx_k\}_{k \geq 0}$ initialized at some $\vx_0 \in \cC$ and updated according to \eqref{eq:main-iteration} for $k \geq 0$, where $\lambda = 1 - \frac{\beta}{\gamma}$. Then, for any $\epsilon \in (0, \|\mT(\vx_0) - \vx_0\|],$ we have $\|\mT(\vx_k) - \vx_k\| \leq \big(\frac{\gamma}{\beta} - 1\big)D + \epsilon$ after at most $k  = \big\lceil \frac{\ln(\|\mT(\vx_0) - \vx_0\|/\epsilon)}{\ln(1/\beta)} \big\rceil$ iterations.  In particular, for $\beta = 1/(1 + \frac{\epsilon}{2\gamma D})$, we have $\|\mT(\vx_k) - \vx_k\| \leq (\gamma - 1)D + \epsilon$ after $k = \cO\big(\frac{\gamma D \ln(\|\mT(\vx_0) - \vx_0\|/\epsilon)}{\epsilon}\big)$ iterations. 
\end{corollary}
\begin{proof}
    Observe first that $(1-\lambda)\gamma = \beta \in (0, 1)$ and $\frac{\lambda}{1-\lambda}\|\vx_0 - \vx_{k}\| \leq  \big(\frac{\gamma}{\beta} - 1\big) D.$ Thus, \Cref{lemma:fixed-step-convergence} implies that for all $k \geq 0,$ $\|\mT(\vx_k) - \vx_k\| \leq \beta^k \|\mT(\vx_0) - \vx_0\| + \big(\frac{\gamma}{\beta} - 1\big)D.$ Now, for any $\epsilon > 0,$ we have that $\beta^k \|\mT(\vx_0) - \vx_0\| \leq \epsilon$ after $k \geq \log_{1/\beta}\big(\frac{\|\mT(\vx_0) - \vx_0\|}{\epsilon}\big) = \frac{\ln({\|\mT(\vx_0) - \vx_0\|}/{\epsilon})}{\ln(1/\beta)}$ iterations. The last bound follows by a direct calculation, choosing $k$ slightly larger so that  $\beta^k \|\mT(\vx_0) - \vx_0\| \leq \epsilon/2.$
\end{proof}

We now provide a discussion of the obtained results. The entire discussion applies under \Cref{assp:compact-domain}.

\paragraph{Finite-dimensional vector space} Consider first the case of a finite-dimensional vector space, for instance, $\mathbb{R}^d.$ In this case, the existence of a fixed point is guaranteed via Brouwer's fixed point theorem, as discussed before. 
To approximate the fixed point error to some $\epsilon > 0,$ \Cref{lemma:mild-expansion} requires the operator $\mT$ to have only very mild expansion, with Lipschitz constant allowed to be only slightly larger than one. This may not seem impressive on its own, since it only allows $\gamma$ to take values up to $1 + \frac{\beta \epsilon}{D}$ for $\beta \in (0, 1).$ However, it turns out that (up to absolute constants), this is the best we can hope to get with an algorithm that makes a subexponential (in the dimension) number of queries to $\mT$, due to the lower bound from \cite{hirsch1989exponential}. In particular, let $M$ be the Lipschitz constant of $\mT - \mId,$ and observe that $\gamma \leq  1 + M$ (so $M \geq \gamma -1 $) by the triangle inequality. Theorem 2 in \cite{hirsch1989exponential} implies that for any $\epsilon \in (0, 1/10),$ any oracle-based algorithm applied to an operator $\mT$ that maps the unit hypercube $[0, 1]^d$ (so $D=1$), where $d \geq 3$, to itself and is $(1+M)$-Lipschitz continuous with respect to the  $\ell_\infty$ norm, in the worst case must make at least %
\begin{equation}\label{eq:lower-bound}
    \Big(c\Big(\frac{1}{\epsilon} - 10\Big)M\Big)^{d-2}
\end{equation}
queries to $\mT$ to reach a vector $\vxh$ in the unit hypercube such that its fixed point error is at most $\epsilon$ (i.e., $\|\mT(\vxh) - \vxh\|_\infty \leq \epsilon$), where $c \geq 10^{-5}$. Here, the important thing about $c$ is that it is an absolute constant, independent of the problem parameters. The authors conjecture that $c$ can in fact be made close to one. Observe  that \Cref{lemma:mild-expansion} gives a polynomial bound on the number of queries of \eqref{eq:main-iteration} for problems described in this lower bound whenever $\gamma \leq 1 + {\beta \epsilon}$ (as $D = 1$) and $\beta$ is smaller than one by some absolute constant (e.g., we can take $\beta = 0.99$ and bound the number of queries by $\cO\big(\frac{ \ln(1/\epsilon)}{\epsilon}\big)$). On the other hand, the lower bound stated in \eqref{eq:lower-bound} implies that $M$ must be bounded by an absolute constant times $\epsilon$ for the number of worst-case oracle queries of any algorithm to be subexponential, meaning that there exists an absolute constant $c' > 0$ such that any algorithm that finds an $\epsilon$-fixed point of a $(1 + c'\epsilon)$-Lipschitz operator must make an exponential in the dimension number of queries in the worst case. Thus, the result from \Cref{lemma:mild-expansion} cannot be improved, modulo the constant multiplying $\epsilon$ in the bound on $\gamma.$ 

While the above discussion considered the case $D=1,$ we note that, by a simple rescaling argument, the dependence on a more general value of $D>0$ as in our bound is similarly unavoidable (up to absolute constants, as discussed above). To see this, consider the following. Let $\mF = \mT - \mId$, so that $M$ is the Lipschitz constant of $\mF,$ for $\mF$ mapping the unit hypercube in $\sR^d$ to itself, as in the above discussion. Define $\mFh$---an operator mapping the hypercube $[0, D]^d$ to itself---via $\mFh(\vx) := \mF(\vx/D)$. It is immediate that the Lipschitz constant of $\mFh$ is $\widehat{M} = M/D$ and that for any $\epsilon > 0$ and any $\vx$ from the side-$D$ hypercube we have $\|\mFh(\vx)\|_\infty \leq \epsilon$ if and only if $\|\mF(\vx/D)\|_\infty \leq \epsilon$. Now, the lower bound stated in \eqref{eq:lower-bound} implies that for $d \geq 3$ and any $\epsilon \in (0, 1/10)$, any algorithm based on oracle queries to $\mFh$ must make at least
\begin{equation}\label{eq:lower-bound-D}
    \Big(c\Big(\frac{1}{\epsilon} - 10\Big)\widehat{M}D\Big)^{d-2}
\end{equation}
queries to $\mFh$ in the worst case to output a point $\vx$ with  error $\|\mFh(\vx)\|_\infty \leq \epsilon$. A consequence of this statement is that there exists an absolute constant $c'$ such that no (oracle query-based) algorithm can output a point with fixed-point error $\epsilon \in (0, 1/10)$ for an arbitrary $(1 + c'\epsilon/D)$-Lipschitz operator mapping a bounded convex set of diameter $D$ to itself without incurring an exponential dependence in the dimension (assuming $d\geq 3$) for the number of required oracle queries to the operator. Thus, for $d \geq 3,$ the result from \Cref{lemma:mild-expansion} cannot be improved (modulo the absolute constant multiplying $\epsilon/D$) without either making additional assumptions about the operator $\mT$ or incurring exponential complexity in the dimension. Put differently, for $\gamma > 1$ by any however small constant positive amount, no oracle-based algorithm can guarantee convergence to an arbitrarily small fixed-point error in polynomial time. 

\paragraph{Infinite-dimensional space.} For infinite-dimensional Banach spaces $(\cE, \norm{\cdot})$, the existence of a fixed point is not guaranteed  under \Cref{assp:compact-domain}, even if the operator $\mT$ is nonexpansive \cite{browder1965nonexpansive,gohde1965prinzip,kirk1965fixed}. Instead, for $\gamma$-Lipschitz operators $\mT$ with $\gamma > 1,$ it is only possible to guarantee that $\inf_{\vx \in \cC}\|\mT(\vx) - \vx\| \leq (1 - 1/\gamma)D,$ and this inequality is tight in general \cite{goebel1973minimal}.\footnote{Since the problem is no longer guaranteed to have a fixed point, the problem now becomes to find solutions with minimum value of $\|\mT(\vx) - \vx\|,$ which is also referred to as the ``minimal displacement problem;'' see \cite{goebel1973minimal,baillion1996rate}.} In other words, the best fixed-point error we can hope to achieve with any algorithm would be of the form $(1 - 1/\gamma)D + \epsilon,$ for $\epsilon > 0.$   Based on \Cref{cor:mild-expansion}, \eqref{eq:main-iteration} can attain fixed-point error $(\gamma - 1)D + \epsilon$ after $k = \cO\big(\frac{\gamma\ln(\|\mT(\vx_0) - \vx_0\|/\epsilon)}{\epsilon}\big)$ iterations, which is off by a multiplicative factor $\gamma$ and additive error $\epsilon$ compared to the bound on the minimal error $(1 - 1/\gamma)D,$ since $(\gamma - 1)D + \epsilon = \gamma(1 - 1/\gamma)D + \epsilon$. To the best of our knowledge, this is the first result that gives any nontrivial bound on the fixed-point error of general Lipschitz operators using a finite number of operator evaluations, despite a relevant open question being raised in \cite{baillion1996rate}.  

We note here that for sufficiently large $\gamma,$ the lower bound from \cite{hirsch1989exponential} stated in \eqref{eq:lower-bound-D} for the hypercube of side $D$ in $\sR^d$ ($d \geq 3$) does not preclude results of the form $(1 - 1/\gamma)D + \epsilon,$ for $\epsilon > 0,$ on the basis of such problems being included in the more general fixed-point problems in Banach spaces $(\cE, \norm{\cdot})$. The reason is that the shift by 10 in \eqref{eq:lower-bound-D} makes the lower bound meaningful only for sufficiently small values of the target fixed-point error. It is thus an interesting question whether error closer to $(1 - 1/\gamma)D + \epsilon$ than what we obtained is algorithmically attainable in general Banach spaces. This was conjectured true for the KM iteration in \cite{baillion1996rate}; however, to the best of our knowledge, such a result has not been established to date. As mentioned before, subsequent work \cite{bravo2026minimax} has established such an upper bound for a Halpern-based algorithm. 

%

\subsection{Extension to Geodesic Metric Spaces}\label{sec:busemann}

Even though our results are presented for normed vector spaces, they can be generalized to certain geodesic metric spaces. The reason is that very few properties of the norm are used in the analysis of \eqref{eq:main-iteration} and algorithms developed based on it in the upcoming section; namely, in addition to properties defining any metric, only convexity of the norm (or a metric) is needed. This means that the results extend to non-positively curved geodesic metric spaces known as Busemann spaces, which are fully characterized by convex metrics on every geodesic (see the textbook \cite{bacak2014convex} for relevant definitions and, in particular, Proposition 1.1.5 for the characterization of Busemann geodesic metric spaces via convexity of the metric). 

For concreteness, let $(\cG, d)$ denote a geodesic metric space. With a slight abuse of notation confined to the rest of this section, we use $d$ to denote the metric associated with the space $(\cG, d)$. Let $\mT: \cG \to \cG$ be an operator that is $\gamma$-Lipschitz for some $\gamma >0,$ which for the present setting means that $d(\mT(\vx), \mT(\vy)) \leq \gamma\, d(\vx, \vy)$ for any $\vx, \vy \in \cG.$

For iterates of \eqref{eq:main-iteration} to remain in $\cG,$ the iteration update for a step size $\lambda \in (0, 1)$ is written as
\begin{equation}\label{eq:fixHal-geodesic}
    \vx_{k+1} = \lambda \vx_0 \oplus (1-\lambda)\mT(\vx_k), 
\end{equation}
where the above equation is interpreted as $\vx_{k+1}$ being the element of $\cG$ on a geodesic from $\vx_0$ to $\mT(\vx_k)$ such that $d(\vx_{k+1}, \vx_0) = (1-\lambda)d(\mT(\vx_k), \vx_0).$ 

For concreteness, we show how to obtain a counterpart to \Cref{lemma:fixed-step-convergence} for a $\gamma$-Lipschitz operator $\mT$ on a Busemann space $(\cG, d).$ With this argument, all other results in this paper can be generalized to Busemann spaces in a straightforward manner, which is omitted for brevity.
\begin{lemma}[Convergence of the Fixed-Stepsize Iteration in a Busemann Space]\label{lemma:fixed-step-convergence-Busemann}
    Let  $(\cG, d)$ be a Busemann space and $\mT: \cG \to \cG$ be an operator that is $\gamma$-Lipschitz on $\cG$ for some $\gamma >0.$ Given $\vx_0 \in \cG$ and iterates $\vx_{k+1}$ defined by \eqref{eq:fixHal-geodesic} for $k \geq 0$, we have that for all $k \geq 1,$
    \begin{equation}\notag
        \begin{aligned}
             d(\mT(\vx_{k}), \vx_{k}) \leq\; & (1-\lambda)^{k}\gamma^{k} d(\mT(\vx_0), \vx_0) + \frac{\lambda}{1-\lambda}d(\vx_0, \vx_{k}).
        \end{aligned} 
    \end{equation}
    If $\gamma \in (0, 1]$ and $\vx_*$ is a fixed point of $\mT,$ then we further have
    \begin{equation}
        d(\mT(\vx_{k}), \vx_{k}) \leq (1-\lambda)^{k}\gamma^{k} d(\mT(\vx_0), \vx_0) + \frac{2\lambda}{1-\lambda}d(\vx_0, \vx_{*}).
    \end{equation}
    As a consequence for $\gamma \in (0, 1]$, given any $\epsilon > 0,$ if $\lambda = \frac{\epsilon}{4D_* + \epsilon}$, where $D_* \geq  d(\vx_0, \vx_*),$ then $d(\mT(\vx_k), \vx_k) \leq \epsilon$ after at most $k$ iterations, where
    \[
    k = \Big\lceil \frac{\ln(2 d(\mT(\vx_0), \vx_0)/\epsilon)}{\ln(1/(1 - \lambda)) + \ln(1/\gamma)} \Big\rceil = \cO\Big(\frac{\ln(d(\mT(\vx_0), \vx_0)/\epsilon)}{\epsilon/D_* + \ln(1/\gamma)}\Big).
    \]
\end{lemma}
\begin{proof}
    Similar to the proof of \Cref{lemma:fixed-step-convergence}, we first bound the distance between successive iterates. For $k = 0$, we have, by the definition of the iteration \eqref{eq:fixHal-geodesic}, that $d(\vx_1, \vx_0) = (1-\lambda)d(\mT(\vx_0), \vx_0)$. For $k \geq 1$, by convexity of the metric $d$ and by \eqref{eq:fixHal-geodesic}, 
    \begin{align}
        d(\vx_{k+1}, \vx_k) \leq \; & \lambda d(\vx_0, \vx_0) + (1-\lambda)d(\mT(\vx_k), \mT(\vx_{k-1})) \notag\\
        \leq \; & (1-\lambda)\gamma d(\vx_k, \vx_{k-1}),\notag
    \end{align}
    where the last inequality is by $d(\vx, \vx) = 0$ for any $\vx \in \cG$ (by the definition of a metric) and by $d(\mT(\vx_k), \mT(\vx_{k-1})) \leq \gamma d(\vx_k, \vx_{k-1})$, which holds by the definition of $\gamma$-Lipschitzness of $\mT$.

    Applying the last inequality recursively and using  $d(\vx_1, \vx_0) = (1-\lambda)d(\mT(\vx_0), \vx_0)$, we thus have
    \begin{equation}\label{eq:Buseman-displacement-bnd}
        d(\vx_{k+1}, \vx_k) \leq (1-\lambda)^{k + 1}\gamma^{k} d(\mT(\vx_0), \vx_0). 
    \end{equation}
    Now, again using the definition of the iteration \eqref{eq:fixHal-geodesic} and convexity of $d$, we get
    \ifsiopt
     \begin{align}
        d(\mT(\vx_{k+1}), \vx_{k+1}) \leq \; & \lambda d(\vx_0, \mT(\vx_{k+1})) + (1-\lambda) d(\mT(\vx_{k+1}), \mT(\vx_k))\notag\\
        \leq \; & \lambda d(\vx_0, \mT(\vx_{k+1})) + (1-\lambda)\gamma d(\vx_{k+1}, \vx_k)\notag\\
        \leq \; & \lambda d(\vx_0, \mT(\vx_{k+1})) + (1-\lambda)^{k + 2}\gamma^{k + 1} d(\mT(\vx_0), \vx_0) \notag\\
        \leq \; & \lambda d(\vx_0, \vx_{k+1}) + \lambda d(\mT(\vx_{k+1}), \vx_{k+1})\notag\\
        &+ (1-\lambda)^{k + 2}\gamma^{k + 1} d(\mT(\vx_0), \vx_0), \label{eq:Buseman-first-claim}
    \end{align}
    \else 
    \begin{align}
        d(\mT(\vx_{k+1}), \vx_{k+1}) \leq \; & \lambda d(\vx_0, \mT(\vx_{k+1})) + (1-\lambda) d(\mT(\vx_{k+1}), \mT(\vx_k))\notag\\
        \leq \; & \lambda d(\vx_0, \mT(\vx_{k+1})) + (1-\lambda)\gamma d(\vx_{k+1}, \vx_k)\notag\\
        \leq \; & \lambda d(\vx_0, \mT(\vx_{k+1})) + (1-\lambda)^{k + 2}\gamma^{k + 1} d(\mT(\vx_0), \vx_0) \notag\\
        \leq \; & \lambda d(\vx_0, \vx_{k+1}) + \lambda d(\mT(\vx_{k+1}), \vx_{k+1}) + (1-\lambda)^{k + 2}\gamma^{k + 1} d(\mT(\vx_0), \vx_0), \label{eq:Buseman-first-claim}
    \end{align}
    \fi
    where the second inequality uses $\gamma$-Lipschitzness of $\mT$, the third inequality uses \eqref{eq:Buseman-displacement-bnd}, and the last inequality is by applying the triangle inequality as $d(\vx_0, \mT(\vx_{k+1})) \leq d(\vx_0, \vx_{k+1}) + d(\mT(\vx_{k+1}), \vx_{k+1})$. Grouping the like terms in \eqref{eq:Buseman-first-claim} and simplifying now gives the first claim. 

    The rest of the lemma is proved in the same way as \Cref{lemma:fixed-step-convergence}, as the only properties of the norm used there are the triangle inequality (which holds for an arbitrary metric by definition) and convexity of the metric (which holds for metrics in Busemann spaces). Thus, we omit the proof for brevity.
\end{proof}

\subsection{Connection to Resolvents and Regularization}\label{sec:fix-Hal-as-resolvent}

Although this was not the original motivation for considering \eqref{eq:main-iteration}, upon writing the initial version of the results, the author realized that the considered iteration effectively defines the resolvent of the mapping $\mT.$ To be more specific, assuming $(\cG, d)$ is a Hadamard space (a special case of Busemann spaces) and $\mT$ is nonexpansive w.r.t.\ the metric $d$, \cite{bacak2014convex} (see Definition 4.2.1 in the same book) defines the resolvent $\mR_{\tau}(\vx)$ of the operator $\mT$ as the fixed point of the mapping $\mS_{\vx, \tau}$ defined via
\begin{equation}\label{eq:resolvent-proxy}
    \mS_{\vx, \tau}(\vy) = \frac{1}{1 + \tau}\vx \oplus \frac{\tau}{1 + \tau}\mT(\vy), \; \vy \in \cG.
\end{equation}
This definition remains meaningful if we consider more general Busemann spaces and possibly expansive operators. It is immediate from this definition that $\mR_{\tau}(\vx) = \vx$ if and only if $\mT(\vx) = \vx.$ 

It is immediate from this definition that the fixed-step Halpern iteration \eqref{eq:fixHal-geodesic} is equivalent to $\vx_{k + 1} = \mS_{\vx, \tau}(\vx_k)$ for $\tau$ satisfying $\frac{1}{1 + \tau} = \lambda.$ In other words, \eqref{eq:fixHal-geodesic} corresponds to  Picard iteration applied to the mapping $\mS_{\vx, \tau},$ whose fixed point is precisely the resolvent of $\mT$ at $\vx_0,$ $\mR_{\tau}(\vx_0) = \frac{1}{1 + \tau}\vx_0 \oplus \frac{\tau}{1 + \tau}\mT(\mR_{\tau}(\vx_0))$. Thus, another interpretation of the fixed-step Halpern iteration is as a ``regularization'' method, porting the original problem with possibly (mildly) expansive operator to one with a contractive operator $\mS_{\vx, \tau},$ much like adding a quadratic term to an objective function in smooth (weakly) convex optimization is used to replace the original problem by a smooth strongly convex one. This interpretation further allows viewing our algorithms introduced in the next section as variants of \emph{inexact} proximal point algorithms with geometrically decreasing error of the computed resolvent estimate and geometrically increasing step size parameter $\tau_k$ (i.e., geometrically decreasing $\frac{\lambda_k}{1-\lambda_k} = \frac{1}{ \tau_k}$). Thus, the proposed approach can also be viewed as a generalization of a recent regularization-based approach for smooth (convex and nonconvex) optimization in Euclidean setups \cite{lan2023optimal} to the more general fixed point equations with Lipschitz operators in Busemann spaces.      

\section{Gradual Halpern Algorithm and its Adaptive Version}\label{sec:GHAL-and-AdaGHAL}

Although it is possible to use the fixed-step Halpern iteration \eqref{eq:main-iteration} to approximately solve fixed-point equations with contractive, nonexpansive, or even mildly expansive operators, there are multiple downsides to its direct use. First, even if we knew everything about the problem, we fixed the target approximation $\epsilon > 0$, and set the step size $\lambda = \lambda(\epsilon)$ accordingly, the iteration complexity obtained from \Cref{lemma:fixed-step-convergence} would be suboptimal by a factor $\ln\big(\frac{\|\mT(\vx_0) - \vx_0\|}{\epsilon})$ for nonexpansive operators \cite{diakonikolas2021potential,park2022exact}. Second, it is generally undesirable to set the target error in advance as done in this case, because, in the worst case, the algorithm would not converge to arbitrarily small error, even asymptotically. Third, the value of $\lambda$ prescribed by \Cref{lemma:fixed-step-convergence} requires an upper bound on the iterate distance $\|\vx_{k+1} - \vx_0\| \leq D$ to be useful, which is a strong assumption particularly when the operator $\mT$ is nonexpansive (i.e., when $\gamma \leq 1$). Finally, since the algorithm can provide guarantees for any value of $\gamma \in (0, 1+\epsilon/D)$, it would be useful to have an algorithm that can automatically adapt to the value of $\gamma,$ possibly even locally, without explicit prior knowledge about local or global values of $\gamma.$ In this section, we provide and discuss algorithms that address these issues. %
\ifsiopt %
\else

\fi
For clarity of exposition, we handle separately the cases in which the operator $\mT$ is nonexpansive and in which it is mildly expansive, and build from the setting where we know the diameter of the set containing algorithm iterates to a fully parameter-free setting.

\subsection{Gradual Halpern Algorithm (GHAL)}

We consider first a simple setting in which we have an upper estimate $D$ for the diameter of the set containing algorithm's iterates. In particular, for nonexpansive operators ($\gamma \leq 1$), we have $D \leq 2 \|\vx_0 - \vx_*\|$, while under \Cref{assp:compact-domain}, $D$ is the diameter of the set $\cC$ for which $\mT: \cC \to \cC.$ The pseudocode for the algorithm (Gradual Halpern Algorithm---GHAL) is provided in \Cref{algo:grad-iterated-Halpern}. In the algorithm, there is no apparent benefit to choosing $\beta$ to be anything other than an absolute constant, so in the analysis below it helps to regard $\beta$ as some number between zero and one, such as, for example, $0.1, 0.5,$ or $0.9.$

\begin{algorithm}
\caption{Gradual Halpern Algorithm (GHAL)}\label{algo:grad-iterated-Halpern}
    \begin{algorithmic}[1]
    \Statex \textbf{Input}: $\epsilon > 0, \vx_0, D < \infty,$ $\beta \in (0, 1), \beta' \in (0, 1)$
     \Statex \textbf{Initialization}: $\epsilon_0 \leftarrow \|\mT(\vx_0) - \vx_0\|,\, k \leftarrow 0,\, \vxh_0 \leftarrow \vx_0$
     \While{$\|\mT(\vxh_k) - \vxh_k\| > \epsilon$}
     \State $k \leftarrow k + 1$
     \State $\epsilon_k \leftarrow \beta \epsilon_{k-1}$
     \State $\lambda_k \leftarrow \frac{\beta  \epsilon_k/D}{1 + \beta \epsilon_k/D}$
     \State $\vy_0 \leftarrow \vxh_{k-1}$, $j \leftarrow 0$ \Comment{Initialize \eqref{eq:main-iteration}}
     \While{$\|\mT(\vy_j) - \vy_j\| > \epsilon_k$} \Comment{Run \eqref{eq:main-iteration} until target error is reached}
     \State $\vy_{j+1}\leftarrow \lambda_k \vy_0 + (1-\lambda_k)\mT(\vy_j)$
     \If{$j \geq 1$ and $\|\vy_{j+1} - \vy_j\| \geq (1-\beta'\lambda_k)\|\vy_j - \vy_{j-1}\|$} \Comment{Safeguard for potentially expansive operators}
     \State Halt the algorithm and return $\vxh_k = \argmin\{\|\mT(\vy_j) - \vy_j\|, \|\mT(\vy_0) - \vy_0\|\}$
     \EndIf
     \State $j \leftarrow j + 1$
     \EndWhile
     \State $\vxh_k \leftarrow \vy_{j}$
     \EndWhile
     \State\Return $\vxh_{k}$
    \end{algorithmic}
\end{algorithm}

\begin{proposition}\label{prop:opt-it-complexity}
     Let $\mT$ be a $\gamma$-Lipschitz operator and let $\vxh_k$ be the output of \Cref{algo:grad-iterated-Halpern} initialized at some $\vx_0$ and run for $\beta \in (0, 1)$ such that $1/\beta, 1/(1-\beta)$ are absolute constants independent of the problem parameters. Suppose $\norm{\vy_j - \vy_0} \leq D$ for any iterates $\vy_j$ within \Cref{algo:grad-iterated-Halpern}. If $\gamma \in (0, 1],$ then $\|\mT(\vxh_k) - \vxh_k\| \leq \epsilon$ and the algorithm takes 
    $$
    \cO\Big(\min\Big\{ {\ln\Big(\frac{\|\mT(\vx_0) - \vx_0\|}{\epsilon}\Big)} \frac{1 + \ln(1/\gamma)}{\ln(1/\gamma)},\, \ln\Big(\frac{1}{\gamma}\Big) + \frac{\ln(\gamma \|\mT(\vx_0) - \vx_0\|/\epsilon)}{\ln(1/\gamma)},\, \frac{D}{\epsilon} \Big\}\Big)
    $$ 
    \eqref{eq:main-iteration} iterations (and consequently the same number of oracle calls to $\mT$). 

    If $\gamma \in (1, 1 + \frac{\|\mT(\vx_0) - \vx_0\| \beta^2(1-\beta')}{D}],$ then \Cref{algo:grad-iterated-Halpern} outputs a point $\vxh_k$ with $\|\mT(\vxh_k) - \vxh_k\| \leq \max\{\epsilon, \bar{\epsilon}\}$, where $\bar{\epsilon} = \frac{D}{\beta^2}\frac{\gamma - 1}{1 - \beta'}$, after  $\cO(\min\{\frac{D}{\epsilon \beta'},\, \frac{\beta(1-\beta')}{\gamma - 1}\})$ oracle calls to $\mT.$
\end{proposition}
\begin{proof}
We carry out the proof by separately handling the cases in which the operator is nonexpansive and in which it is mildly expansive.

\noindent\textbf{Case 1: Nonexpansive operator.} 
We first consider the case in which $\gamma \in (0, 1]$ and observe that it is impossible for the algorithm to halt due to the condition in Line 8, since we established in the proof of \Cref{lemma:fixed-step-convergence} that for nonexpansive operators, \eqref{eq:main-iteration} must be contracting distances between successive iterates by a factor at most $1-\lambda_k$. 
    The claim $\|\mT(\vxh_k) - \vxh_k\| \leq \epsilon$ thus follows by the outer while loop condition---for the while loop to terminate and the algorithm to halt, this condition must hold. Because $\epsilon_k$ reduces by  a factor $\beta \in (0, 1)$ in each sequence of calls to \eqref{eq:main-iteration} within the inner while loop, there are at most $\lceil\log_{1/\beta}(\epsilon_0/(\beta\epsilon))\rceil$ iterations of the while loop. Using \Cref{lemma:fixed-step-convergence}, each iteration $k$ of the while loop makes at most $\Big\lceil \frac{\log_{1/\beta}(\|\mT(\vxh_{k-1}) - \vxh_{k-1}\|/(\beta \epsilon_k))}{\log_{1/\beta}(1/(1 - \lambda_k)) + \log_{1/\beta}(1/\gamma)} \Big\rceil = \cO(1 + \frac{1}{\epsilon_k/D + \ln(1/\gamma)}) = \cO(1 + \frac{1}{\epsilon_0 \beta^k/D + \ln(1/\gamma)})$ calls to \eqref{eq:main-iteration}, where we used that $\|\mT(\vxh_{k-1}) - \vxh_{k-1}\| \leq \epsilon_{k-1} = \epsilon_k/\beta$ and that $1/\beta$ is an absolute constant.  Hence the total number of calls to \eqref{eq:main-iteration}, each requiring a single oracle call to $\mT$, is
    \ifsiopt
    \begin{align*}
        &\cO\bigg(\sum_{k=1}^{\log_{1/\beta}(\epsilon_0/\epsilon)}\Big(1 + \frac{1}{\epsilon_0\beta^k/D + \ln(1/\gamma)}\Big) \bigg)\\ =\; & \cO\bigg(\min\Big\{\ln(\epsilon_0/\epsilon) \frac{ 1 + \ln(1/\gamma)}{\ln(1/\gamma)},\, \frac{D}{\epsilon_0} \Big(\frac{1}{\beta}\Big)^{\log_{1/\beta}(\epsilon_0/\epsilon)} \Big\}\bigg)\\
        = \; &\cO\bigg(\min\Big\{\ln(\epsilon_0/\epsilon) \frac{ 1 + \ln(1/\gamma)}{\ln(1/\gamma)},\, \frac{D}{\epsilon} \Big\}\bigg).
    \end{align*}
    \else
    \begin{align*}
        \cO\bigg(\sum_{k=1}^{\log_{1/\beta}(\epsilon_0/\epsilon)}\Big(1 + \frac{1}{\epsilon_0\beta^k/D + \ln(1/\gamma)}\Big) \bigg) &= \cO\bigg(\min\Big\{\ln(\epsilon_0/\epsilon) \frac{ 1 + \ln(1/\gamma)}{\ln(1/\gamma)},\, \frac{D}{\epsilon_0} \Big(\frac{1}{\beta}\Big)^{\log_{1/\beta}(\epsilon_0/\epsilon)} \Big\}\bigg)\\
        &= \cO\bigg(\min\Big\{\ln(\epsilon_0/\epsilon) \frac{ 1 + \ln(1/\gamma)}{\ln(1/\gamma)},\, \frac{D}{\epsilon} \Big\}\bigg).
    \end{align*}
    \fi

    Now consider $\gamma < \beta/2.$ Observe that, by the above argument, no iteration of the while loop can result in more than a constant number of iterations, since $\cO(1 + 1/\ln(1/\gamma))$ is bounded by a constant in this case. We claim that in this case, after an initial small number of while loop iterations, the inner while loop of \Cref{algo:grad-iterated-Halpern} makes only one call to \eqref{eq:main-iteration}, and, moreover, the contraction of the fixed point error quickly reaches the  order of $\gamma$. Let $\vy_0 = \vxh_{k-1},$ $\vy_1 = \lambda_k \vy_0 + (1 - \lambda_k)\mT(\vy_0).$ Then, rearranging the $\vy_1$ update equation, we get that
    \begin{align}\label{eq:first-it-rearranged}
        \vy_1 - \vy_0 = (1 - \lambda_k)(\mT(\vy_0) - \vy_0) & \Rightarrow \|\vy_1 - \vy_0\| \leq (1 - \lambda_k)\|\mT(\vy_0) - \vy_0\|,
    \end{align}
    and, further
    \begin{equation}\notag
        (1 - \lambda_k)(\vy_1 - \mT(\vy_1)) = \lambda_k(\vy_0 - \vy_1) + (1-\lambda_k)(\mT(\vy_0) - \mT(\vy_1)),
    \end{equation}
    which, using the triangle inequality, $\gamma$-contractivity of $\mT$, and \eqref{eq:first-it-rearranged}, implies
    \begin{align}
        \|\vy_1 - \mT(\vy_1)\| &\leq \Big(\frac{\lambda_k}{1 - \lambda_k} + \gamma\Big)\|\vy_1 - \vy_0\|\notag\\
        &\leq \big(\lambda_k + (1 - \lambda_k)\gamma\big) \|\mT(\vy_0) - \vy_0\|.
    \end{align}
    Now, since $\|\mT(\vy_0) - \vy_0\| \leq \epsilon_{k-1}$ and calls to \eqref{eq:main-iteration} in the inner while loop terminate as soon as $\|\vy_j - \mT(\vy_j)\| \leq \epsilon_k = \beta\epsilon_{k-1},$ a sufficient condition for the while loop to terminate in one iteration is that $\lambda_k + (1 - \lambda_k)\gamma \leq \beta.$ Equivalently, this holds if $\lambda_k \leq \frac{\beta - \gamma}{1 - \gamma}.$ More generally, we have that $\Bar{\lambda}_k := \frac{\lambda_k}{1-\lambda_k} \leq \gamma$ is sufficient to ensure that this while loop terminates in one iteration and $\lambda_k + (1 - \lambda_k)\gamma < 2\gamma.$ This means that once $\bar{\lambda}_k \leq \gamma$, each iteration of \Cref{algo:grad-iterated-Halpern} contracts the fixed point error by a factor at most $2 \gamma.$

    By the definition of $\lambda_k,$ we have that $\Bar{\lambda}_k := \frac{\lambda_k}{1-\lambda_k} = \frac{\beta \epsilon_k}{D}$, thus the condition $\bar{\lambda}_k \leq \gamma$ is equivalent to $\epsilon_k \leq D \gamma/\beta.$ The total number of iterations until this value of $\epsilon_k$ is reached (of which, recall, each makes a constant number of calls to \eqref{eq:main-iteration}) is $\cO(\log_{1/\beta}(\epsilon_0/\epsilon_k)) = \cO\big(\ln\big(\frac{\epsilon_0 \beta}{D \gamma}\big)\big).$ All subsequent iterations contract the fixed point error by a factor at most $2\gamma,$ so the remaining number of iterations until the algorithm halts is $\cO\big(\frac{\ln(D\gamma /\epsilon)}{\ln(1/\gamma)}\big).$ Observe here that, by the triangle inequality and $\gamma$-contractivity of $\mT,$ we can conclude that $\epsilon_0 = \|\vx_0 - \mT(\vx_0)\| \geq \|\vx_0 - \vx_*\| - \|\mT(\vx_0) - \mT(\vx_*)\| \geq (1 - \gamma)\|\vx_0 - \vx_*\|.$ Thus, if $D = \cO(\|\vx_0 - \vx_*\|),$ then we also have $D = \cO(\epsilon_0),$ and the iteration/oracle complexity in this case simplifies to 
    \begin{equation}\notag
        \cO\Big(\ln({1}/{\gamma}) + \frac{\ln(\gamma \epsilon_0/\epsilon)}{\ln(1/\gamma)}\Big).
    \end{equation}

    To complete the proof of the first case, it remains to recall that $\epsilon_0 = \|\mT(\vx_0) - \vx_0\|$ and observe that when $\gamma$ is close to one (but smaller than one), the first two terms in the oracle complexity bound are of the same order.

    \noindent\textbf{Case 2: Mildly expansive operator.} Suppose now that $\gamma \in (1, 2)$. Then it is possible for the algorithm to halt either due to the exit condition of the outer while loop or due to the safeguard condition in Line 8. If the algorithm halts because of the while loop exit condition, then it outputs an $\vxh_k$ with the fixed point error $\|\mT(\vxh_k) - \vxh_k\| \leq \epsilon.$ Moreover, because in this case the halting condition from Line 8 is never reached, we have that the \eqref{eq:main-iteration} iterations in the inner loop are always contracting distances by a factor at most $1-\beta' \lambda_k,$ and so we can carry out the same argument as in Case 1 to bound the total number of iterations by $\cO(\frac{D}{\epsilon \beta'}).$

    Now suppose that \Cref{algo:grad-iterated-Halpern} halts due to the safeguard condition in Line 8, and suppose that this happens in iteration $k$. The solution output by the algorithm, by the condition in Line 9, has fixed-point error no worse than $\|\mT(\vxh_{k-1}) - \vxh_{k-1}\| \leq \epsilon_{k-1}.$ Because, by \eqref{eq:main-iteration}, we have for $j \geq 1$ that
    \begin{equation*}
        \|\vy_{j+1} - \vy_j\| = (1-\lambda_k)\|\mT(\vy_j)- \mT(\vy_{j-1})\| \leq (1-\lambda_k)\gamma \|\vy_j - \vy_{j-1}\|,
    \end{equation*}
    it must be $(1-\lambda_k)\gamma > 1 - \beta' \lambda_k,$ or, equivalently, $\lambda_k < \frac{\gamma - 1}{\gamma - \beta'}$. Since $\lambda_k = \frac{\beta \epsilon_k/D}{1 + \beta \epsilon_k/D},$ we also have $\epsilon_k < \frac{D}{\beta}\frac{\gamma - 1}{1- \beta'}.$ Thus, the fixed-point error of the solution output by the algorithm is at most $\epsilon_{k-1} = \frac{\epsilon_k}{\beta} < \frac{D}{\beta^2}\frac{\gamma - 1}{1- \beta'} \leq \epsilon_0$. Since $\epsilon_k$ decreases by a factor $\beta$ in each outer while loop iteration, there are at most $\lceil \log_{1/\beta} (\frac{\epsilon_0}{\epsilon_k}) \rceil = \lceil \log_{1/\beta}(\frac{\epsilon_0}{D}\frac{\beta(1-\beta')}{\gamma-1})\rceil$  such iterations. Since each iteration makes at most $\cO(\frac{1}{\beta'\lambda_k}) = \cO(\frac{D}{\beta'\beta^k \epsilon_0})$ calls to \eqref{eq:main-iteration}, the total number of calls to \eqref{eq:main-iteration} (and, consequently, the total number of oracle calls to $\mT$) that \Cref{algo:grad-iterated-Halpern} makes is $\cO(\frac{\beta(1-\beta')}{\beta'(\gamma - 1)})$.
\end{proof}

A few remarks are in order here. First, we can observe that \Cref{algo:grad-iterated-Halpern} simultaneously recovers the optimal complexity bounds for nonexpansive and contractive operators in high dimensions, as long as $\gamma$ is not very close to zero, since in that case the complexity would be $\Theta\big(\frac{\ln(\|\mT(\vx_0) - \vx_0\|/\epsilon)}{\ln(1/\gamma)}\big)$ \cite{sikorski2009computational}. Further, because the step size parameter $\lambda_k$ is adjusted gradually, the algorithm can adapt its convergence speed if the operator happens to be locally contractive (but globally nonexpansive). This will also be shown numerically in \Cref{sec:num-exp}. 

Second, GHAL is guaranteed to converge to an approximate fixed-point of $\mT$ even if $\mT$ is not guaranteed to be non-expansive, but only $\gamma$-Lipschitz, with $\gamma \in (1, 1 + \|\mT(\vx_0) - \vx_0\|/D),$ where the approximation error scales with the diameter of the set $D$ and with $\gamma -1.$ Note here that the assumption that $\gamma < 1 + \|\mT(\vx_0 - \vx_0\|/D$ is to avoid trivialities. If $\gamma \geq 1 + \|\mT(\vx_0) - \vx_0\|/D$, then the initial solution has error $\|\mT(\vx_0) - \vx_0\| \leq (\gamma - 1)D.$ Since the algorithm cannot output a point with fixed-point error higher than the initial one, the error in this case is guaranteed to be at most  $(\gamma - 1)D \leq \ebar.$ The number of oracle queries is no higher than constant, following the same argument as in the proof of \Cref{prop:opt-it-complexity}. 

Finally, we note that for the purpose of the analysis, the algorithm is stated in a way that guarantees finite halting time when a solution with target error is reached. If the algorithm were to, instead, be called for a fixed number of iterations (as will be done in numerical examples in \Cref{sec:num-exp}), there is no need to halt early. Instead, one can simply replace the ``halt'' in Line 9 with a ``break'' for the inner while loop and revert to the previous value of $\lambda_{k-1}$ for the step size. Because Line 9 can only be reached for $\lambda_k < \frac{\gamma - 1}{\beta' - 1},$ we can retain the same approximation error guarantee for the algorithm output. 
%

\subsection{{Adaptive Gradual Halpern Algorithm}}\label{sec:AdaGHAL}

{In this subsection, we introduce the AdaGHAL algorithm (\Cref{algo:adaGHAL-rev}) and present our main results for general $\gamma$-Lipschitz operators. AdaGHAL adaptively estimates the value of $D$ and is thus fully parameter-free. Similar to GHAL, it adapts to the value of $\gamma$ while retaining optimal oracle complexity on nearly the entire range of its possible values. Our main result is stated in \Cref{thm:adaGHAL-Lipschitz}, while the rest of this section is dedicated to proving this theorem.}

\begin{algorithm}
\caption{{Adaptive Gradual Halpern Algorithm (AdaGHAL)}}\label{algo:adaGHAL-rev}
    \begin{algorithmic}[1]
    \Statex \textbf{Input}: $\epsilon > 0, \vx_0,$ $\beta \in (0, 1), \beta' \in (0, 1)$
     \Statex \textbf{Initialization}: $\epsilon_0{, D_0} \leftarrow \|\mT(\vx_0) - \vx_0\|,\, k \leftarrow 0,\, \vxh_0 \leftarrow \vx_0$
     \While{$\|\mT(\vxh_k) - \vxh_k\| > \epsilon$}
     \State $k \leftarrow k + 1$, $D_k = D_{k-1}$
     \State $\epsilon_k \leftarrow \beta \epsilon_{k-1}$
     \State $\lambda_k \leftarrow \frac{\beta  \epsilon_k/D_k}{1 + \beta \epsilon_k/D_k}$ \Comment{{Type~(1) update}}
     \State $\vy_0 \leftarrow \vxh_{k-1}$, $j \leftarrow 0$ \Comment{Initialize \eqref{eq:main-iteration}}
     \While{$\|\mT(\vy_j) - \vy_j\| > \epsilon_k$} \Comment{Run \eqref{eq:main-iteration} until target error is reached}
     \State $\vy_{j+1}\leftarrow \lambda_k \vy_0 + (1-\lambda_k)\mT(\vy_j)$
     \If{$j \geq 1$ and $\|\vy_{j+1} - \vy_j\| \geq (1-\beta'\lambda_k)\|\vy_j - \vy_{j-1}\|$} \Comment{Safeguard}
     \State Halt and return $\vxh_k = \argmin\{\|\mT(\vy_j) - \vy_j\|, \|\mT(\vy_0) - \vy_0\|\}$
     \EndIf
     \If{{$\max\{\norm{\vy_j - \vy_0}, \norm{\mT(\vy_j) - \vy_0}\} \leq D_k$}}
     \State {$j \leftarrow j + 1$}
    \Else 
     \State {$D_k \leftarrow D_k /\beta$}
     {\State $\lambda_k \leftarrow \frac{\beta  \epsilon_k/D_k}{1 + \beta \epsilon_k/D_k}$ \Comment{Type~(2) update}}
     {\State $\vy_0 \leftarrow \argmin\{\|\mT(\vy_{j+1}) - \vy_{j+1}\|, \|\mT(\vy_0) - \vy_0\|\}$, $j \leftarrow 0$}
     \EndIf
     \EndWhile
     \State $\vxh_k \leftarrow \vy_{j}$
     \EndWhile
     \State\Return $\vxh_{k}$
    \end{algorithmic}
\end{algorithm}

\begin{theorem}\label{thm:adaGHAL-Lipschitz}
    {Let $\mT$ be a $\gamma$-Lipschitz operator. Consider \Cref{algo:adaGHAL-rev} applied to $\mT$ for a target $\epsilon > 0,$ initialized at some $\vx_0$. Then all of the following statements apply, depending on the value of $\gamma:$
    \begin{enumerate}
        \item If $\gamma < 1,$ then the algorithm outputs $\vxh_k$ such that $\norm{\mT(\vxh_k) - \vxh_k} \leq \epsilon$ within the number of oracle queries to $\mT$ bounded by
        \[\cO\Big(\log_{1/\beta}\Big(\frac{1}{\gamma}\Big) + \frac{\ln(\beta(1-\beta))}{\ln(\gamma)} + \frac{\ln(\norm{\mT(\vx_0) - \vx_0}/\epsilon)}{\ln(1/\gamma)}\Big(1 + \frac{\ln(1-\beta)}{\ln(\beta)}\Big)\Big).\]
        \item If $\gamma = 1$ and $\vx_*$ is a fixed point of $\mT,$ then the algorithm outputs $\vxh_k$ such that $\norm{\mT(\vxh_k) - \vxh_k} \leq \epsilon$ within the number of oracle queries to $\mT$ bounded by
        \[
            \cO\Big(\frac{\norm{\vx_0 - \vx_*}}{\epsilon}\cdot \frac{-\ln(\beta(1-\beta))}{\beta^4}\Big).
        \]
        \item If $\gamma > 1$, then under \Cref{assp:compact-domain} and $\vx_0 \in \cC$,  \Cref{algo:adaGHAL-rev} outputs $\vxh_k$ such that $\norm{\mT(\vxh_k) - \vxh_k} \leq \min\{\epsilon_{k-1}/\beta, D\}=: \bar{\epsilon},$ where 
        \[\epsilon_{k-1} < (1+\beta^2)\frac{D_k}{\beta^2}\frac{\gamma - 1}{1- \beta'} \leq \frac{D(1+\beta^2)}{\beta^3}\frac{\gamma - 1}{1- \beta'}.\]
        If the algorithm outputs $\vxh_k$ with the fixed-point residual $\norm{\mT(\vxh_k) - \vxh_k} = \hat{\epsilon} \in (\beta \epsilon, \bar{\epsilon}],$ then the total number of oracle queries it utilizes is bounded by 
        \[\cO\Big(\frac{D}{\hat{\epsilon}}\cdot \frac{-\ln(\beta(1-\beta))}{\beta'\beta^4}\Big).\]
    \end{enumerate}
    }
\end{theorem}

A few remarks are in order here. {Suppose that $\beta \in (0, 1)$ is chosen so that both $1/\beta$ and $1/(1-\beta)$ can be treated as universal constants.} First, we can observe that {\Cref{algo:adaGHAL-rev}} simultaneously recovers the optimal complexity bounds for nonexpansive and contractive operators in high dimensions, as long as $\gamma$ is not very close to zero {(namely, as long as $\gamma = \Omega(e^{-\sqrt{\ln(\norm{\mT(\vx_0) - \vx_0}/\epsilon)}})$)}. Further, because the step size parameter $\lambda_k$ is adjusted gradually, the algorithm can adapt its convergence speed if the operator happens to be locally contractive (but globally nonexpansive). %
Second, AdaGHAL is guaranteed to converge to an approximate fixed-point of $\mT$ even if $\gamma > 1,$ where for $\gamma$ close to one, the worst-case approximation error scales with the diameter of the set $D$ and with $\gamma -1,$ matching (up to constants), the guarantee of \eqref{eq:main-iteration} with perfect knowledge of the problem parameters. 
Finally, we note that for the purpose of the analysis, the algorithm is stated in a way that guarantees finite halting time when a solution with target error is reached. If the algorithm were to, instead, be called for a fixed number of iterations, there is no need to halt early. Instead, one can simply replace the ``halt'' in Line 9 with a ``break'' for the inner while loop and revert to a larger value of $\lambda_{k}$ for the step size (e.g., reverting to $\lambda_{k-1}$ or even $c \lambda_k$ for some constant $c > 1$). Because Line 9 can only be reached for $\lambda_k < \frac{\gamma - 1}{\gamma - \beta'},$ we can retain the same approximation error guarantee for the algorithm output. 

\paragraph{{Bounding contributions of Type~(1) and Type~(2) updates}}
{To bound the number of oracle queries to $\mT$ that AdaGHAL makes, it suffices to bound  the total number of inner while loop iterations in \Cref{algo:adaGHAL-rev} (which is the same as the number of \eqref{eq:main-iteration} calls, up to a constant factor). To do so, let us use $(k, j)$ to index the outer loop/inner loop iteration pairs. Let $\cS$ denote the ordered set containing (as an ordered sequence, in order by occurrence) $(k, j)$ in which updates to $\lambda_k$ occur (see Lines 4 and 14), where $(k, 0)$ indexes the Type~(1) update occurring in Line 4, before the inner while loop is entered. We begin by bounding the number of iterations between subsequent elements in $\cS$. For notational convenience, define $\rho_{k} \in (0, 1)$ as: $\rho_k = \gamma(1-\lambda_k)$ if $\gamma \leq 1$ and $\rho_k = 1 - \beta'\lambda_k$ otherwise.}
{
\begin{lemma}\label{lemma:subsequent-type-updates}
    Let $\mT$ be a $\gamma$-Lipschitz operator. Consider \Cref{algo:adaGHAL-rev}, and let the set $\cS$ and contraction factor $\rho_k$ be defined as in the preceding text. Let $(k, j)$ be any elements in $\cS.$ Then, the total number of inner while loop iterations between  $(k, j)$ and either the subsequent update to $\lambda_k$ or the algorithm halting is at most $J = \big\lceil \frac{\ln(\beta(1-\beta))}{\ln(\rho_{k})}\big\rceil.$
\end{lemma}}
\begin{proof}
    {First, observe that because the safeguard halts the algorithm, it cannot occur between elements in $\cS$. Suppose first that $j = 0$ (i.e., $(k, j)$ indexes a Type~(1) update). Consider the while loop iterations $j$ until  either the next update to $\lambda_k$ (in $(k, j)$) occurs, or the algorithm halts. In all these iterations, there is no update to either $\lambda_k$ or $D_k.$   
    For each while loop iteration $j,$ using \Cref{lemma:fixed-step-convergence} and the safeguard condition, we have $\norm{\vy_{j+1} - \vy_j} \leq (1-\lambda_{k})\rho_k^j \norm{\mT(\vy_0) - \vy_0}$, and, as a consequence, by the same argument as in \Cref{lemma:fixed-step-convergence}: $\norm{\mT(\vy_j) - \vy_j} \leq \rho_k^j \norm{\mT(\vy_0) - \vy_0} + \frac{\lambda_k}{1-\lambda_k}\norm{\vy_j - \vy_0}.$ By the inner loop condition, we have $\norm{\mT(\vy_0) - \vy_0}  \leq \epsilon_{k-1} = \epsilon_k/\beta$. Since there was no Type~(2) update up to iteration $j,$  it must be $\norm{\vy_j - \vy_0} \leq D_k$ and so $\frac{\lambda_k}{1-\lambda_k}\norm{\vy_j - \vy_0} \leq \beta \epsilon_k.$ By the inner while loop exit condition and the definition of $\vy_0$, we have $\norm{\mT(\vy_0) - \vy_0} \leq \epsilon_{k-1} = \epsilon_k/\beta.$ Hence:
    \[
        \norm{\mT(\vy_j) - \vy_j} \leq \epsilon_k ((1/\beta)\rho^j  + \beta).
    \]
    If $j \geq J$, then $\rho^j \leq \beta(1-\beta),$ implying $\norm{\mT(\vy_j) - \vy_j} \leq \epsilon_k,$ which triggers the inner while loop exit condition upon which either a Type~(1) update must occur in the outer while loop or the outer while loop (and thus the algorithm) halts. Thus, either an update to $\lambda_k$ or algorithm halting must occur within $J$ iterations.}

    {Now suppose $j > 0$, in which case $(k, j)$ indexes a Type~(2) update. Observe that on a Type~(2) update,  \eqref{eq:main-iteration} is effectively restarted (Line~15) ensuring that the new fixed-point residual satisfies $\norm{\mT(\vy_j) - \vy_j} \leq \norm{\mT(\vxh_{k-1}) - \vxh_{k-1}} \leq \epsilon_{k-1}.$ Now the same argument as above applies: the next update to $\lambda_k$ (either Type~(1) or Type~(2)) or algorithm halting must occur within $J$ iterations, since otherwise the inner while loop exit condition would be reached, leading to the Type~(1) update or a halt. }
\end{proof}

{We now bound the total number of Type~(1) and Type~(2) updates, as follows.}
{
\begin{lemma}\label{lemma:Type1Type2-updates-count}
    Given a $\gamma$-Lipschitz operator $\mT: \cC \to \cC$, for $\cC$   a bounded closed convex set of diameter $D$, and given $\vx_0 \in \cC$, \Cref{algo:adaGHAL-rev} makes at most: $I_1 = \big\lceil \log_{\frac{1}{\beta}}\frac{\norm{\mT(\vx_0) - \vx_0}}{\beta \epsilon} \big\rceil$ Type~(1) updates and $I_2 = \big\lceil \log_{\frac{1}{\beta}}\frac{D}{\beta \norm{\mT(\vx_0) - \vx_0}} \big\rceil$ Type~(2) updates.
\end{lemma}}
\begin{proof}
   {Every Type~(1) update coincides with a decrease in $\epsilon_k$ by a factor $\beta \in (0, 1).$ Since $\epsilon_k$ is never increased, it is initialized at $\epsilon_0 = \norm{\mT(\vx_0) - \vx_0}$, and the algorithm halts for $\epsilon_k \leq \epsilon,$ the total number of times it can get decreased by a factor $\beta$ is at most $\lceil \log_{1/\beta}(\frac{\norm{\mT(\vx_0) - \vx_0}}{\beta \epsilon}) \rceil,$ which is precisely the claimed $I_1$.}

    {Consider now all the Type~(2) updates. First, observe that since the algorithm only updates its iterates ($\vy_j$ or $\vxh_k$) as convex combinations of prior iterates and their mappings via $\mT(\cdot),$ by the lemma assumptions, all iterates must belong to $\cC.$  Each Type~(2) update increases the estimate $D_k$ by a factor $1/\beta > 1.$ Since $D_k$ is initialized at $D_0 = \norm{\mT(\vx_0) - \vx_0},$ it is never decreased, and $D_k \leq D/\beta$ (because all iterates remain in $\cC$), the total number of times that $D_k$ can be increased by a factor $(1/\beta)$ is at most $\lceil \log_{1/\beta}(\frac{D}{\beta \norm{\mT(\vx_0) - \vx_0}}) \rceil$. Since each Type~(2) update coincides with an update to $D_k$, the same bound applies to the maximum number of Type~(2) updates. }
\end{proof}

\paragraph{{Remaining Proof of \Cref{thm:adaGHAL-Lipschitz}}}
{
We are now ready to prove our main theorem for general $\gamma$-Lipschitz operators.}
\begin{proof}[{Proof of \Cref{thm:adaGHAL-Lipschitz}}]
   {Let us first quickly argue that \Cref{lemma:Type1Type2-updates-count} applies under either regime of $\gamma$, that is, that $\mT: \cC \to \cC$, where $\cC$  is a bounded closed convex set of diameter $D < \infty$, and that $\vx_0 \in \cC$. For $\gamma > 1,$ this is immediate, due to \Cref{assp:compact-domain} and the theorem assumption. When $\gamma \leq 1,$ this claim holds for $\cC = \{\vx: \norm{\vx - \vx_*} \leq \norm{\vx_0 - \vx_*}\},$ as $\vx_0 \in \cC$ and any new algorithm iterate $\vx^+$ is constructed as $\vx^+ = \lambda \vx + (1-\lambda)\mT(\vy)$ where both $\vx, \vy$ are prior algorithm iterates and $\lambda \in (0, 1)$. Thus, $\norm{\vx^+ - \vx_*} \leq \lambda\norm{\vx - \vx_*} + (1-\lambda)\norm{\mT(\vy) - \mT(\vx_*)} \leq \lambda\norm{\vx - \vx_*} + (1-\lambda)\norm{\vy - \vx_*},$ where we used the triangle inequality, $\mT(\vx_*) = \vx_*$, and $\gamma \leq 1.$ Hence it inductively follows that all iterates remain in $\cC.$ Observe that in this case $D = 2\norm{\vx_0 - \vx_*}.$ Further, $D_k \leq D/\beta$, for any algorithm iterate $k \geq 0.$}

    {Suppose first that the algorithm never enters the safeguard update. Certainly, for $\gamma \leq 1,$ this is always true, as for inner iteration $j \geq 1$ and outer iteration $k \geq 1$, $\norm{\vy_{j+1} - \vy_j} = (1 - \lambda_k)\norm{\mT(\vy_j) - \mT(\vy_{j-1})} \leq (1-\lambda_k) \norm{\vy_j - \vy_{j-1}}.$ Then if the algorithm halts, it must return $\vxh_k$ with $\norm{\mT(\vxh_k) - \vxh_k} \leq \epsilon$, so all we need to argue is that this happens within the claimed number of iterations/oracle queries to $\mT.$}

    {Observe that the first update to $\lambda_k$ occurs in the outer loop (Type~(1) update), before the inner while loop is entered for the first time. Partition the set $\cS$ into ordered sets $\cS_1 := \{(k, j) \in \cS: j = 0\}$ and $\cS_2 = \{(k, j) \in \cS: j \geq 1\},$ respecting the order in $\cS.$  That is, $\cS_1$ contains all the $(k, j)$ indexing Type~(1) updates, while $\cS_2$ contains all the $(k, j)$ indexing Type~(2) updates. Consider first the elements of $\cS_1.$ By \Cref{lemma:Type1Type2-updates-count}, $|\cS_1| \leq \big\lceil \log_{1/\beta}(\frac{\norm{\mT(\vx_0) - \vx_0}}{\beta \epsilon}) \big\rceil$. Let $(k, 0)$ be the $i^{\mathrm{th}}$ element of $\cS_1$. Because $\epsilon_k$ is only updated when a Type~(1) update to $\lambda_k$ occurs, we have that $\epsilon_k = \epsilon_0 \beta^i$ and so $\lambda_k = \frac{(\epsilon_0/D_k) \beta^{i+1}}{1 + (\epsilon_0/D_k) \beta^{i+1}}.$ By \Cref{lemma:subsequent-type-updates}, the total number of iterations until either a next update to $\lambda_k$ or a halting condition is reached is at most $\lceil \frac{\ln(\beta(1-\beta))}{\ln(\rho_k)} \rceil,$ where $\rho_k = \gamma(1-\lambda_k)$ if $\gamma \leq 1$ and $\rho_k  = 1-\beta' \lambda_k$ if $\gamma >1$. Call these iterations the $i^{\mathrm{th}}$ Type~(1) interval, and let $J_1$ denote the sum of the lengths of all Type~(1) intervals. We now bound $J_1$ case-by-case, depending on the value of $\gamma.$}

{
 \noindent\textbf{Case 1: $\gamma <1$.} In this case, $\rho \leq \gamma,$ and we have
    \begin{equation}\label{eq:J1-contractive-gen}
        J_1 \leq \Big\lceil\frac{\ln(\beta(1-\beta))}{\ln(\gamma)}\Big\rceil I_1 \leq \Big(\frac{\ln(\beta(1-\beta))}{\ln(\gamma)} + 1\Big)\Big(\frac{\ln(\norm{\mT(\vx_0) - \vx_0}/\epsilon)}{\ln(1/\beta)} + 1\Big).
    \end{equation}
If $\gamma \geq \beta(1-\beta),$ then $\frac{\ln(\beta(1-\beta))}{\ln(\gamma)} \geq 1,$ and so 
\begin{equation}\label{eq:J1-contractive-large-gamma}
J_1 = \cO\Big(\frac{\ln(\beta(1-\beta))}{\ln(\gamma)} + \frac{\ln(\norm{\mT(\vx_0) - \vx_0}/\epsilon)}{\ln(1/\gamma)}\Big(1 + \frac{\ln(1-\beta)}{\ln(\beta)}\Big)\Big).
\end{equation}
Observe that for $\beta \in (0, 1)$, chosen so that both $1/\beta$ and $1/(1-\beta)$ can be treated as universal constants, independent of any problem parameters, the above bound simplifies to $J_1 = \cO\big(\frac{\ln(\norm{\mT(\vx_0) - \vx_0}/\epsilon)}{\ln(1/\gamma)}\big).$}

{
Now suppose $\gamma < \beta(1-\beta).$ Let $\vy_0 = \vxh_{k-1},$ $\vy_1 = \lambda_k \vy_0 + (1 - \lambda_k)\mT(\vy_0).$ Then, rearranging the $\vy_1$ update equation, we get that
    \begin{align}\label{eq:first-it-rearranged-ada}
        \vy_1 - \vy_0 = (1 - \lambda_k)(\mT(\vy_0) - \vy_0) & \Rightarrow \|\vy_1 - \vy_0\| \leq (1 - \lambda_k)\|\mT(\vy_0) - \vy_0\|,
    \end{align}
    and, further,  
    \(
        (1 - \lambda_k)(\vy_1 - \mT(\vy_1)) = \lambda_k(\vy_0 - \vy_1) + (1-\lambda_k)(\mT(\vy_0) - \mT(\vy_1)),
    \)
    which, using the triangle inequality, $\gamma$-contractivity of $\mT$, and \eqref{eq:first-it-rearranged-ada}, implies
    \begin{align}\notag
        \|\vy_1 - \mT(\vy_1)\| &\leq \Big(\frac{\lambda_k}{1 - \lambda_k} + \gamma\Big)\|\vy_1 - \vy_0\| 
        \leq \big(\lambda_k + (1 - \lambda_k)\gamma\big) \|\mT(\vy_0) - \vy_0\|.
    \end{align}
    Now, since $\|\mT(\vy_0) - \vy_0\| \leq \epsilon_{k-1}$ and calls to \eqref{eq:main-iteration} in the inner while loop terminate as soon as $\|\vy_j - \mT(\vy_j)\| \leq \epsilon_k = \beta\epsilon_{k-1},$ a sufficient condition for the while loop to terminate in one iteration is that $\lambda_k + (1 - \lambda_k)\gamma \leq \beta.$ Since $\lambda_k \leq \lambda_1 \leq \beta^2$ and $\gamma < \beta(1-\beta),$ this condition is always satisfied, and so the inner while loop always halts within one iteration. Further, if $\epsilon_k \leq \gamma \epsilon_0$ then also  $\lambda_k \leq \beta \gamma.$ This happens after at most $i = \lceil \log_{1/\beta}(1/\gamma)\rceil$ Type~(1) updates, at which point the effective contraction factor is $\lambda_k + (1 - \lambda_k)\gamma \leq (1+\beta)\gamma.$ Thus, the total number of iterations $J_1$ %
    is
    \begin{equation}\label{eq:J1-contractive-small-gamma}
        J_1 
        = \cO\Big(\log_{1/\beta}(1/\gamma) + \frac{\ln(\gamma \norm{\mT(\vx_0) - \vx_0}/\epsilon)}{\ln(1/\gamma)}\Big). 
    \end{equation}
To finish bounding $J_1$ for $\gamma < 1,$ it remains to combine the bounds for ``large'' and ``small'' $\gamma$, namely, \eqref{eq:J1-contractive-large-gamma} and \eqref{eq:J1-contractive-small-gamma}, and simplify. This leads to:
\begin{equation}\label{eq:J1-final-gamma-bound}
    J_1 = \cO\Big(\log_{1/\beta}\Big(\frac{1}{\gamma}\Big) + \frac{\ln(\beta(1-\beta))}{\ln(\gamma)} + \frac{\ln(\norm{\mT(\vx_0) - \vx_0}/\epsilon)}{\ln(1/\gamma)}\Big(1 + \frac{\ln(1-\beta)}{\ln(\beta)}\Big)\Big).
\end{equation}
}
{\noindent\textbf{Case 2: $\gamma =1$.} If $\gamma = 1,$ then $\rho_k = 1-\lambda_k = \frac{1}{1 + (\epsilon_0/D_k)\beta^{i+1}}\leq \frac{1}{1 + (\epsilon_0/D)\beta^{i+2}},$ and so
    \begin{equation}\label{eq:J1-nonexpansive}
    \begin{aligned}
        J_1 &\leq \sum_{i=1}^{I_1} \Big\lceil\frac{-\ln(\beta(1-\beta))}{\ln(1 + (\epsilon_0/D)\beta^{i+2})}\Big\rceil = \cO\Big(\frac{-\ln(\beta(1-\beta))}{\beta^3}\frac{D}{\epsilon_0}\frac{\norm{\mT(\vx_0) - \vx_0}}{\beta \epsilon}\Big)\\
        &= \cO\Big(\frac{\norm{\vx_0 - \vx_*}}{\epsilon}\cdot \frac{-\ln(\beta(1-\beta))}{\beta^4}\Big). 
    \end{aligned}
    \end{equation}
 \noindent\textbf{Case 3: $\gamma > 1$.} If $\gamma >1$, then \[\rho_k = 1 - \beta' \lambda_k \leq \frac{1}{1 + \beta'(\epsilon_0/D_k)\beta^{i+1}/(1+\beta)}\leq \frac{1}{1 + \beta'(\epsilon_0/D)\beta^{i+2}/(1+\beta)},\] and so using the same derivation as for $\gamma = 1$ leads to
    \begin{equation}\label{eq:J1-mildly-expansive}
    \begin{aligned}
        J_1 = \cO\Big(\frac{D}{\epsilon}\cdot \frac{-\ln(\beta(1-\beta))}{\beta'\beta^4}\Big). 
    \end{aligned}
    \end{equation}
}
 
    {Our next step is to bound the cumulative number of iterations occurring in intervals beginning with a Type~(2) update end ending with either a subsequent update to $\lambda_k$ or the algorithm halt, whichever happens first. Call those intervals Type~(2) intervals and let $J_2$ denote the cumulative number of iterations within the union of those intervals. Fix first some $(k, j) \in \cS_2,$ and let it be the $i^{\mathrm{th}}$ element of $\cS_2.$ As the updates to $D_k$ only happen on Type~(2) updates, we have that immediately upon the update, $D_k = D_0/\beta^{i} = \epsilon_0/\beta^i.$ Further, since for any $k,$ $\epsilon_k \geq \beta \epsilon$ (otherwise the algorithm would have halted), we have \(\lambda_k = \frac{\beta \epsilon_k/D_k}{1 + \beta \epsilon_k/D_k} \geq \frac{\beta^{i+2}\epsilon/\epsilon_0}{1 + \beta^{i+2}\epsilon/\epsilon_0}\). We now similarly use \Cref{lemma:subsequent-type-updates,lemma:Type1Type2-updates-count} to obtain the bound on $J_2$, considering the three cases for $\gamma.$}

    \noindent{\textbf{Case 1: $\gamma < 1$.} In this case, $\rho_k < \gamma,$ and so the length of the $i^{\mathrm{th}}$ Type~(2) interval is at most $\lceil \frac{\ln(\beta(1-\beta))}{\ln(\gamma)} \rceil.$ Further, by the triangle inequality and $\gamma$-contractivity, we have $D_0 = \norm{\mT(\vx_0) - \vx_0} \geq (1- \gamma)\norm{\vx_0 - \vx_*} \geq \frac{1-\gamma}{2}D.$ Thus, by \Cref{lemma:Type1Type2-updates-count}, the total number of Type~(2) intervals is $\cO(\log_{1/\beta}(\frac{2}{1-\gamma})),$ and so
    \begin{equation}\label{eq:J2-gamma<1}
        J_2 = \cO\Big(\log_{1/\beta}\Big(\frac{2}{1-\gamma}\Big)\Big(1 + \frac{\ln(\beta(1-\beta))}{\ln(\gamma)}\Big)\Big).
    \end{equation}
    \noindent\textbf{Case 2: $\gamma = 1$.} In this case, $\rho_k = 1- \lambda_k \leq \frac{1}{1 + \beta^{i+2}\epsilon/\epsilon_0}$, and the length of the $i^{\mathrm{th}}$ Type~(2) interval is $\cO(\frac{-\ln(\beta(1-\beta))}{\beta^{i+2}\epsilon/\epsilon_0}).$ By \Cref{lemma:Type1Type2-updates-count}, the total number of Type~(2) intervals is at most $I_2 = \big\lceil \log_{1/\beta}(\frac{D}{\beta \norm{\mT(\vx_0) - \vx_0}}) \big\rceil$. Thus:
    \begin{equation}\label{eq:J2-gamma=1}
        \begin{aligned}
            J_2 &= \sum_{i=1}^{I_2} \cO\Big(\frac{-\ln(\beta(1-\beta))}{\beta^{i+2}\epsilon/\epsilon_0}\Big) = \cO\Big(\frac{\norm{\vx_0 - \vx_*}}{\epsilon}\cdot \frac{-\ln(\beta(1-\beta))}{\beta^4}\Big).
        \end{aligned}
    \end{equation}
    \noindent\textbf{Case 3: $\gamma > 1$.} In this case, $\rho_k = 1 - \beta' \lambda_k \leq \frac{1}{1 + \beta'(\epsilon_0/D)\beta^{i+2}/(1+\beta)},$ and so the same derivation as for Case 2 leads to:
    \begin{equation}\label{eq:J2-gamma>1}
    \begin{aligned}
        J_2 = \cO\Big(\frac{D}{\epsilon}\cdot \frac{-\ln(\beta(1-\beta))}{\beta'\beta^4}\Big). 
    \end{aligned}
    \end{equation}
    }

    {To complete the proof, it remains to consider the case in which the algorithm halts due to the safeguard condition, which, as discussed at the beginning, can only happen for $\gamma > 1.$ Suppose that this happens in outer loop iteration $k$. Observe that the algorithm can never return a solution with error worse than $\norm{\mT(\vx_0) - \vx_0} \leq D,$ as the algorithm never increases the fixed-point residual of candidate output points $\vxh_k.$ The solution output by the algorithm, by the condition in Line 9, has fixed-point residual no worse than $\|\mT(\vxh_{k-1}) - \vxh_{k-1}\| \leq \epsilon_{k-1}.$ Because, by \eqref{eq:main-iteration}, we have for $j \geq 1$ that
    \(
        \|\vy_{j+1} - \vy_j\| = (1-\lambda_k)\|\mT(\vy_j)- \mT(\vy_{j-1})\| \leq (1-\lambda_k)\gamma \|\vy_j - \vy_{j-1}\|,
    \)
    it must be $(1-\lambda_k)\gamma > 1 - \beta' \lambda_k,$ or, equivalently, $\lambda_k < \frac{\gamma - 1}{\gamma - \beta'}$. Since $\lambda_k = \frac{\beta \epsilon_k/D_k}{1 + \beta \epsilon_k/D_k}$ and $1 + \beta \epsilon_k/D_k \leq 1 + \beta^{k+1}\epsilon_0/D_0 \leq 1 + \beta^2$, we also have $\epsilon_k < (1+\beta^2)\frac{D}{\beta}\frac{\gamma - 1}{1- \beta'}.$ Thus, the fixed-point residual of the solution output by the algorithm is at most $\epsilon_{k-1} = \frac{\epsilon_k}{\beta} < (1+\beta^2)\frac{D_k}{\beta^2}\frac{\gamma - 1}{1- \beta'} \leq \frac{D(1+\beta^2)}{\beta^3}\frac{\gamma - 1}{1- \beta'}$, since $D_k \leq D /\beta$. Thus, we have argued that the output solution has fixed-point residual no larger than $\ebar.$}

    {Finally, it remains to observe that, if the safeguard condition is triggered at the output fixed-point residual $\hat{\epsilon} \in (\beta\epsilon, \ebar]$, then by the earlier argument in this proof, the total number of iterations/oracle queries to $\mT$ until this can happen is at most $\cO\big(\frac{D}{\hat{\epsilon}}\cdot \frac{-\ln(\beta(1-\beta))}{\beta'\beta^4}\big),$ which concludes the proof. }
\end{proof}
%


\subsection{Comparison to Restarting-Based Approaches}\label{sec:restart}

We briefly discuss here how our results compare to restarting-based strategies as in prior work \cite{diakonikolas2020halpern} that sought to obtain an algorithm that simultaneously addresses cases where $\mT$ can be nonexpansive or contractive, adapting to the better between the sublinear convergence rate (of Halpern iteration) and linear convergence rate (of Halpern iteration). While the results in \cite{diakonikolas2020halpern} apply only to Euclidean spaces, the underlying restarting idea can be applied more broadly to the classical Halpern iteration and any related algorithm that leads to results of the form\footnote{Here we do not explicitly discuss the Krasnosel'skii-Mann iteration, as its (tight) convergence rate is $1/\sqrt{k}$ \cite{baillion1996rate,cominetti2014rate,diakonikolas2021potential}, which is slower than the rate $1/k$ of Halpern iteration and leads to worse complexity results.}
\begin{equation}\label{eq:generic-halpern-bnd}
    \|\mT(\vx_k) - \vx_k\| = \cO\Big(\frac{\|\vx_0 - \vx_*\|}{k}\Big),
\end{equation}
after running the algorithm for $k \geq 1$ iterations, where the Lipschitz constant $\gamma$ of $\mT$ satisfies $\gamma \leq 1$ and $\vx_*$ is a fixed point of $\mT$ (assuming it exists). 

When $\gamma < 1,$ since $\mT(\vx_*) = \vx_*$, we also have, by the triangle inequality, that $\|\vx - \vx_*\| \leq \|\vx - \mT(\vx)\| + \|\mT(\vx) - \mT(\vx_*)\| \leq \|\vx - \mT(\vx)\| + \gamma\|\vx - \vx_*\|$ and so $\|\vx - \vx_*\| \leq \frac{\|\vx - \mT(\vx)\|}{1-\gamma}$. As a consequence, the bound in \eqref{eq:generic-halpern-bnd} implies $ \|\mT(\vx_k) - \vx_k\| = \cO\Big(\frac{\|\vx_0 - \mT(\vx_0)\|}{(1-\gamma)k}\Big)$, and so the fixed-point error $\|\mT(\vx_k) - \vx_k\|$ can be halved within $k = \cO(1/(1-\gamma))$ iterations. Thus, by restarting the algorithm each time the fixed-point error halves, we get an algorithm that converges to a fixed-point error $\epsilon > 0$ within $\cO\big(\frac{\ln(\|\mT(\vx_0 - \vx_0\|/\epsilon)}{1-\gamma}\big)$ iterations. When $\gamma$ is close to one, $\ln(1/\gamma) \approx 1 - \gamma$ and this bound matches the optimal iteration complexity $\cO\big(\frac{\ln(\|\mT(\vx_0 - \vx_0\|/\epsilon)}{\ln(1/\gamma)}\big)$ of Picard iteration. However, when $\gamma$ is small, this bound is suboptimal, as $\ln(1/\gamma)$ can be much larger than $1 - \gamma.$ Observe that the oracle complexity of GHAL and AdaGHAL in this case is $\cO\big(\ln(1/\gamma) + \frac{\ln(\gamma\|\mT(\vx_0 - \vx_0\|/\epsilon)}{\ln(1/\gamma)}\big)$, which is lower than $\cO\big(\frac{\ln(\|\mT(\vx_0 - \vx_0\|/\epsilon)}{1-\gamma}\big)$ whenever $\gamma$ is small and closer to the optimal iteration complexity of Picard iteration, matching it up to constant factors whenever $\gamma$ is not trivially small. This complexity advantage is clearly observed in the numerical examples in \Cref{sec:num-exp}.

We remark here that (as also observed in \cite{diakonikolas2020halpern}) the oracle complexity bound of $\cO\big(\frac{\ln(\|\mT(\vx_0 - \vx_0\|/\epsilon)}{\mu}\big)$ for restart-based strategies applies even for operators that are not contractive, but instead are nonexpansive and satisfy a local error bound of the form $\|\mT(\vx) - \vx\| \geq \mu\, \dist(\vx, \cX_*)$, where $\cX_*$ denotes the set of fixed points of $\mT$ and $\dist(\vx, \cX_*) = \inf_{\vy \in \cX_*}\|\vx - \vy\|$. In this case, it can be argued that restarted \eqref{eq:main-iteration} with an appropriate choice of the step size $\lambda$ converges linearly. 
\begin{lemma}\label{lemma:LEB-convergence}
    Suppose that $\mT$ is a nonexpansive operator that satisfies $\|\mT(\vx) - \vx\| \geq \mu\, \dist(\vx, \cX_*)$, where $\cX_*$ denotes the set of fixed points of $\mT$ (assuming $\cX_*$ is non-empty). Let $\vx_k$ denote the $k^{\rm th}$ iterate of \eqref{eq:main-iteration} applied to $\mT,$ initialized at some $\vx_0$ and using some step size $\lambda \in (0, 1)$. Let $\beta \in (0, 1).$ If $\lambda = \frac{\mu\beta/4}{1 + \mu\beta/4},$ then $\|\mT(\vx_k) - \vx_k\| \leq \beta\|\mT(\vx_0) - \vx_0\|$ after at most $k = \lceil \frac{\ln(2/\beta)}{\ln(1/(1-\lambda))} \rceil$ iterations. As a result, choosing $\beta = 1/2$ and restarting \eqref{eq:main-iteration} each time the fixed-point error is halved, we have that $\|\mT(\vx_k) - \vx_k\| \leq \epsilon$ using a total of $\cO\big(\frac{\ln(\|\mT(\vx_0) - \vx_0\|/\epsilon)}{\mu}\big)$ oracle queries to $\mT$ (equivalently, iterations of \eqref{eq:main-iteration}).
\end{lemma}
\begin{proof}
If $\|\mT(\vx_0) - \vx_0\| = 0,$ there is nothing to prove, so assume this is not the case. From \Cref{lemma:fixed-step-convergence}, we have that for all $k \geq 1$ and all $\vx_* \in \cX_*,$ $\|\mT(\vx_k) - \vx_k\| \leq (1-\lambda)^k\|\mT(\vx_0) - \vx_0\| + \frac{2\lambda}{1-\lambda}\|\vx_0 - \vx_*\|$. 
    Thus, 
    \begin{align*}
        \|\mT(\vx_k) - \vx_k\| &\leq (1-\lambda)^k\|\mT(\vx_0) - \vx_0\| + \frac{2\lambda}{1-\lambda}\dist(\vx_0, \cX_*)\\
        &\leq (1-\lambda)^k\|\mT(\vx_0) - \vx_0\| + \frac{2\lambda}{1-\lambda}\frac{1}{\mu}\|\mT(\vx_0) - \vx_0\|, 
    \end{align*}
    where the last inequality is by the lemma assumption. The first lemma claim now follows by solving the inequalities $(1-\lambda)^k \leq \beta/2$ and $\frac{2\lambda}{1-\lambda}\frac{1}{\mu}\|\mT(\vx_0) - \vx_0\| \leq (\beta/2)\|\mT(\vx_0) - \vx_0\|$ for $k$ and $\lambda,$ respectively. 

    Now consider restarting \eqref{eq:main-iteration} each time the fixed-point error of the iterate $\vx_k$ gets halved (i.e., whenever $\norm{\mT(\vx_k) - \vx_k} \leq (1/2)\norm{\mT(\vx_0) - \vx_0}$). Here, restarting means that the new calls to the iteration \eqref{eq:main-iteration} are initialized at the output $\vx_k$ that triggered the restart. Using the first part of the lemma, a restart is triggered every $k = \lceil \frac{\ln(2/\beta)}{\ln(1/(1-\lambda))} \rceil = \cO\big(\frac{1}{\mu}\big)$ iterations, in which case the fixed-point error is halved. Hence, there can be at most $\lceil\log_2(\norm{ 2\mT(\vx_0) - \vx_0}/\epsilon)\rceil$ restarts in total, proving the second claim. 
\end{proof}
We remark here that the result of \Cref{lemma:LEB-convergence} does not directly translate into linear convergence for GHAL and AdaGHAL under a local error  bound, as both algorithms progressively decrease the step size $\lambda$ as the fixed-point error of the iterates is decreased. While it is possible to change the while-loop conditions of these two algorithms to prevent reducing the step size once it is sufficiently small, this would, however, come at a cost of higher $\cO\big(\frac{\ln(\|\mT(\vx_0 - \vx_0\|/\epsilon)}{1-\gamma}\big)$ oracle complexity for contractive operators with a small value of $\gamma.$ For this reason, we do not consider such algorithm modifications and defer investigation of convergence under local error bounds to future work.  

\section{Gradually Expansive Operators and Convergence of AdaGHAL}\label{sec:grad-expansive}

Even though improving the results under mild operator expansion seems unlikely, we argue below that by using {\Cref{algo:adaGHAL-rev}}, we can ensure convergence to a fixed point for a class of operators with \emph{gradual expansion} (defined below), whose bound on the Lipschitz constant is dependent on the fixed-point residual (thus the bound is required to \emph{gradually} decrease with the fixed-point residual). The basic idea is that in {\Cref{algo:adaGHAL-rev}}, the inner while loop can ensure a reduction in error within a finite number of iterations as long as $(1-\lambda_k)\norm{\mT(\vy_{j+1}) - \mT(\vy_j)} < \norm{\vy_{j+1} - \vy_j}.$ The definition of gradually expansive operators was obtained by ensuring such a condition.

\begin{definition}[$\alpha$-Gradually Expansive Operators]\label{def:gradually-expansive-op}
    Given an operator $\mT$ that satisfies \Cref{assp:compact-domain}, we say that $\mT$ is $\alpha$-gradually expansive for some $\alpha \in [0, 1)$, if the following holds: for all $\vx, \vy \in \cC$ and all $\epsilon > 0$,
    \begin{equation}\notag
       \max\{\|\mT(\vx) - \vx\|, \|\mT(\vy) - \vy\|\} \leq \epsilon \;\Rightarrow \; \|\mT(\vx) - \mT(\vy)\| \leq \Big(1 + \frac{\alpha \epsilon}{D}\Big)\|\vx - \vy\|.
    \end{equation}
\end{definition}

It is immediate that the above definition generalizes nonexpansive operators, as the case $\alpha = 0$ corresponds to nonexpansivity. However, for $\alpha > 0,$ the operator need not be nonexpansive, with allowed expansion depending on the fixed-point residual relative to the diameter of the set $\cC.$

Our main result for this setting is summarized in the following theorem.

\begin{theorem}\label{thm:GHAL-grad-expansive}
    Let $\mT:\cC \to \cC$ be an $\alpha$-gradually expansive operator for $\alpha \in (0, \sqrt{2}-1)$, where the diameter of $\cC$ is bounded by $D < \infty.$ Given $\epsilon > 0,$ if {\Cref{algo:adaGHAL-rev}} is invoked with parameters $\beta, \beta' \in (0, 1)$ such that $\beta^3 >  \alpha(1+\alpha + \beta^2)$ and $\beta' \in {(0, \frac{c}{1+c}]}$, where $c > 0$ is the constant such that  $\beta^3 = (1+c)\alpha(1+\alpha+\beta^2)$,  then it outputs $\vxh_k \in \cC$ such that $\|\mT(\vxh_k) - \vxh_k \| \leq \epsilon$ within the number of oracle queries    
    \[{\cO\Big( \frac{-\ln(\beta(1-\beta))}{\beta^4} \min\Big\{ \frac{D}{\beta' \epsilon}, \frac{D^2}{\beta^4 \epsilon^2}\Big\} \Big)}.\]  
\end{theorem}
\begin{proof}
We first note that the parameter $\beta'$ was chosen so that the safeguard condition in {\Cref{algo:adaGHAL-rev}} never gets triggered, as will be argued below by proving that for any fixed outer iteration $k,$ all iterates $\vy_{j}$ satisfy $\|\vy_{j+1} - \vy_j\| < (1 - \beta' \lambda_k)\|\vy_j - \vy_{j-1}\|$ for $j \geq 1$ (see \eqref{eq:grad-contraction-final} in the proof below). Then, all we need to do is bound above the contraction factor $\rho_k \leq 1 - \beta' \lambda_k,$ since we can invoke the same argument that we used in bounding the oracle complexity in \Cref{thm:adaGHAL-Lipschitz}.  
{Observe first that we can assume w.l.o.g.\ that $\max\{\norm{\vy_j - \vy_0}, \norm{\mT(\vy_j) - \vy_0}\} \leq D_k,$ since otherwise a Type (2) update would be triggered, which cannot be followed by a safeguard, since $j$ is reset to zero.}
    
    We argue first by induction on $j$ that for all $j \geq 0,$ $\|\mT(\vy_j) - \vy_j\| \leq (1-\lambda_k)(1 + \alpha + \beta^2) \epsilon_{k-1}.$ We then use this bound to conclude that by \Cref{def:gradually-expansive-op}, $\|\mT(\vy_{j+1}) - \mT(\vy_j)\| \leq (1 + \frac{(1-\lambda_k)\alpha (1 + \alpha + \beta^2) \epsilon_{k-1}}{D} )\|\vy_{j+1} - \vy_j\|.$ %
    For the base case, we consider $j = \{0, 1\}.$ Because $\vy_0 = \vxh_{k-1},$ we have from {\Cref{algo:adaGHAL-rev}} that $\|\mT(\vy_0) - \vy_0\| \leq \epsilon_{k-1}.$ By the definition of $\vy_1 := \lambda_k \vy_0 + (1 - \lambda_k) \mT(\vy_0)$ we then also have $\|\vy_1 - \vy_0\| = (1 - \lambda_k)\|\mT(\vy_0) - \vy_0\| \leq (1 - \lambda_k)\epsilon_{k-1}.$ To finish the base case, it remains to bound $\|\mT(\vy_1) - \vy_1\|.$ By the definition of $\vy_1,$ subtracting $\mT(\vy_1)$ and using the triangle inequality, %
    \begin{equation}\label{eq:init-fixed-point}
    \begin{aligned}
        \|\mT(\vy_1) - \vy_1\| &\leq \lambda_k\|\vy_0 - \mT(\vy_1)\| + (1 - \lambda_k)\|\mT(\vy_1) - \mT(\vy_0)\|\\ %
        &\leq (1 - \lambda_k)\beta^2  \epsilon_{k-1} + (1 - \lambda_k)\|\mT(\vy_1) - \mT(\vy_0)\|,
    \end{aligned}
    \end{equation}
    where we have used the definition of $\lambda_k,$ by which $\frac{\lambda_k}{1 - \lambda_k} = \beta^2 \epsilon_{k-1}/D_k$ and  that $\|\mT(\vy_1) - \vy_0\| \leq D_k$.     
    Observe that for any $\vy \in \cC,$ $\|\mT(\vy) - \vy\| \leq D,$ and thus for any $\vx, \vy \in \cC,$ by \Cref{def:gradually-expansive-op}, we have $\|\mT(\vx) - \mT(\vy)\| \leq (1 + \alpha)\|\vx - \vy\|.$ As a consequence, applying \Cref{def:gradually-expansive-op}, we have
    \begin{align}
        \|\mT(\vy_1) - \vy_1\| &\leq (1 - \lambda_k)\beta^2  \epsilon_{k-1} + (1 - \lambda_k)(1 + \alpha)\|\vy_1 - \vy_0\|\notag\\
        &\leq (1 - \lambda_k)\beta^2  \epsilon_{k-1} + (1 - \lambda_k)^2(1 + \alpha)\epsilon_{k-1}\notag\\%
        &< (1-\lambda_k)(1 + \alpha + \beta^2) \epsilon_{k-1}. 
        \label{eq:fp1-1}
    \end{align}

    Now suppose that for some $j \geq 1$ we have for all $j' \in\{0, 1, \dots, j\},$ $\|\mT(\vy_{j'}) - \vy_{j'}\| < (1-\lambda_k)(1 + \alpha + \beta^2) \epsilon_{k-1}$. Then, by \Cref{def:gradually-expansive-op}, for all $j' \in\{0, 1, \dots, j\},$
    \begin{equation}\notag
        \|\mT(\vy_{j'}) - \mT(\vy_{j'-1})\| < \Big(1 + \frac{\alpha(1 + \alpha + \beta^2) \epsilon_{k-1}}{D}\Big)\|\vy_{j'} - \vy_{j' - 1}\|, 
    \end{equation}
    which further implies that
    \begin{align}
        \|\vy_{j'+1} - \vy_{j'}\| &= (1 - \lambda_k)\|\mT(\vy_{j'}) - \mT(\vy_{j'-1})\|\notag\\
        & < (1 - \lambda_k) \Big(1 + \frac{\alpha(1 + \alpha + \beta^2) \epsilon_{k-1}}{D}\Big)\|\vy_{j'} - \vy_{j' - 1}\|. \notag
    \end{align}
    Observing that
    \begin{align*}
        (1 - \lambda_k) \Big(1 + \frac{\alpha(1 + \alpha + \beta^2) \epsilon_{k-1}}{D}\Big) &= \frac{1}{1 + \beta^2 \epsilon_{k-1}/D_k} \Big(1 + \frac{\alpha(1 + \alpha + \beta^2) \epsilon_{k-1}}{D}\Big) \leq 1,
    \end{align*}
    where the inequality is from the condition $\beta^2(\beta - \alpha) \geq  \alpha(1 + \alpha)$ from the theorem statement, we conclude that for all $j' \in\{0, 1, \dots, j\},$ we have $\|\vy_{j'+1} - \vy_{j'}\| < \|\vy_{j'} - \vy_{j' -1}\|$ and so it must be
    \begin{equation}\label{eq:grad-dist-dec}
        \|\vy_{j+1} - \vy_j\| < \|\vy_1 - \vy_0\| \leq (1 -\lambda_k)\epsilon_{k-1}. 
    \end{equation}
    Using the definition of $\vy_{j+1}$ and the triangle inequality, 
    \begin{align*}
        \|\mT(\vy_{j+1}) - \vy_{j+1}\| &\leq \lambda_k\|\vy_0 - \mT(\vy_{j+1})\| + (1 - \lambda_k)\|\mT(\vy_{j+1}) - \mT(\vy_j)\|\\
        &{<} (1-\lambda_k)\beta^2 \epsilon_{k-1} + (1 + \alpha)(1 - \lambda_k)^2 \epsilon_{k-1}\\
        &< (1-\lambda_k)(1 + \alpha + \beta^2)\epsilon_{k-1}, 
    \end{align*}
    where the second line is by the definition of $\lambda_k$, $\|\vy_0 - \mT(\vy_{j+1})\| \leq D_k$, $ \|\mT(\vy_{j+1}) - \mT(\vy_j)\| \leq (1+\alpha)\|\vy_{j+1} - \vy_j\|$, and \Cref{eq:grad-dist-dec}, while the third line is by $\lambda_k \in (0, 1).$ 
    
    Thus, we conclude by induction on $j$ that all iterates $j$ in the inner while loop of {\Cref{algo:adaGHAL-rev}} satisfy $\|\mT(\vy_j) - \vy_j\| < (1-\lambda_k)(1 + \alpha + \beta^2)\epsilon_{k-1}.$ Hence, for all $j \geq 0,$
    \begin{equation}\label{eq:grad-halpern-contraction-factor-1}
        \|\mT(\vy_{j+1}) - \mT(\vy_j)\| \leq \Big(1 + (1-\lambda_k) \frac{\alpha \epsilon_{k-1}}{D}(1 + \alpha + \beta^2)\Big)\|\vy_{j+1} - \vy_j\|.
    \end{equation}
    Suppose now that $\beta^3 = (1 + c)\alpha(1 + \alpha + \beta^2)$ for some $c > 0$, where we require $\beta \in (\alpha, 1)$ (note that such a choice must exist, due to the assumption that $\alpha < \sqrt{2}-1${, which is equivalent to $\alpha\frac{1+\alpha}{1-\alpha}<1$}). Then:
    \begin{align}
        1 - \lambda_k &= \frac{1}{1 + \beta^2 \epsilon_{k-1}/D_k} 
        \leq \frac{1}{1 + (1 + c)\alpha(1 + \alpha + \beta^2)\epsilon_{k-1}/D}.\notag
    \end{align}
    Using $D_k \leq D/\beta,$ a rearrangement of the last inequality gives
    \(
        \alpha(1 + \alpha + \beta^2)\epsilon_{k-1}/D \leq \frac{1}{1+c}\cdot\frac{\lambda_k}{1- \lambda_k}, 
    \)
    which, when plugged back into \eqref{eq:grad-halpern-contraction-factor-1} leads to
    \begin{equation}\label{eq:grad-halpern-contraction-factor-2}
        \|\mT(\vy_{j+1}) - \mT(\vy_j)\| \leq \Big(1 + \frac{\lambda_k}{1+c}\Big)\|\vy_{j+1} - \vy_j\|, \; \forall j \geq 0.
    \end{equation}
    As a consequence, we have that for all $j \geq 1,$
    \begin{equation}\label{eq:grad-contraction-final}
    \begin{aligned}
         \|\vy_{j+1} - \vy_j\| \leq \Big(1 - \frac{\lambda_k}{1+c}(c + \lambda_k)\Big)\|\vy_{j} - \vy_{j-1}\| {\leq (1-\beta'\lambda_k)\norm{\vy_j - \vy_{j-1}}},
    \end{aligned}
    \end{equation}
    {thus the safeguard condition is never triggered. The iteration/oracle query upper bound $\cO\big(\frac{D}{\epsilon}\cdot \frac{-\ln(\beta(1-\beta))}{\beta'\beta^4}\big)$ now follows as an application of \Cref{thm:adaGHAL-Lipschitz}. For the latter bound in the min, one needs to use the stronger bound on $\rho_k$ obtained in the analysis above, by which $1 - \frac{\lambda_k}{1+c}(c + \lambda_k) \leq  1 - \lambda_k^2.$ This inequality holds because $\frac{\lambda_k}{1+c}(c + \lambda_k) \geq \lambda_k^2$ is (after simplifying) equivalent to $c\lambda_k \geq c\lambda_k^2$, which is true since $\lambda_k \in (0, 1)$ and $c > 0.$ The argument is similar to what was used in proving \Cref{thm:adaGHAL-Lipschitz}, and is omitted for brevity.}
\end{proof}

A direct consequence of \Cref{thm:GHAL-grad-expansive} is that  for $\alpha$-gradually expansive operators with $\alpha \in (0, \sqrt{2}-1),$ we can guarantee that $\inf_{\vx \in \cC}\norm{\mT(\vx) - \vx} = 0,$ even in infinite-dimensional spaces, despite the Lipschitz constant of $\mT$ being possibly as large as $\gamma = 1+ \alpha \in (1, \sqrt{2})$. Recall here that, without additional assumptions, in infinite-dimensional Banach spaces, the best that can be shown for a $\gamma$-Lipschitz operator $\mT$ is that $\inf_{\vx \in \cC}\norm{\mT(\vx) - \vx} \leq (1 - 1/\gamma)D,$ and this inequality is tight in general \cite{goebel1973minimal}. 

{The bound in \Cref{thm:GHAL-grad-expansive} is not fully specified, as the algorithm parameters $\beta', \beta$ are implicitly defined. The following corollary provides a full parameter specification.}

\begin{corollary}\label{cor:grad-exp}
   { Suppose $\alpha \in (0, \sqrt{2}-1)$ and define $\delta = \sqrt{2}-1 -\alpha.$ Then any parameter choice $\beta, \beta'$ that satisfies the following inequalities:
    \begin{equation}\label{eq:beta-beta'}
        1 > \beta \geq \big(1 - \delta(1 + \sqrt{2}/2)\big)^{1/3},\quad 0 < \beta' \leq \frac{\beta^3 - \alpha(1 + \alpha + \beta^2)}{\beta^3}
    \end{equation}
    also satisfies the conditions from \Cref{thm:GHAL-grad-expansive}. As a consequence, by choosing $\beta, \beta'$ to satisfy the right inequalities in \eqref{eq:beta-beta'} with equality, the oracle complexity of \Cref{algo:adaGHAL-rev} stated in \Cref{thm:GHAL-grad-expansive} is bounded by $\cO\big(\ln(\frac{1}{\delta})\min\{\frac{D}{\delta \epsilon}, \frac{D^2}{\epsilon^2}\}\big)$.}

    {In particular, if $\alpha \leq 0.4,$ then there is a universal choice of $\beta = 0.992,$ $\beta' = 0.02$ for which the oracle  complexity of \Cref{algo:adaGHAL-rev} is $\cO(\frac{D}{\epsilon}).$ }
\end{corollary}
\begin{proof}
Let $h:= 1 + \sqrt{2}/2$, so the stated bound on $\beta$ is $1 > \beta \geq (1 - h\delta)^{1/3}.$ Observe first that as $\delta = \sqrt{2} - 1 - \alpha \in (0, \sqrt{2} - 1),$ we have that $h\delta < \sqrt{2}/2 < 1,$ so $1-h\delta > 0,$ and thus under the stated condition it holds $\beta \in (0, 1),$ as required. For the choice of $\beta$ to be valid under the requirements of \Cref{thm:GHAL-grad-expansive}, we need to verify that $\beta^3 > \alpha(1 + \alpha + \beta^2).$ Since $\beta^3 \geq 1 - h\delta$ and $\alpha(1 + \alpha + \beta^2) < \alpha(2 + \alpha)$ (since $\beta < 1$), it suffices to verify that $1 - h\delta \geq \alpha(2 + \alpha).$ Recalling that $h = 1 + \sqrt{2}/2$ and $\alpha = \sqrt{2} - 1 - \delta,$ the last inequality is equivalent to $1 - (1+\sqrt{2}/2)\delta \geq 1 - 2\sqrt{2}\delta + \delta^2,$ which, after rearranging, is equivalent to $(3\sqrt{2}/2 - 1)\delta \geq \delta^2.$ Since $\delta \in (0, \sqrt{2}-1)$ by assumption, this inequality clearly holds.

    Now, to bound the oracle complexity of \Cref{algo:adaGHAL-rev}, let $\beta = (1-h\delta)^{1/3}$ and observe that since $\delta < \sqrt{2} - 1,$ we have $\beta > (1-\sqrt{2}/2)^{1/3} > 0.66$, meaning it is well separated from zero and so all terms in the oracle complexity depending on $1/\beta$ can be treated as constants. On the other hand, when $\delta \to 0,$ we have $\beta = 1 - h\delta/3 + \cO(\delta^2)$, so $1-\beta = \Theta(\delta)$. It remains to determine the order of $\beta'$, which we can choose as $\beta' = \frac{\beta^3 - \alpha(1+\alpha + \beta^2)}{\beta^3} > \beta^3 - \alpha(2+\alpha),$ as $\beta < 1$. By the calculation from the first paragraph in the proof, $\beta^3 - \alpha(2+\alpha) \geq (3\sqrt{2}/2 - 1)\delta - \delta^2,$ so $\beta' = \Omega(\delta).$ Plugging into the oracle complexity bound from \Cref{algo:adaGHAL-rev}, we get that the resulting oracle complexity in this case is $\cO\big(\ln(\frac{1}{\delta})\min\{\frac{D}{\delta \epsilon}, \frac{D^2}{\epsilon^2}\}\big).$

    {Now let $\alpha \leq 0.4,$ so $\delta \geq 0.01421$ can be treated as a constant. Then, computing the numerical values in \eqref{eq:beta-beta'}, we get that it suffices that $\beta \geq 0.992$, $\beta' \leq 0.02.$ Since all algorithm parameters are now absolute constants, the oracle complexity is $\cO(\frac{D}{\epsilon})$.}
\end{proof}
%

\subsection{A Simple Example and Properties}\label{sec:grad-expansive-examples}

Gradual expansion is a technical condition that, as mentioned earlier, was discovered when studying the convergence properties of AdaGHAL. However, its interpretation is intuitive: it bounds the excess expansion ($\gamma - 1$) of an operator between two points proportional to the maximum of fixed-point residuals of the two points relative to the diameter of the set. Further, this condition can be seen as a quantitative version of the ``asymptotically nonexpansive'' property studied in \cite{lauster2021convergence,berdellima2022alpha}, defined as the property that for an operator $\mT$ and any $\delta > 0,$ there exists a sufficiently small neighborhood $\cN_\delta$ of the fixed point set of $\mT$ on which $\mT$ is $(1 + \delta)$-Lipschitz. Conversely: any asymptotically nonexpansive operator that satisfies a local error bound  such as $\|\mT(\vx) - \vx\| \geq \mu\, \dist(\vx, \cX_*)$, where $\cX_* \neq \emptyset$ is the set of fixed points of $\mT$, can be argued to be gradually nonexpansive in a local neighborhood of $\cX_*$; for instance, this is true for generalizations of proximal mappings of convex functions in CAT($\kappa$) spaces; see \cite{lauster2021convergence,berdellima2022alpha} for relevant definitions.\footnote{The author thanks Prof.\ Russell Luke for these insights.}

In the lemma below, we provide an explicit construction of an operator that is gradually expansive on its domain. 

\begin{lemma}\label{lemma:ex-grad-exp}
    Given $d \geq 1,$ $\alpha \in [0, 1),$ $D > 0$, and $\cC = [-D/2, D/2]^d,$ let $\mT(\vx) := \proj_{\cC}(\vx + \exp(\frac{\alpha \vx}{D}))$, where $\proj_{\cC}$ is the orthogonal projection onto $\cC$ and $\exp(\frac{\alpha \vx}{D})$ applies coordinate-wise. Then $\mT(\vx)$ is $\alpha$-gradually expansive w.r.t.\ the $\ell_\infty$ norm.
\end{lemma}
\begin{proof}
    Observe that $\mT: \cC\to \cC$, by its definition. The claim of the lemma is immediate when $\alpha = 0$, as $\mT$ is nonexpansive in this case (coordinate-wise shift + clipping to the interval $[-D/2, D/2]$), so assume that $\alpha \in (0, 1).$ Consider first a univariate operator $t(x)$ defined via $\proj_{[-D/2, D/2]}(x + f(x))$, where $\proj_{[-D/2, D/2]}(t) = \min\{D/2, \max\{-D/2, t\}\}$ is the projection onto the  interval $[-D/2, D/2]$ and we define $f(x) = \exp(\alpha x/D)$. In the rest of the proof, we argue that $t$ is $\alpha$-gradually expansive. This suffices for proving that $\mT$ is $\alpha$-gradually expansive, since (using superscript notation $^{(i)}$ for indexing  coordinates)
    \ifsiopt
    \begin{align*}
        \|&\mT(\vx) - \mT(\vy)\|_\infty = \max_{1 \leq i \leq d}|t(\vx^{(i)}) - t(\vy^{(i)})|\\
        &\leq  \max_{1 \leq i \leq d}\Big(1 + \frac{\alpha\max\{|t(\vx^{(i)}) - \vx^{(i)}|,\, |t(\vy^{(i)}) - \vy^{(i)}|\}}{D} \Big)|\vx^{(i)} - \vy^{(i)}|\\
        &\leq \Big(1 + \frac{\alpha\max\{\norm{\mT(\vx) - \vx}_\infty,\, \norm{\mT(\vy) - \vy}_\infty\}}{D} \Big)\norm{\vx - \vy}_\infty,
    \end{align*}
    \else
    \begin{align*}
        \|\mT(\vx) - \mT(\vy)\|_\infty &= \max_{1 \leq i \leq d}|t(\vx^{(i)}) - t(\vy^{(i)})|\\
        &\leq  \max_{1 \leq i \leq d}\Big(1 + \frac{\alpha\max\{|t(\vx^{(i)}) - \vx^{(i)}|,\, |t(\vy^{(i)}) - \vy^{(i)}|\}}{D} \Big)|\vx^{(i)} - \vy^{(i)}|\\
        &\leq \Big(1 + \frac{\alpha\max\{\norm{\mT(\vx) - \vx}_\infty,\, \norm{\mT(\vy) - \vy}_\infty\}}{D} \Big)\norm{\vx - \vy}_\infty,
    \end{align*}
    \fi
    where the first inequality is by $t$ being $\alpha$-gradually expansive. 
    
    Observe that $f'(x) = \frac{\alpha}{D}f(x)$ and so $|f(x) - f(y)| \leq \max\{|f'(x)|, |f'(y)|\}|x-y| = \frac{\alpha}{D}\max\{|f(x)|, |f(y)|\}|x-y|$. Thus, it is immediate that $|t(x) - t(y)| \leq (1 + \frac{\alpha}{D}\max\{|f(x)|, |f(y)|\})|x-y|$. As a consequence, if $t(x) = x + f(x)$ and $t(y) = y + f(y)$ (i.e., if neither $x$ nor $y$ gets clipped by the projection onto $[-D/2, D/2]$), then the inequality defining $\alpha$-gradual expansiveness  holds. 

    Observe further that $t(x)$ is increasing in $x,$ so the only way it can happen that $t(x) \neq x + f(x)$ is if $x + f(x) > D/2,$ in which case $t(x) = D/2.$ If both $x$ and $y$ get clipped, then both are equal to $D/2,$ so in this case $|t(x) - t(y)| = 0$ and the inequality defining $\alpha$-gradually expansive operators holds trivially. Thus, it remains to argue that this inequality still holds when only one of $x, y$ gets clipped. Suppose w.l.o.g.\ that $t(x) = D/2$ and $t(y) < D/2.$ Observe that in this case it must be $x > y$. Let $\Delta(x) = |t(x) - x| = t(x) - x = D/2 - x.$ We now have:
    \begin{align*}
        |t(x) - t(y)| &= D/2 - y - e^{\alpha y /D}\\
                     &= x - y + \Delta(x) - e^{\alpha y /D}.
    \end{align*}
    To complete the proof of the lemma, it suffices to argue that in this case we have
    \begin{equation}\label{eq:ex-grad-exp-condition}
        \Delta(x) - e^{\alpha y /D} \leq \frac{\alpha}{D}\max\{\Delta(x),\, e^{\alpha y /D}\}(x - y),
    \end{equation}
    as $x > y,$ $\Delta(x) = |t(x) - x|$, and $e^{\alpha y /D} = |t(y) - y|$. 
    
    When $\Delta(x) \leq e^{\alpha y /D}$, \eqref{eq:ex-grad-exp-condition} holds trivially, as the left-hand side is non-positive, while the right-hand side is positive. Thus, suppose now that $\Delta(x) > e^{\alpha y /D}$. Then, \eqref{eq:ex-grad-exp-condition} is equivalent to 
    \(
        \Delta(x) - e^{\alpha y /D} \leq \frac{\alpha}{D}\Delta(x)(x - y),
    \)
    which, rearranging and recalling that $\Delta(x) = D/2 - x \geq 0,$ is equivalent to
    \begin{equation}\label{eq:ex-grad-exp-condition-final}
        (D/2 - x)\Big(1 - \frac{\alpha}{D}(x - y)\Big) \leq e^{\alpha y /D}. 
    \end{equation}
    Observe that the left-hand side of \eqref{eq:ex-grad-exp-condition-final} is decreasing in $x.$ This is true because both $(D/2 - x)$ and $\big(1 - \frac{\alpha}{D}(x - y)\big)$ are non-negative and decreasing in $x$ (the latter term is non-negative because $(x - y) \leq D$ and $\alpha < 1$). Thus, the left-hand side is maximized when $x$ is minimized, which for the present case means that $x$ is such that $x + e^{\alpha x/D} = D/2$ (the smallest value at which $x$ gets clipped; if no such point exists, then $t(x) = D/2$ for all $x \in [-D/2, D/2]$ and the lemma claim holds trivially). Plugging this value of $x$ into \eqref{eq:ex-grad-exp-condition-final}, we get that it suffices to ensure that
    \begin{equation}\notag
        e^{\alpha x/D}\Big(1 - \frac{\alpha}{D}(x - y)\Big) \leq e^{\alpha y/D}.
    \end{equation}
    This inequality is equivalent to $1 + \frac{\alpha}{D}(y - x) \leq e^{\alpha/D(y - x)}$, which is true, as $e^a \geq 1 + a$ for any real number $a$. 
\end{proof}
\Cref{lemma:ex-grad-exp} provides an example of an operator $\mT$ that is $\alpha$-gradually expansive (but not nonexpansive, for any $\alpha \in (0, 1)$ and $D > 2$). However, this example is not interesting for numerical evaluations, as any reasonable fixed-point iteration (including Picard) would converge to a fixed point of $\mT$. This is because $\mT(\vx) \geq \vx$ coordinate-wise, and as soon as any of the coordinates reaches $D/2,$ a fixed point of $\mT$ is found. Nevertheless, we hope that this simple example can serve as a starting point for constructing more interesting gradually expansive operators, possibly by trying to replace the separable exponential mapping in its definition by a different mapping that would prevent Picard iteration from converging, while retaining the gradual expansion property. Towards this goal, we also state and prove a simple structural lemma that provides sufficient conditions for an operator to be  $\alpha$-gradually expansive.

\begin{lemma}\label{lemma:grad-exp-composition}
    Let $\mT: \cC \to \cC$ be an operator that satisfies \Cref{assp:compact-domain}. Let $\mF := \mT - \mId.$ If for all $\vx, \vy \in \cC$ it holds that \ifsiopt$$\|\mF(\vx) - \mF(\vy)\| \leq \frac{\alpha}{D}\max\{\|\mF(\vx)\|,\, \|\mF(\vy)\|\}\|\vx - \vy\|$$\else$\|\mF(\vx) - \mF(\vy)\| \leq \frac{\alpha}{D}\max\{\|\mF(\vx)\|,\, \|\mF(\vy)\|\}\|\vx - \vy\|$\fi for some $\alpha \in [0, 1)$, then $\mT$ is $\alpha$-gradually expansive. Further, in this case, if there exists a nonexpansive operator $\mR$ such that $\mR(\vx) \in \cC$ and $\vx + \mF(\mR(\vx)) \in \cC'$ for all $\vx \in \cC',$ on some bounded closed convex set $\cC'$ of diameter $D' < \infty$, then the operator $\mTt = \mId + \mF\circ \mR$ is $\alpha'$-gradually expansive with $\alpha' = \alpha D'/D.$ 
\end{lemma}
\begin{proof}
    The first claim of the lemma is immediate, as, by the triangle inequality and the definition of $\mF,$ we have 
    \begin{align*}
        \|\mT(\vx) - \mT(\vy)\| &\leq \|\vx - \vy\| + \|\mF(\vx) - \mF(\vy)\|\\
        &\leq \Big(1 + \frac{\alpha}{D}\max\{\|\mF(\vx)\|,\, \|\mF(\vy)\|\}\Big)\|\vx - \vy\|\\
        &= \Big(1 + \frac{\alpha}{D}\max\{\|\mT(\vx) - \vx\|,\, \|\mT(\vy) - \vy\|\}\Big)\|\vx - \vy\|. 
    \end{align*}
    For the remaining  claim, since  $\|\mF(\vx) - \mF(\vy)\| \leq \frac{\alpha}{D}\max\{\|\mF(\vx)\|,\, \|\mF(\vy)\|\}\|\vx - \vy\|$ and $\mR$ is nonexpansive, we have
    \begin{align*}
        \|\mF(\mR(\vx)) - \mF(\mR(\vy))\| &\leq \frac{\alpha}{D}\max\{\|\mF(\mR(\vx))\|,\, \|\mF(\mR(\vy))\|\}\|\mR(\vx) - \mR(\vy)\|\\
        &\leq \frac{\alpha}{D}\max\{\|\mF(\mR(\vx))\|,\, \|\mF(\mR(\vy))\|\}\|\vx - \vy\|\\
        &= \frac{\alpha}{D}\max\{\|\mTt(\vx) - \vx\|,\, \|\mTt(\vy) - \vy\|\}\|\vx - \vy\|,
    \end{align*}
    where the second inequality is by the nonexpansiveness of $\mR$ and the last inequality is by the definition of $\mTt.$ To complete the proof, it remains to use the triangle inequality by which $\|\mTt(\vx) - \mTt(\vy)\| \leq \|\vx - \vy\| +  \|\mF(\mR(\vx)) - \mF(\mR(\vy))\|\leq \big(1 + \frac{\alpha}{D}\max\{\|\mTt(\vx) - \vx\|,\, \|\mTt(\vy) - \vy\|\}\big)\|\vx - \vy\|$ and $\frac{\alpha}{D} = \frac{\alpha'}{D'}$.  
\end{proof}

\subsection{Implications of Mild and Gradual Expansion on Root Finding Problems}\label{sec:root-finding}

In this section, we discuss some implications of the presented results on related root finding problems for maximal, possibly multi-valued operators $\mF: \cE \rightrightarrows \cE$. In this subsection only, we assume that $\norm{\cdot}$ is inner-product induced, so that the space $\cE$ is self-dual. We assume that the resolvent operator $\mJ = ( \mF + \mId)^{-1}$ is well-defined and  reason about root finding problems, which amount to finding $\vx$ such that $\vzero \in \mF(\vx)$, by considering the problem of finding a fixed point for the resolvent operator.  A standard observation here is that if $\mJ(\vx) = \vu$, then, by the definition of $\mJ,$ we have that there exists $\mG_{\vu} \in \mF(\vu)$ such that
\begin{equation}\label{eq:resolvent-mapping}
 \mG_{\vu} + \vu = \vx.
\end{equation}
Thus, given any $\epsilon \geq 0,$ we have that $\|\mG_\vu\| \leq \epsilon$ if and only if $\|\mJ(\vx) - \vx\| \leq \epsilon.$ Hence we can reduce the problem of finding an $\epsilon$-approximate root of $\mF$ (namely, a point $\vu$ such that $\exists \mG \in \mF(\vu)$ with $\|\mG\| \leq \epsilon$) to the problem of finding an $\epsilon$-approximate fixed point of $\mJ.$ Thus, in the following, we provide sufficient conditions for $\mF$ that ensure $\mJ$ is either mildly or gradually expansive, so the associated root finding problems can be approximately solved by finding an approximate fixed point of $\mJ$. 

First, we establish a sufficient condition for $\gamma$-Lipschitzness of $\mJ$, as a condition on $\mF$ stated in the following lemma, which in turn provides a sufficient condition on $\mF$ that leads to a mildly nonexpansive resolvent. This lemma is simple to prove and is provided for completeness.

\begin{lemma}\label{lem:resolvent-mildly-exp}
    Given $\mF: \cE \rightrightarrows \cE$ such that $\dom(\mF) \neq \emptyset,$ let $\mJ  =  (\mF + \mId)^{-1}.$ If for some $\mu \in [0, 1)$, any $\vu, \vv \in \dom(\mF)$, and any $\mG_u \in \mF(\vu), \mG_\vv \in \mF(\vv),$ we have
    \begin{equation}\label{eq:hypomotone}
        \innp{\mG_\vu - \mG_\vv, \vu - \vv} \geq - \mu \|\vu - \vv\|^2,
    \end{equation}
    then $\mJ$ is $\frac{1}{1-\mu}$-Lipschitz. 
    In particular, suppose that $\sup_{\vx, \vy \in \dom(\mF)}\|\vx - \vy\| = D.$ Then a sufficient condition for $\mJ$ to be $\gamma$-Lipschitz with $\gamma = 1 + \frac{\epsilon \beta}{D}$ for some $\epsilon > 0,$ $\beta \in (0, 1)$ is that the above inequality applies with $\mu = \frac{\beta \epsilon}{D + \beta \epsilon}.$ 
\end{lemma}
\begin{proof}
    The second part of the lemma is an immediate corollary of the first part, so we only prove the first statement. Fix any $\vx, \vy \in \cE$ and let $\vu = \mJ(\vx),$ $\vv = \mJ(\vy)$. Observe that $\vu, \vv \in \dom(\mF)$. Using the definition of the resolvent, we further have $\vx = \mG_\vu + \vu$ for some $\mG_\vu \in \mF(\vu)$ and $\vy = \mG_\vv + \vv$ for some $\mG_\vv \in \mF(\vv).$ Thus, 
    \begin{equation}\label{eq:res-1st-eq}
        \|\vu - \vv\|^2 = \innp{\vx - \vy, \vu - \vv} - \innp{\mG_\vu - \mG_\vv, \vu - \vv}. 
    \end{equation}
    Since, by assumption, $\innp{\mG_\vu - \mG_\vv, \vu - \vv} \geq - \mu \|\vu - \vv\|^2,$ rearranging \eqref{eq:res-1st-eq} and applying Cauchy-Schwarz inequality, we get
    \begin{align*}
        (1-\mu)\|\vu - \vv\|^2 \leq \innp{\vx - \vy, \vu - \vv} \leq \|\vx - \vy\|\|\vu - \vv\|.
    \end{align*}
    Recalling that $\vu = \mJ(\vx), \vv = \mJ(\vy),$ this in turn implies that $\mJ$ is $\frac{1}{1-\mu}$-Lipschitz. 
\end{proof}
One way of satisfying the conditions of \Cref{lem:resolvent-mildly-exp} is when $\mF$ is the sum of a single-valued operator $\mG$ that satisfies \eqref{eq:hypomotone} and a normal cone of a convex set of bounded diameter $D.$ For instance, such a condition would hold for weakly convex-weakly concave min-max optimization problems with a sufficiently small (order $\epsilon/D$) weak convexity parameter $\mu$, without imposing any additional assumptions such as the existence of weak solutions to the associated variational inequality problem \cite{rafique2022weakly}.\footnote{Assuming that there exists a weak (or Minty) solution to the associated variational inequality is very restricting. For instance, for Lipschitz operators, such an assumption suffices for solving the associated problem with basic extragradient-type methods without any further assumptions on weak monotonicity as in \eqref{eq:hypomotone}; see, for example, \cite{diakonikolas2021efficient}.}

Finally, it is possible to derive a sufficient condition for the resolvent of a possibly nonmonotone operator $\mF$ to be $\alpha$-gradually expansive, as stated in the lemma below.

\begin{lemma}\label{lem:VI-grad-exp}
    Given $\mF: \cE \rightrightarrows \cE$ such that $\dom(\mF) \neq \emptyset$ and $\sup_{\vx, \vy \in \dom(\mF)}\|\vx - \vy\| = D,$ let $\mJ  =  (\mF + \mId)^{-1}.$ %
    If for some $\alpha \in [0, 1)$, any $\vu, \vv \in \dom(\mF)$, and any $\mG_u \in \mF(\vu), \mG_\vv \in \mF(\vv),$ we have
    \begin{equation}\label{eq:hypomotone-grad}
        \innp{\mG_\vu - \mG_\vv, \vu - \vv} \geq - \frac{\tau_{\vu, \vv}}{1 + \tau_{\vu, \vv}} \|\vu - \vv\|^2,
    \end{equation}
    where $\tau_{\vu, \vv}:= \frac{\alpha}{D}\max\{\|\mG_\vu\|, \|\mG_\vv\|\}$, then $\mJ$ is $\alpha$-gradually expansive. 
\end{lemma}
\begin{proof}
    Similar to the proof of \Cref{lem:resolvent-mildly-exp}, fix any $\vx, \vy \in \cE$ and let $\vu = \mJ(\vx),$ $\vv = \mJ(\vy)$. Recall that $\vx = \mG_\vu + \vu$ for some $\mG_\vu \in \mF(\vu)$ and $\vy = \mG_\vv + \vv$ for some $\mG_\vv \in \mF(\vv).$ Then: 
    \begin{equation}\label{eq:res-1st-eq-grad}
        \|\vu - \vv\|^2 = \innp{\vx - \vy, \vu - \vv} - \innp{\mG_\vu - \mG_\vv, \vu - \vv}. 
    \end{equation}
    As, by assumption, $ \innp{\mG_\vu - \mG_\vv, \vu - \vv} \geq - \frac{\tau_{\vu, \vv}}{1 + \tau_{\vu, \vv}} \|\vu - \vv\|^2,$ rearranging \eqref{eq:res-1st-eq-grad} yields
    \begin{equation}\notag
        \|\vu - \vv\|^2 \leq (1 + \tau_{\vu, \vv}) \innp{\vx - \vy, \vu - \vv}.
    \end{equation}
    Thus, using the Cauchy-Schwarz inequality and simplifying,
     \begin{equation}\notag
        \|\vu - \vv\| \leq (1 + \tau_{\vu, \vv}) \|\vx - \vy\|.
    \end{equation}
    To complete the proof, it remains to recall that $\vu = \mJ(\vx),$ $\vv = \mJ(\vy),$ $\mG_\vu = \vx - \mJ(\vx),$ and $\mG_\vv = \vy - \mJ(\vy).$ Plugging these identities into the last inequality leads to the conclusion that $\mJ$ is $\alpha$-gradually expansive, by definition. 
\end{proof}

A few remarks are in order here. First, observe that it is not necessary for \eqref{eq:hypomotone-grad} to hold for all elements of $\mF(\vu)$ but only those selected by the resolvent; i.e., those that satisfy $\mG_\vx \in \mF(\mJ(\vx))$ and $\vx = \mG_\vx + \mJ(\vx)$ (though it is not immediately clear how to enforce such a condition). Second, perhaps the most basic example of an operator $\mF$ to consider is when it is expressible as the sum of a 1-Lipschitz-continuous\footnote{Here, 1-Lipschitz is w.l.o.g., as we can rescale $\mF_1$ by its Lipschitz constant.} operator $\mF_1$ and the normal cone of a closed convex set $\cC$ with a nonempty interior. Then, in the interior of the set $\cC,$ we have $\mF = \mF_1$, which is single-valued and 1-Lipschitz-continuous. The condition \eqref{eq:hypomotone-grad} for $\vx, \vy$ in the interior of $\cC$ then  becomes
\begin{equation}\notag
    \innp{\mF(\vx) - \mF(\vy), \vx - \vy} \geq -\frac{(\alpha/D)\max\{\norm{\mF(\vx)}, \norm{\mF(\vy)}\}}{1 + (\alpha/D)\max\{\norm{\mF(\vx)}, \norm{\mF(\vy)}\}}\|\vx - \vy\|^2.
\end{equation}
{This condition is similar to the $(L_0, L_1)$-Lipschitzness studied in the recent literature (with $L_0 = 0$, see \cite{choudhury2025extragradient} and references therein); however, here it is only used to bound how non-monotone the operator may be, rather than to generalize Lipschitz conditions for classes of problems previously addressed under standard Lipschitzness.}
Since it is possible, in general, for $\max\{\norm{\mF(\vx)}, \norm{\mF(\vy)}\}$ to take values of order-$D$ (or even larger, if $\cC$ does not contain a root of $\mF_1$), this condition allows for the operator $\mF_1$ to be hypomonotone between some pairs of points with a hypomonotonicity parameter that is an absolute constant smaller than one (but can be close to one). 

The condition \eqref{eq:hypomotone-grad} is required to hold on the boundary of the set $\cC$ as well for the resolvent to be gradually expansive, and such a condition seems more challenging to ensure. Even so, we note that the current literature addressing root finding problems with operators that are not monotone is only able to address problems under strong conditions either requiring $\mJ(\vx)$ to be nonexpansive (see \cite{bauschke2021generalized} for sufficient conditions that ensure the resolvent is nonexpansive) or slightly relaxing this property to the weak MVI condition introduced in \cite{diakonikolas2021efficient}; see, for instance, \cite{alacaoglu2024revisiting} and references therein. 

Finally, another avenue for applying results from this paper to root finding problems is, similar to the example above, by considering problems where $\mF$ is the sum of a 1-Lipschitz operator $\mF_1$ and the normal cone of a closed convex set $\cC$ with a nonempty interior and seeking a fixed point of the projection operator $\mT(\vx) = \proj_{\cC}(\vx - \mF_1(\vx))$. Note that, by standard arguments, points $\vx \in \cC$ with $\norm{\mT(\vx) - \vx} \leq \epsilon$ translate into $\cO(\epsilon)$-approximate roots of $\mF.$ In this case, the norm of $\mId - \mT$ is also known as the ``natural residual'' $r(\vx) = \norm{\proj_{\cC}(\vx - \mF_1(\vx)) - \vx}$ (see, e.g., \cite{facchinei2007finite}). The $\alpha$-gradual nonexpansiveness in this case becomes the following condition for $\vx, \vy \in \cC$:
\begin{equation}\label{eq:grad-exp-for-proj}
\begin{aligned}
      \norm{\proj_{\cC}(\vx - \mF_1(\vx)) - \proj_{\cC}(\vy - \mF_1(\vy))}
     \leq  \Big(1 + \frac{\alpha\max\{r(\vx),\, r(\vy)%
     \}}{D}\Big)\norm{\vx - \vy}.
\end{aligned}
\end{equation}
This seems like a more tractable condition that may be possible to verify (at least numerically) for different example problems arising in machine learning, such as, for example, those discussed in \cite{daskalakis2021independent,daskalakis2022non}. 

\subsection{{Implications on Weakly Convex Optimization}}\label{sec:weakly-convex}

{We now briefly discuss some implications of our results on \emph{last iterate} convergence for a class of weakly convex optimization problems in Euclidean spaces. We begin by reviewing basic definitions and facts pertaining to weakly convex functions. Let $\norm{\cdot}$ be a Euclidean norm on $\sR^d$, induced by an inner product $\innp{\cdot, \cdot}$. Recall that a function is said to be $\rho$-weakly convex if $\vx \mapsto f(\vx) +  \frac{\rho}{2}\norm{\vx}^2$ is convex. For weakly convex functions, the (Fr\'{e}chet) subdifferential $\partial f (\vx)$ for any $\vx \in \sR^d$  is defined as the set of all vectors $\vg \in \sR^d$ such that $f(\vy) \geq f(\vx) + \innp{\vg, \vy - \vx} - o(\norm{\vy - \vx})$ as $\vy \to \vx.$ %
Further, the following inequalities hold (see, e.g., \cite{davis2019stochastic} for the fact below and additional properties and definitions).}
{\begin{fact}
    Any lower-semicontinuous function $f: \sR^d \to \sR \cup \{+\infty\}$ is $\rho$-weakly convex if and only if either of the following conditions holds for all $\vx, \vy \in \dom(f)$:
    \begin{align}
        f(\vy) - f(\vx) - \innp{\vg_\vx, \vy - \vx} &\geq  - \frac{\rho}{2}\norm{\vy - \vx}^2, \quad&&\forall \vg_\vx \in \partial f(\vx), \\
        \innp{\vg_\vx - \vg_\vy, \vx - \vy} &\geq - \rho \norm{\vx - \vy}^2, \quad&&\forall \vg_\vx \in \partial f(\vx), \forall \vg_{\vy} \in \partial f(\vy).
    \end{align}
\end{fact}}
{Without loss of generality and to simplify notation, we assume that $\rho < 1;$ this can always be achieved by rescaling $f$ (i.e., considering $h = c f$ for $c < 1/\rho$ instead of $f$ if $\rho \geq 1$). In this case, the proximal operator at any $\vx \in \sR^d$, defined by
\begin{equation}\label{eq:prox-def}
    \prox_f(\vx) := \argmin_{\vu \in \sR^d} \Big\{f(\vu) + \frac{1}{2}\norm{\vu - \vx}^2\Big\}
\end{equation}
is well-defined and unique, and, further, $\vu = \prox_f(\vx)$ satisfies $\vg_{\vu} + \vu - \vx = \vzero$ for some $\vg_\vu \in \partial f(\vu).$ We additionally make the following assumption:
\begin{assumption}\label{assp:bnded-sublevel-set}
    Given some $\vx_0 \in \sR^d,$ the sublevel set $\cL := \{\vx: f(\vx) \leq f(\vx_0)\}$ is convex and bounded, with diameter $D < \infty.$ 
\end{assumption}
Observe that under \Cref{assp:bnded-sublevel-set}, for a $\rho$-weakly convex function $f$ with $\rho < 1,$ $\mT(\vx):= \prox_f(\vx)$ maps $\cL$ to itself, and so \Cref{assp:compact-domain} holds. 
}

{We are now ready to provide a characterization of a class of weakly convex minimization problems for which the proximal operator is gradually expansive. As a result, \Cref{thm:GHAL-grad-expansive} applies in this case and leads to the fast convergence, of the order $1/k$, for the stationarity guarantee of the last algorithm iterate. More details are provided below the proof of the following lemma.}
{\begin{lemma}\label{lemma:weakly-covnvex-grad-exp}
    Given $\vx_0$ and a proper lower-semicontinuous $\rho$-weakly convex function $f: \sR^d \to \sR \cup \{+\infty\}$ with $\rho < 1,$ suppose \Cref{assp:bnded-sublevel-set} holds. Suppose further that there exists $\hat{\alpha} \in (0, 1)$ such that for all $\vu, \vv \in \cL$ and all $\vg_\vu \in \partial f(\vu), \vg_\vv \in \partial f(\vv)$,  %
    \begin{equation}\label{eq:weakly-convex-grad-exp}
        \innp{\vg_\vu - \vg_\vv, \vu - \vv} \geq - \frac{\hat{\alpha}}{D} \max\{\norm{\vg_\vu},\, \norm{\vg_\vv}\}\norm{\vu - \vv}^2.
    \end{equation}
    Then, $\mT(\vx) := \prox_f(\vx)$ is $\alpha$-gradually expansive on $\cL$, with $\alpha = \frac{\hat{\alpha}}{1 - \hat{\alpha}}$.
\end{lemma}}
\begin{proof}
    {First, because $\rho < 1,$ $\prox$ operator is well-defined and unique. Fix any $\vx, \vy \in \cL$ and let $\vu := \mT(\vx) = \prox_f(\vx)$, $\vv := \mT(\vy) = \prox_f(\vy)$. By the definition of prox operator \eqref{eq:prox-def}, there is some $\vg_\vu \in \partial f(\vu)$ and some $\vg_\vv \in \partial f(\vv)$ such that
    \begin{equation}\label{eq:vuvv-prox-equality}
        \vg_\vu + \vu - \vx = \vzero,\quad \vg_\vv + \vv - \vy = \vzero.
    \end{equation}
    We thus have that, using \eqref{eq:vuvv-prox-equality},
    \begin{align}
        \norm{\vu - \vv}^2 &= \innp{\vx - \vy - (\vg_\vu - \vg_\vv), \vu - \vv} \notag \\
        &\leq \innp{\vx - \vy, \vu - \vv} + \frac{\hat{\alpha}}{D} \max \{\norm{\vg_\vu}, \norm{\vg_\vv}\}\norm{\vu - \vv}^2, \label{eq:weakly-convex-derivation-1}
    \end{align}
    where the last inequality is by \eqref{eq:weakly-convex-grad-exp}. 
    Observe that, because $\vx, \vy \in \cL$ and because $f(\vu) \leq f(\vx),$ $f(\vv) \leq f(\vy)$ (by the definition of prox, \eqref{eq:prox-def}), we have $\vu, \vv \in \cL$. As a consequence, $\norm{\vu - \vx} \leq D,$ $\norm{\vy - \vv} \leq D,$ and so \eqref{eq:vuvv-prox-equality} implies that $\max \{\norm{\vg_\vu}, \norm{\vg_\vv}\} \leq D.$ Thus, we can rearrange \eqref{eq:weakly-convex-derivation-1} to get
    \begin{equation}\label{eq:weakly-convex-derivation-2}
        \norm{\vu - \vv}^2 \leq \Big(1 - \frac{\hat{\alpha}}{D} \max \{\norm{\vg_\vu}, \norm{\vg_\vv}\}\Big)^{-1} \innp{\vu - \vv, \vx - \vy}. 
    \end{equation}
    If $\norm{\vu - \vv} = 0,$ then trivially $\norm{\vu - \vv} \leq \norm{\vx - \vy},$ so suppose this is not the case. Applying Cauchy-Schwarz to the right-hand side of \eqref{eq:weakly-convex-derivation-2} and dividing both sides by $\norm{\vu - \vv},$ we get
    \begin{equation}\notag
        \norm{\vu - \vv} \leq \Big(1 - \frac{\hat{\alpha}}{D} \max \{\norm{\vg_\vu}, \norm{\vg_\vv}\}\Big)^{-1} \norm{\vx - \vy}. 
    \end{equation}
    To complete the proof, it remains to argue that $\big(1 - \frac{\hat{\alpha}}{D} \max \{\norm{\vg_\vu}, \norm{\vg_\vv}\}\big)^{-1} \leq 1 + \frac{\hat{\alpha}}{1 - \hat{\alpha}} \frac{\max \{\norm{\vg_\vu}, \norm{\vg_\vv}\}}{D}.$ This follows by letting $\tau : = \frac{\hat{\alpha}}{D} \max \{\norm{\vg_\vu}, \norm{\vg_\vv}\},$ observing that $\tau \in [0, \hat{\alpha}],$ and using that $\frac{1}{1-\tau} \leq 1 + \frac{\tau}{1-\tau} \leq 1 +\frac{\tau}{1-\hat{\alpha}}.$
    }
\end{proof}
{A few remarks are in order here. First, as can be observed from the proof of \Cref{lemma:weakly-covnvex-grad-exp}, the condition \eqref{eq:weakly-convex-grad-exp} need not apply to the entire subdifferential of $f$ on $\cL;$ rather only those subgradients selected by the prox mapping. Second, \eqref{eq:weakly-convex-grad-exp} on the whole subdifferential of $f$ on $\cL$ implies weak convexity, even if it is not explicitly assumed. The statement of \Cref{lemma:weakly-covnvex-grad-exp} was chosen as is, however, for clarity.}

{Further, \Cref{lemma:weakly-covnvex-grad-exp} implies that if \eqref{eq:weakly-convex-grad-exp} holds on $\cL$ with $\hat{\alpha} \in (0, 1 - \sqrt{2}/2)$, then $\mT(\vx) := \prox_f(\vx)$ is $\alpha$-gradually expansive with $\alpha \in (0, \sqrt{2}-1)$ on $\cL.$ It follows then, by \Cref{thm:GHAL-grad-expansive}, that, given any $\epsilon > 0$ and if $\hat{\alpha}$ is not trivially close to $1 - \sqrt{2}/2$, \Cref{algo:adaGHAL-rev} applied to $\mT(\vx) := \prox_f(\vx)$ outputs an $\vxh$ with
\begin{equation}\notag
    \dist(\vzero, \partial f(\prox_f(\vxh))) \leq \norm{\prox_f(\vxh) - \vxh} \leq \epsilon
\end{equation}
within $\cO(\frac{D}{\epsilon})$ iterations/prox oracle queries to $f.$ Further, this guarantee applies to the last iterate of the algorithm. This is a much stronger guarantee than what can be achieved for general weakly convex functions, where the guarantee is on the best iterate and the oracle complexity scales with $\frac{D^2}{\epsilon^2}$ \cite{davis2019stochastic,carmon2020lower}. Instead, the oracle complexity obtained here matches the optimal oracle complexity for convex functions; see \cite{woodworth2016tight}.  
}

\section{Numerical Examples} \label{sec:num-exp}

In this section, we provide numerical examples illustrating different properties of AdaGHAL. In all the examples, AdaGHAL (cons.) refers to \Cref{algo:adaGHAL-rev} with the conservative parameter setting $\beta = 0.99, \beta' = 0.02$ from \Cref{cor:grad-exp}, while AdaGHAL (aggr.) refers to  
 refers to \Cref{algo:adaGHAL-rev} with an aggressive parameter setting $\beta = 0.5, \beta' = 0.1$. Convergence of AdaGHAL is compared to the following baselines: classical Picard iteration, classical Halpern iteration with steps $\lambda_k = \frac{1}{k+1},$ and restarted Halpern iteration---which runs the Halpern iteration until the fixed-point residual halves, then restarts the algorithm, initializing at the output of the previous run. 

 \subsection{Nonexpansive and Contractive Operators}

 First, we evaluate the performance on AdaGHAL on hard instances over the classes of nonexpansive and contractive operators. The results are shown in \Cref{fig:hard-instance}. For the hard instance, we use the construction from \cite{park2022exact}, where it was used in proving oracle complexity lower bounds for finding fixed points of $\gamma$-contractive operators with $\gamma \in (0, 1].$  The operator $\mT$ is in this case defined by $\mT = \mR_\gamma,$ where the $i^\text{th}$ coordinate of $\mR_\gamma:\sR^d \to \sR^d$, $\mR^{(i)}_\gamma$, is defined by
\begin{equation}\label{eq:hard-rotation-op}
    \mR_\gamma^{(i)}(\vx):= \begin{cases}
        s - \gamma \vx^{(d)}, &\text{ if } i = 1,\\
        \gamma \vx^{(i-1)}, &\text{ if } i \in\{2, \dots, d\}
    \end{cases},
\end{equation}
and where $s$ is a shift parameter that only affects the initial distance to the fixed point (set to $s =2$ for $\gamma < 1$ and to $s = 2/\sqrt{d}$ for $\gamma = 1$). In the plots, $d = 500.$

 \begin{figure}[t!]
    \centering
    \hspace*{\fill}
    \begin{subfigure}[b]{0.3\textwidth}
        \centering
        \includegraphics[width=\textwidth]{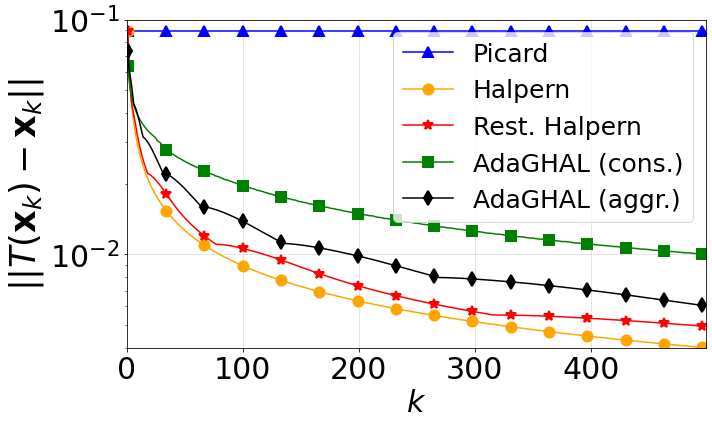}
        \caption{$\gamma = 1$}
    \end{subfigure} \hfill
    \begin{subfigure}[b]{0.3\textwidth}
        \centering
        \includegraphics[width=\textwidth]{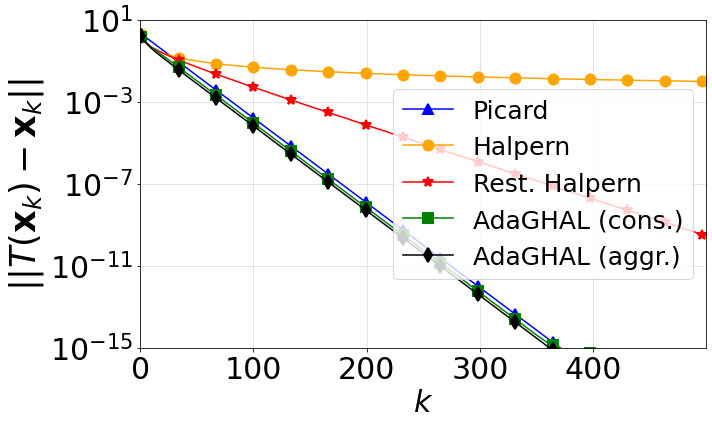}
        \caption{$\gamma = 10/11$}
    \end{subfigure} \hfill
    \begin{subfigure}[b]{0.3\textwidth}
        \centering
        \includegraphics[width=\textwidth]{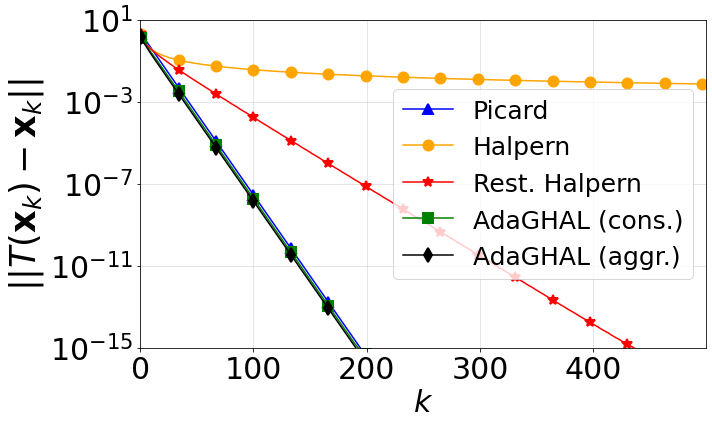}
        \caption{$\gamma = 5/6$}
    \end{subfigure}
    \hspace*{\fill}
    \caption{Algorithm comparison in terms of the fixed-point residual over the iterations, on a hard, lower bound instance.}
    \label{fig:hard-instance}
\end{figure}

To assess the local adaptivity properties of algorithms, we also evaluate them on a composition of two nonexpansive operators, of which one is contractive in a region around the fixed point. Namely, $\mT$ is defined by $\mT = \mR_1 \circ \mS$, where $\mS: \sR^d \to \sR^d$ is defined coordinatewise by $\mS^{(i)}(\vx) = c |x_i|$ for $|x_i| \leq 1/c$ and $\mS^{(i)}(\vx) = |x_i| - 1/c + 1$ otherwise, where $c \in (0, 1)$ is a constant.  
\Cref{fig:simple-adaptivity} illustrates the results for the case where the pieces $|x| \leq 1$ of $\mS^{(i)}$ have the slopes of magnitude $c \in \{0.7, 0.8, 0.9\}$, while for $|x| > 1$ the magnitude is $1.$ In other words, in this case the region close to the fixed point is contractive whereas the ``further away'' region is only nonexpansive. %

\begin{figure}[t!]
    \centering
    \hspace*{\fill}
    \begin{subfigure}[b]{0.32\textwidth}
        \centering
        \includegraphics[width=\textwidth]{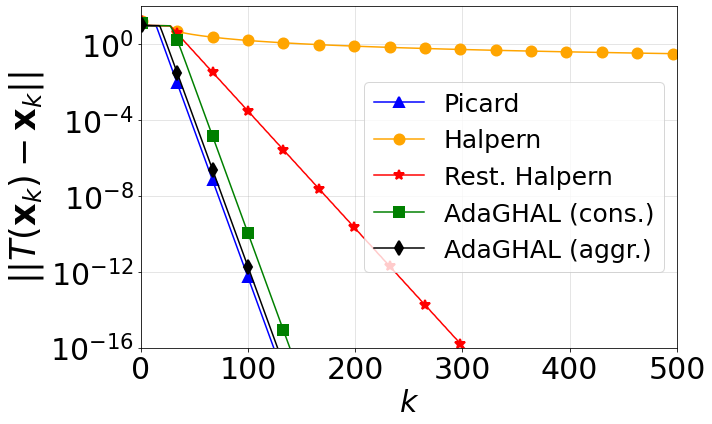}
        \caption{$c = 0.7$}
    \end{subfigure} \hfill
    \begin{subfigure}[b]{0.32\textwidth}
        \centering
        \includegraphics[width=\textwidth]{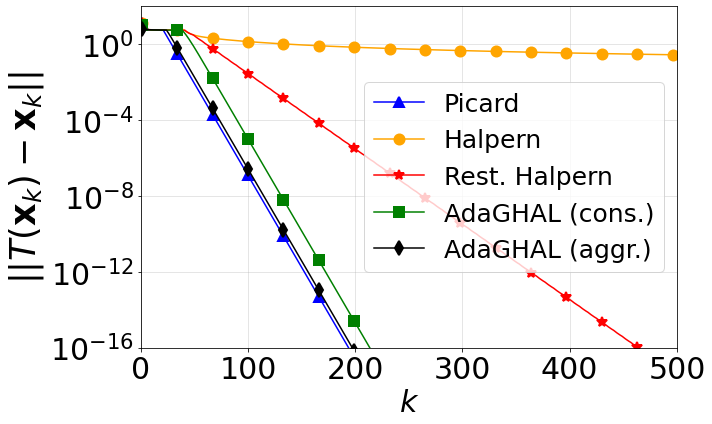}
        \caption{$c = 0.8$}
    \end{subfigure} \hfill
    \begin{subfigure}[b]{0.32\textwidth}
        \centering
        \includegraphics[width=\textwidth]{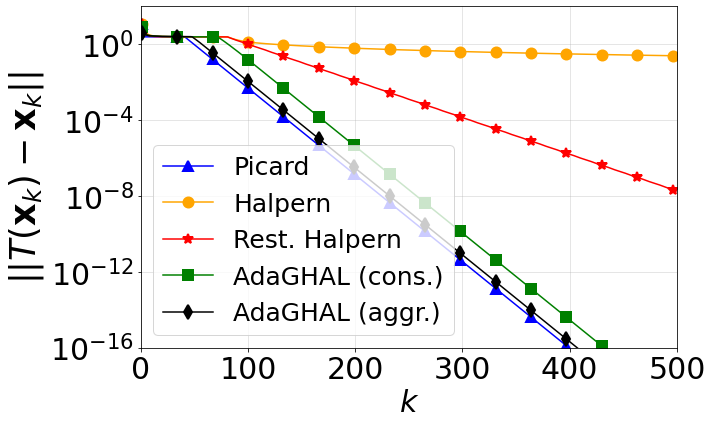}
        \caption{$c = 0.9$}
    \end{subfigure}
    \hspace*{\fill}
    \caption{Algorithm comparison in terms of the fixed-point residual over the iterations, on an instance that is nonexpansive away from the fixed point and contractive (with $\gamma = c$) in a region around the fixed point.}
    \label{fig:simple-adaptivity}
\end{figure}

\subsection{Expansive Operators}

To evaluate performance of AdaGHAL on expansive operators and compare it to baseline algorithms, we consider two examples. The first example is a simple operator that composes an expansive piecewise-affine operator with the rotation operator defined in \eqref{eq:hard-rotation-op}. In particular, in the first example, $\mT = \mS \circ \mR_1,$ where $\mS$ is defined by $\mS^{(i)}(\vx) =  \gamma x_i$ if $|x_i| \leq 1/2$ and $\mS^{(i)}(\vx) =  x_i + (1/2)(\gamma -1)$ otherwise. The results are plotted in \Cref{fig:piecewise-expansive}. 

\begin{figure}[t!]
    \centering
    \hspace*{\fill}
    \begin{subfigure}[b]{0.3\textwidth}
        \centering
        \includegraphics[width=\textwidth]{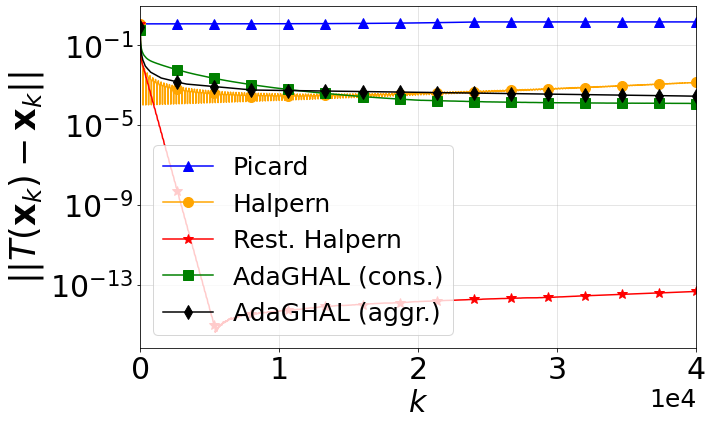}
        \caption{$\gamma =  1 + 10^{-4}$}
    \end{subfigure} \hfill
    \begin{subfigure}[b]{0.3\textwidth}
        \centering
        \includegraphics[width=\textwidth]{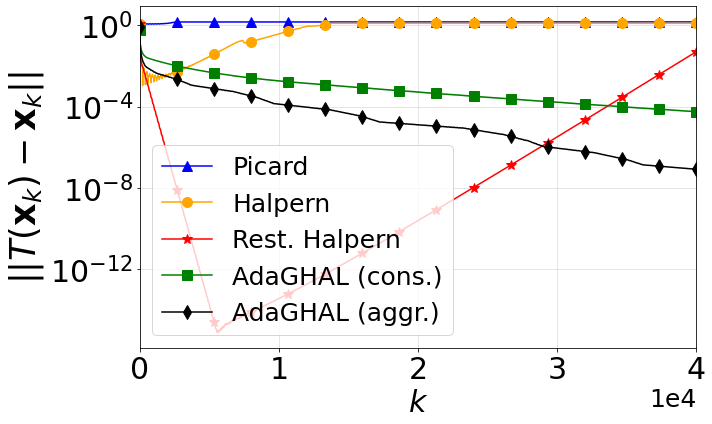}
        \caption{$\gamma = 1 + 10^{-3}$}
    \end{subfigure} \hfill
    \begin{subfigure}[b]{0.3\textwidth}
        \centering
        \includegraphics[width=\textwidth]{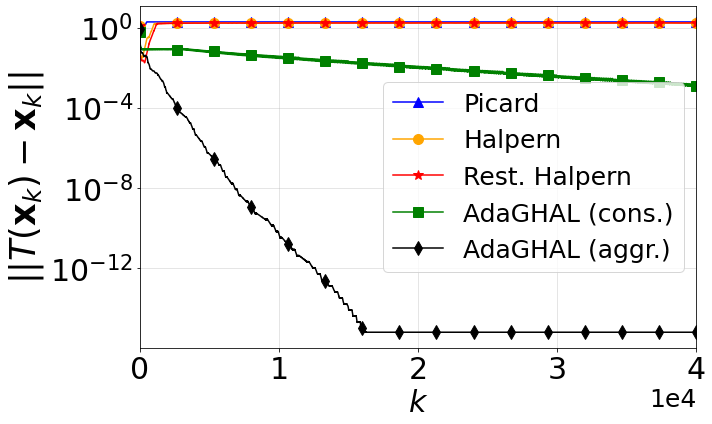}
        \caption{$\gamma = 1 + 10^{-2}$}
    \end{subfigure}
    \hspace*{\fill}
    \caption{Algorithm comparison in terms of the fixed-point residual over the iterations, on an instance that is nonexpansive away from the fixed point and expansive (with Lipschitz constant $\gamma$) in a region around the fixed point.}
    \label{fig:piecewise-expansive}
\end{figure}

The second example corresponds to a two-dimensional instance of an operator self-mapping the square $[-1, 1]^2$, defined as follows. Let $\mR_{2D}$ be the $90^{\circ}$ rotation operator; namely, $\mR_{2D}^{(1)}(\vx) = - x_2,$ $\mR_{2D}^{(2)}(\vx) = x_1.$ The operator $\mT$ is then defined by $\mT(\vx) = \proj_{[-1, 1]^2}\big(\mR_{2D}(\vx) + 0.05 \big(\exp\big(\frac{\alpha}{2}\mR_{2D}(\vx)\big) - \vone\big)\big).$ This operator has the unique fixed point at $(0, 0)$ and it is $\alpha$-gradually expansive with $\alpha = 0.4$ everywhere except in a small region around the fixed point (see \Cref{fig:gradually-expansive}(a)). The results are plotted in \Cref{fig:gradually-expansive}, with plots (c)--(g) showing the trajectories of individual algorithms.

\begin{figure}[t!]
    \centering
    \hspace*{\fill}
    \begin{subfigure}[b]{0.24\textwidth}
        \centering
        \includegraphics[width=\textwidth]{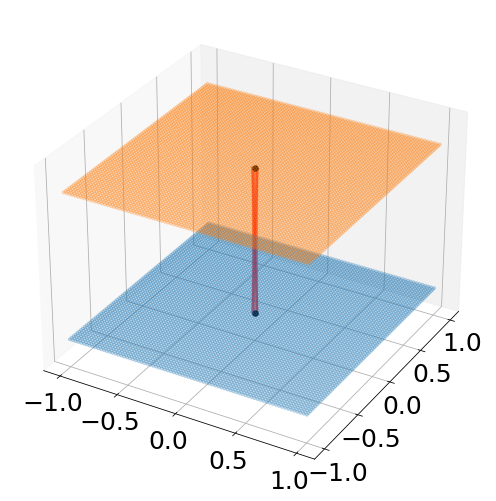}
        \caption{Failure pairs}
    \end{subfigure} \hfill
    \begin{subfigure}[b]{0.35\textwidth}
        \centering
        \includegraphics[width=\textwidth]{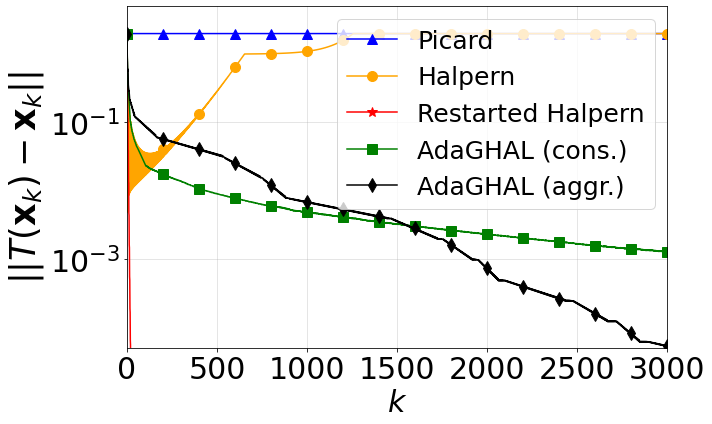}
        \caption{residual comparison}
    \end{subfigure} \hfill
    \begin{subfigure}[b]{0.24\textwidth}
        \centering
        \includegraphics[width=\textwidth]{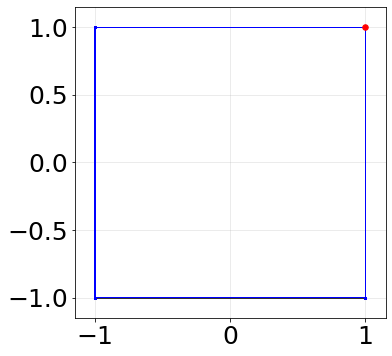}
        \caption{Picard}
    \end{subfigure}
    \hspace*{\fill}
    \hspace*{\fill}
    \begin{subfigure}[b]{0.24\textwidth}
        \centering
        \includegraphics[width=\textwidth]{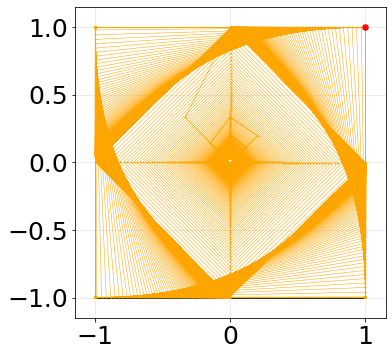}
        \caption{Halpern}
    \end{subfigure} \hfill
    \begin{subfigure}[b]{0.24\textwidth}
        \centering
        \includegraphics[width=\textwidth]{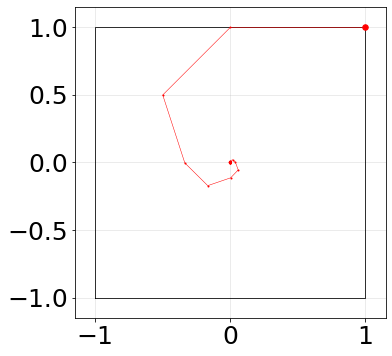}
        \caption{rest.\ Halpern}
    \end{subfigure} \hfill
    \begin{subfigure}[b]{0.24\textwidth}
        \centering
        \includegraphics[width=\textwidth]{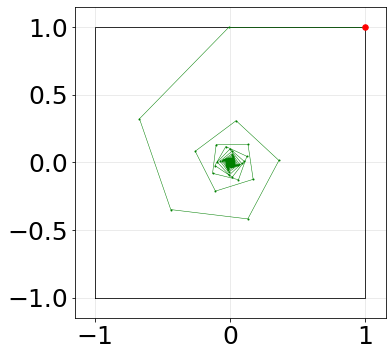}
        \caption{AdaGHAL (cons.)}
    \end{subfigure}\hfill
    \begin{subfigure}[b]{0.24\textwidth}
        \centering
        \includegraphics[width=\textwidth]{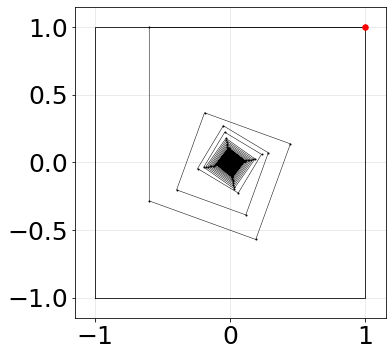}
        \caption{AdaGHAL (aggr.)}
    \end{subfigure}
    \hspace*{\fill}
    \caption{Algorithm comparison on a 2D instance that is 0.4-gradually expansive everywhere except in a small region around the fixed point $(0, 0)$: (a) red lines connect the pairs of points $\vx, \vy$ for which the gradual expansion condition fails, (b) comparison in terms of the fixed-point residual against the iteration count, (c)--(g) algorithm trajectories, initialized in the top-right corner, as indicated by the red dot.}
    \label{fig:gradually-expansive}
\end{figure}

\subsection{Discussion of Numerical Results} 

Because Picard iteration is oracle comp\-lexity-optimal for $\gamma < 1$, it is the best-performing algorithm on instances involving contractive operators. However, Picard generally does not converge when $\gamma \geq 1,$ as can be seen on examples in \Cref{fig:hard-instance}(a), \Cref{fig:piecewise-expansive}(a)--(c), and \Cref{fig:gradually-expansive}(b), (c). Similarly, Halpern iteration is oracle complexity-optimal for $\gamma = 1$, and this agrees with its numerical performance in \Cref{fig:hard-instance}(a). However, Halpern iteration can be slow when $\gamma < 1$ (\Cref{fig:hard-instance}(b), (c), \Cref{fig:simple-adaptivity}(a)--(c)) and it generally does not converge to any useful fixed-point residual on instances with expansive operators, as can be seen from \Cref{fig:piecewise-expansive}(b), (c), \Cref{fig:gradually-expansive}(b), (d). 

Restarted Halpern iteration, as expected, is competitive with standard Halpern iteration for $\gamma = 1$ and it converges linearly for contractive operators ($\gamma <1$). However, its rate of linear convergence is not competitive with Picard iteration, as can be observed from \Cref{fig:hard-instance}(b), (c) and \Cref{fig:simple-adaptivity}(a)--(c). On instances with expansive operators, restarted Halpern iteration can identify solutions with a small fixed point residual, but it is not guaranteed to converge; in fact, on examples in \Cref{fig:piecewise-expansive}, fast initial convergence is followed by increasing fixed-point residual, indicating divergence. On example in \Cref{fig:gradually-expansive}, restarted Halpern iteration converges surprisingly fast, within around 10 iterations. This fast convergence (on the gradually expansive example, and initially on the expansive example) can be explained by the large values of the step size $\lambda_k$ in initial iterations, following the restart. In that sense, restarted Halpern iteration is using a similar mechanism to AdaGHAL to correct for the operator expansion. However, the limitation we see in the diverging behavior in \Cref{fig:piecewise-expansive} comes from situations in which halving the residual is no longer possible but the step sizes $\lambda_k \propto 1/k$ keep decreasing, forcing the algorithm to diverge similar to standard Halpern iteration (which it reduces to in this scenario). Nevertheless, it would be interesting to investigate whether there are appropriate safeguard mechanisms that can prevent divergence. The main challenge is that, unlike AdaGHAL, restarted Halpern algorithm does not appear to have an obvious quantity such as iterate distance (as in AdaGHAL) or the fixed-point residual that is guaranteed to monotonically decrease.

Finally, both parameter settings of AdaGHAL gracefully adapt to $\gamma$, as predicted by the analysis: the algorithm is competitive with Picard iteration when $\gamma < 1$ and competitive with Halpern iteration for $\gamma = 1.$ Further, it converges to approximate solutions with small fixed-point residual even on examples of expansive operators in \Cref{fig:piecewise-expansive} and \Cref{fig:gradually-expansive}, often exhibiting better performance than predicted by the (worst-case) analysis. Interestingly, the aggressive step size parameter choice performs better than the conservative choice, nearly uniformly, barring for the initial iterations in \Cref{fig:gradually-expansive}(b). It would be interesting to investigate whether such  aggressive step size parameters can also be justified analytically.

\section{Further Discussion and Open Problems}

While the present work provides nontrivial classes of problems for which fixed points can be efficiently approximated, there are several avenues for future research that would deepen our understanding of this area and possibly resolve some of the open questions in game theory such as those discussed in \cite{daskalakis2022non}.

\subsection{Improving Lower Bounds (or Reducing the Gap between Upper and Lower Bounds)}

As discussed in \Cref{sec:fixed-mildly-expansive}, the  classical lower bound for solving fixed-point equations with $\gamma$-Lipschitz operators mapping $[0, D]^d$ ($d \geq 3$) to itself, due to \cite{hirsch1989exponential}, bounds below the worst case number of required queries to the operator to find a point $\vx \in [0, D]^d$ such that $\norm{\mT(\vx) - \vx} \leq \epsilon$, where $\epsilon \in (0, 1/10),$ by 
\begin{equation}\notag
    \Big(c\Big(\frac{1}{\epsilon} - 10\Big)(\gamma - 1)D\Big)^{d-2},
\end{equation}
where $c \geq 10^{-5}.$ The authors of \cite{hirsch1989exponential} conjectured that $c$ can in fact be made close to one. Obtaining such a  lower bound with a larger value of $c$ (ideally, $c = 1$) and removing the shift by 10 in the statement of the lower bound would effectively establish that our results get to the boundary of what is information-theoretically possible for polynomial-time algorithms. Note here that in such a statement of a lower bound, one would need to be careful with the interval for $\epsilon,$ as any point in the feasible set has fixed-point error bounded by $D.$ We also note here that none of the recent lower bounds such as \cite{attias2025fixed} implies a tighter range of constants for related settings (notably, obtaining tighter constants in the lower bounds was not their focus).

Conversely, any improvement to the obtained results on approximating $\norm{\mT(\vx) - \vx}$ that nontrivially reduces the gap between upper and lower bounds would be interesting.

\subsection{Convergence Under Local Error Bounds}

As discussed earlier, conditions on expansive operators $\mT$ that enforce an upper bound on some notion of a distance between an element $\vx$ and the set of fixed points of $\mT$ as a function of the fixed point error $\norm{\mT(\vx) - \vx}$ has been used in prior work \cite{natarajan1993condition,russell2018quantitative,lauster2021convergence,hermer2023rates} to ensure convergence to solutions with arbitrarily small fixed-point error is possible, often at a (locally) linear rate (though the ``constant'' of linear convergence can be quite---even exponentially in the dimension---small). Such results rely upon the expansion of $\mT$ being controlled by the parameter of the enforced local error bound. It appears possible (with some technical work) to generalize our results to settings where such a local error bound applies, possibly by slightly modifying the introduced algorithms. Some questions that we find interesting in this context are: Would the range of problems solvable by such algorithms to an arbitrarily small error be extended in some nontrivial way? Even without requiring convergence to an arbitrarily small error, can we improve the fixed-point error bound of the algorithm? How are these two notions/requirements intertwined?

\subsection{Gradual Expansion in Applications}

Finally, this work introduced the class of gradually expansive operators without providing concrete examples for which standard fixed-point iterations would fail. \Cref{sec:grad-expansive-examples} provides an example of an operator that is $\alpha$-gradually expansive (but not nonexpansive) and proves additional sufficient conditions for constructing such operators. It would be interesting to construct other nontrivial examples of operators that are $\alpha$-gradually expansive and/or prove additional structural properties for the associated class of fixed-point problems. 
In a different direction, proving gradual expansiveness for operators arising in common game theoretic settings whose complexity is not yet understood (or even verifying it numerically) could help address some of the recognized  open problems in game theory and machine learning, such as those outlined in \cite{daskalakis2022non}.

\section*{Acknowledgments}

This work was supported in part by an AFOSR Young Investigator Program award, under the contract number FA9550-24-1-0076, by the NSF CAREER Award CCF-2440563, and by the U.S.\ Office of Naval Research under contract number  N00014-22-1-2348. Any opinions, findings and conclusions or recommendations expressed in this material are those of the author(s) and do not necessarily reflect the views of the U.S. Department of Defense.

The author thanks professors Heinz Bauschke, Roberto Cominetti, Russell Luke, and Mihalis Yannakakis for useful comments and pointers to the relevant literature. 

\bibliographystyle{alpha}
\bibliography{references}
\end{document}